\documentclass[a4paper,12pt]{article}
\usepackage[utf8]{inputenc}
\usepackage[T1]{fontenc}
\usepackage[english]{babel}
\usepackage{mathtools,amssymb,amsthm}
\usepackage{hyperref}
\usepackage{enumitem}
\usepackage{esint}
\usepackage{lmodern}
\usepackage{geometry} 
\usepackage{amsbsy}
\usepackage{mathabx}
\usepackage{mathrsfs}
\usepackage{xcolor}
\usepackage{subcaption}
\usepackage{graphicx}
\usepackage{ulem}
\newcommand \D{{\mathbb{D}}}

\newcommand \Oo{\mathcal{O}}
\newcommand \ov{\overline}
\newcommand \A{\mathcal{A}}
\newcommand \N{\mathbb{N}}
\newcommand \Z{\mathbb{Z}}

\newcommand \Sa{\mathcal{S}}

\newcommand \C{\mathbb{C}}
\newcommand \R{\mathbb{R}}
\newcommand \HH{\mathcal{H}}

\newcommand \K{\mathcal{K}}
\newcommand \di{\mathrm{dim}}

\newcommand \Om{\Omega}
\newcommand \om{\omega}
\newcommand \eps{\epsilon}

\newcommand \intK{\overset{\circ}{K}}

\numberwithin{equation}{section}

\newtheorem{theorem}{Theorem} 
 
\newtheorem{lemma} [theorem]{Lemma} 
\newtheorem{remark} [theorem]{Remark} 
\newtheorem{proposition} [theorem]{Proposition} 
 
\theoremstyle{definition} 
\newtheorem{definition}[theorem] {Definition}

\DeclareMathOperator{\acosh}{arccosh}

\title{Computation of harmonic functions on higher genus surfaces}

\author{Micka\"{e}l Nahon\thanks{ Univ. Grenoble Alpes, CNRS, Grenoble INP, LJK, 38000 Grenoble, France.}, Édouard Oudet\thanks{ Univ. Grenoble Alpes, CNRS, Grenoble INP, LJK, 38000 Grenoble, France.}}

\date{}

\begin{document}

\maketitle%

\begin{abstract}
We extend a classical approximation result of harmonic functions in planar domains due to Bernstein and Walsch to the setting of harmonic functions in Riemann surfaces. This result gives an exact characterization of the rate at which a harmonic function in a subdomain of a compact Riemann surface may be approached by globally defined harmonic functions with prescribed poles. We illustrate the effectiveness and the impact of the method solving general boundary value Laplace problems in subdomains of the surface; we lay the groundwork for this numerical method in Riemann surfaces represented by a gluing of hyperbolic polygons. In particular, we give a general approximation procedure that computes this basis efficiently with arbitrary precision.
\end{abstract}

\renewcommand{\thefootnote}{\fnsymbol{footnote}} 
\footnotetext{\emph{Keywords}: Harmonic functions, Laplace equation, Higher genus, Abel-Jacobi, Bernstein-Walsch}     
\footnotetext{\emph{2020 Mathematic Subject Classification} 30F15, 31A25, 31C12, 35C10, 65N15, 65N80}     
\renewcommand{\thefootnote}{\arabic{footnote}} 
\tableofcontents


\section{Introduction}

Harmonic functions are a ubiquitous class of functions that appear in potential theory, fluid dynamics, heat conduction, and more generally in solutions of various optimization problems. We are concerned with the following question: given a compact  Riemannian  surface $K$ with boundary $\partial K$, and a function $g:\partial K\to\R$, how can we approximate precisely and efficiently the harmonic extension
\begin{equation}\label{eq_harmoniqueplan}
u: K\to\R\text{ such that }\begin{cases}\Delta u=0&\text{ in } K\\ u=g & \text{ in }\partial K\end{cases}\ ?
\end{equation}
A classical approach providing an accuracy beyond finite element methods is the method of particular solutions. This method has seen significant developments over the past decade, both in terms of its stability and generality (see for instance \cite{BT05, strohmaier2013algorithm}):

For $K$ a smooth compact subset of $\C$, we denote by $(D_j)_{j=0,\hdots,c}$ the connected component of $\widehat{\C}\setminus K$ (where $\widehat{\C}=\C\cup\{\infty\}$ is the Riemann sphere), with $\infty\in D_0$ by convention, and let $w_j\in D_j$ a choice of pole (with $w_0=\infty$ by convention). Runge's theorem states that any harmonic function on $K$ may be approached as $n\to +\infty$ by functions of the form
\begin{equation}\label{eq_logarithmicconjugation}
a+\Re\sum_{k=1}^{n}a_{0,k}z^k+\sum_{j=1}^{c}a_{j,0} \log|z-w_j|+\Re\sum_{j=1}^{c}\sum_{k=1}^{n}a_{j,k}(z-w_j)^{-k}
\end{equation}
for some real coefficients $a\in\R$, $(a_{j,0})_{1\leq j\leq c}$ and complex coefficients $(a_{j,k})_{0\leq j\leq c,1\leq k\leq n}$. By matching the boundary data (for instance through a least square method for a sampling of the boundary) this gives an efficient method of approximation (this is discussed in \cite{T18} for the approximation of Green functions). What makes this method particularly effective is that for general analytic boundary data one may expect a spectral speed of convergence. More precisely, there is a sequence of functions $(u_n)_{n\in\N}$ of the form \eqref{eq_logarithmicconjugation} that approaches the solution $u$ of \eqref{eq_harmoniqueplan} with
\[\limsup_{n\to\infty}\Vert u-u_n\Vert_{L^\infty(K)}^{1/n}=q\]
where $q(< 1)$ depends on the extension properties of the solution $u$.

In the case where $c=0$ (i.e. $\C\setminus K$ is connected and $u_n$ is a harmonic polynomial of degree at most $n$), the full characterization of $q$ is given by  a generalization of the Bernstein-Walsch theorem (see \cite[Ch. VII]{W35} or the more modern reference for potential theory in the plane \cite[Th 6.3.1]{R95}) to harmonic function: for a \textit{regular enough} set $K$, $q$ is the smallest value for which $u$ extends harmonically to \[\left\{z\in\C:\mathcal{G}_{K,\infty}(z)<\log(1/q)\right\}\]
where $\mathcal{G}_{K,\infty}$ is the potential function of the set $K$ with a logarithmic pole at infinity, as defined in equations \eqref{eq_green} (with $X=\C\cup\{\infty\}$, $w=\infty$ in this case). This result was proved in \cite{walsh1937maximal} when $\partial K$ is a finite disjoint union of Jordan curves, and in \cite[Prop. 3]{NguyenDjebbar} under weaker regularity hypothesis.

The Bernstein-Walsch theorem has been generalized in several directions. Generalizations to holomorphic functions in several complex variables may be found in \cite[§10]{S62}. In \cite[Th 1.2]{BL07}, it is proved that if $K\subset\R^d$ is a compact set such that $\R^d\setminus K$ is connected and verifies some weak regularity hypothesis (i.e. or every point of the boundary being accessible by analytic arcs), then the Bernstein-Walsch theorem for harmonic functions still holds. See also  \cite[Th. 2]{A93} for an extension to John domains, and \cite[Th. B]{Lemniscate} for multiply connected domains in the plane.

Coming back to harmonic functions in two dimensions, this approach has been adapted (theoretically and numerically) to domains with less regularity, for instance in the series of papers \cite{GT19_pnas,GT19} in domains with angular points, where an appropriate positioning of the poles near the corners give $\mathcal{O}(e^{-c\sqrt{n}})$ approximation error. This was also applied to two-dimensional Stokes equation in \cite{BT22}. See also the more recent work \cite{T24}, that explores the rate of convergence in smooth non-convex domains through a careful positioning of the poles.

This method of particular solution was generalized in \cite{KOO23} to the case where $K$ is a subset of a torus, where the authors constructed a similar basis of solutions as in equation \eqref{eq_logarithmicconjugation} using the Weierstrass elliptic functions, which may be computed efficiently with high accuracy; although the authors demonstrate numerically the same spectral speed of convergence as in subsets of the plane, there were no theoretical results to support this.

The goal of this paper is twofold:
\begin{itemize}
\item[1) ] We generalize this method of particular solutions to any compact surface written as a gluing of polygons: suppose that $K$ is a compact subset of some compact surface (without boundary) $X$, we provide a method to compute efficiently with arbitrary precision a basis of harmonic functions on $X$ (with some prescribed poles in $X\setminus K$) that approximates any harmonic function on $K$. Since this is already known when $X$ is the Riemann sphere or a torus, we only consider the case where $X$ is an oriented closed surface of genus $g\geq 2$ ; note that the non-oriented case reduces to the oriented case by taking the oriented double cover.

In terms of numerical efficiency, our goal is that for a fixed surface $X$, we may build once and for all the building blocks to compute directly a basis of harmonic function for any location of the poles: as we will see, our method requires the computation of $g$ holomorphic functions that act as the coordinates of the Abel-Jacobi map.
\item[2) ] We prove a complete analysis of the rate at which this basis approximates a harmonic function defined on a sufficiently regular compact subset $K$ of the surface $X$, depending on its maximal domain of extension outside $K$. 
\end{itemize}

\subsection{Main result}

The construction of the basis is as follows:
\begin{itemize}
\item[1) ] We construct the Abel-Jacobi map $X\to\C^{g}/(\Z^g+\tau\Z^g)$, which is a holomorphic immersion of the surface in the complex torus $\C^g/(\Z^g+\tau\Z^g)$ (where $\tau$ is a certain symmetric matrix with positive imaginary part). While there seem to be no explicit way of constructing these functions (at least through a procedure that lends itself to fast evaluation), we give an elementary least-square method that provides us with a polynomial approximation of the Abel-Jacobi map on each hexagon of the decomposition. This step is independent of the location of the poles, meaning that for a given surface we only need to compute this function once to obtain the basis associated to any pole. This part is heavily dependent on the choice of representation of $X$, whereas the functions we built below are obtained as composition with the Abel-Jacobi map, meaning it is independent of the representation of $X$ so long as we can evaluate the Abel-Jacobi map.
\item[2) ] We construct the Green function associated to any two poles $v,w$, meaning a harmonic function $\log|\widehat{\sigma}_{v,w}|:X\setminus\{v,w\}\to \R$ such that

\[\log|\widehat{\sigma}_{v,w}(z)|=\log|d_X(z,v)|+\Oo_{z\to v}(1),\ \log|\widehat{\sigma}_{v,w}(z)|=-\log |d_X(z,w)|+\Oo_{z\to w}(1).\]
Here $d_X$ is the geodesic distance induced by the metric.
\item[3) ]Finally, for a given generic pole $w$ (in the sense that it is not a Weierstrass point, as is explained in subsection \ref{subsec_intro_weiestrass}), we construct a sequence of functions
\[(\widehat{\wp}_{n}(z,w))_{1\leq n\leq g},\ (\widecheck{\wp}_{n}(z,w))_{g+1\leq n\leq  2g+1},\ (\tilde{\wp}_{n}(z,w))_{n\geq 2g+2}\]
that may respectively be found in the subsections \ref{subsec_phat}, \ref{subsec_pbar}, \ref{subsec_ptilde}, such that for fixed $w$, the $n$-th function of the sequence above is harmonic with respect to $z$ with a unique pole at $w$ of order $n$. More precisely, up to a multiplicative constant, they have the following asymptotic form near $w$:
\[\underbrace{\frac{1}{(z-w)^n}+\mathcal{O}_{z\to w}\left(\frac{1}{(z-w)^{n-1}}\right)}_{\text{meromorphic}}+\underbrace{\mathcal{O}_{z\to w}(1)}_{\text{harmonic}}.\]
When $X$ is the Riemann sphere  $\widehat{\C}$, this is the role played by
\[z\mapsto \frac{1}{(z-w)^{n}}\]
which is meromorphic. Observe however that there is an obstruction to this property already in genus $1$ when $X=\C/(\Z+\tau \Z)$ for some $\tau\in\C\setminus\R$. Indeed, a residue theorem on a fundamental domain (i.e. a parallelogram $0,1,1+\tau,\tau$) shows that there is no meromorphic function with a unique pole of order $1$, and more obstructions appear in higher genus, as is summarized in lemma \ref{lemma_wei}.
\item[4) ]The case where $w$ is a Weierstrass point - something which is only possible for a finite number of points of the surface, see subsection \ref{subsec_intro_weiestrass} - needs only a small tweaking in the construction of the functions to work: the basis is instead
\[(\widehat{\wp}_n)_{1\leq n\leq 2g-1},\ (\widecheck{\wp}_n)_{2g\leq n\leq 4g-1},\ (\tilde{\wp})_{n\geq 4g},\]
where the functions are constructed similarly. We refer to the subsection \ref{subsec_weiestrass} for the necessary modifications.
\end{itemize}

For any finite set $ \mathcal{P}\subset X$, $n\geq 1$, we denote by $\HH_n(X\setminus\mathcal{P})$ the set of harmonic function $h:X\setminus \mathcal{P}\to \R$ such that 
\begin{equation}\label{eq_defHH}
h(z)=\mathcal{O}_{z\to w}\left( d_X(z,w)^{-n}\right)
\end{equation}
for each $w\in \mathcal{P}$.

\begin{proposition}\label{mr_finitepoles}
Let $X$ be an oriented compact surface of genus $g\geq 2$. Let $w_1,\hdots,w_N$ be non-Weierstrass points of $X$, then $\HH_n(X\setminus\{w_1,\hdots,w_N\})$ admits the following basis functions: the constant function, the functions
\[\log\left|\widehat{\sigma}_{w_j,w_{j+1}}\right|,\ j=1,\hdots,N-1\]
and the real and imaginary parts of
\[ (\widehat{\wp}_k(\cdot,w_j))_{1\leq k\leq g\wedge n},\ (\widecheck{\wp}_k(\cdot,w_j))_{g+1\leq k\leq (2g+1)\wedge n},\ (\tilde{\wp}_k(\cdot,w_j))_{2g+2\leq k\leq n},\ j=1,\hdots,N.\]
\end{proposition}
In the case where some  points $(w_j)_{j=1,\hdots,N}$ are Weierstrass points, the result still holds with the modified basis.

To state the next theorem, we introduce the potential function of a compact set $K\varsubsetneq X$. Let $E\subset X$ be a subset of $X$, we say $E$ is \textit{polar} when $E$ has zero capacity i.e. there is a sequence of functions in $\mathcal{C}^{\infty}_c(X\setminus\ov{E},\R)$ that converges to the constant function $1$ in the strong $H^1(X,\R)$ topology. A property verified everywhere outside a polar set is said to hold \textit{nearly everywhere}. When a compact set $K$ is not polar, then for every $w\in X\setminus K$ there exists a unique potential function
$$\mathcal{G}_{K,w}:X\to \R_+\cup\{+\infty\}$$
such that for every $w\in X\setminus K$, $z\in X\setminus \{w\}\mapsto \mathcal{G}_{K,w}(z)$ is subharmonic (i.e. upper semicontinuous with nonnegative Laplacian in the distributional sense) and we have
\begin{equation}\label{eq_green}
\begin{split}
\Delta \mathcal{G}_{K,w}&=0\text{ in }X\setminus (K\cup\{w\})\\
\mathcal{G}_{K,w}(z)&=-\log d_X(z,w)+\mathcal{O}_{z\to w}(1)\\
\mathcal{G}_{K,w}(z)&\underset{z\to z_0}{\longrightarrow}0\text{ for nearly every }z_0\in K 
\end{split}
\end{equation}
When $\mathcal{P}$ is a finite subset of $X\setminus K$, we will define
$$\mathcal{G}_{K,\mathcal{P}}(z):=\sum_{w\in\mathcal{P}}\mathcal{G}_{K,w}(z).$$
See \cite[sec 4.4]{R95} for more detail on the existence and uniqueness of potential functions, which still applies in this setting. In particular when the boundary of $K$ is smooth, then $\mathcal{G}_{K,w}$ is continuous and identically zero on $K$. For less regular sets, we introduce the notion of \textit{(non-)thinness}: a set $E\subset X$ is said to be non-thin at $z_0\in X$ when for any subharmonic function $h$ defined in a neighbourhood $U$ of $z_0$, we have
$$h(z_0)\leq \limsup_{z\to z_0,\ z\in U\cap E\setminus \{z_0\}}h(z).$$
A compact set $K$ is non-thin at every point of $K$ under very weak geometrical conditions, for instance when no connected component of $K$ is reduced to a point (see \cite[Th 4.2.2]{R95}). 

If a compact set $K$ is non-thin at a point $z_0\in K$, and $w\in X\setminus K$, then $\mathcal{G}_{K,w}$ is continuous at $z_0$ with value $0$. This is an equivalence when $K$ is connected, and several equivalent characterizations of thinness may be found in \cite[Th 4.2.4, Th. 4.4.9]{R95}. 



Our main result \ref{mr_bernsteinWalsch} below deals with the approximation of harmonic function on compact subsets $K\subset X$ by harmonic functions with singularities in a finite set $\mathcal{P}\subset X\setminus K$. We define our hypothesis on the triplet $(X,K,\mathcal{P})$ as follows:
\begin{definition}\label{def_hypothesis}
Given a compact surface $X$, a compact subset $K\subset X$ that is neither empty or $X$, and a finite set $\mathcal{P}\subset X\setminus K$, we say $(X,K,\mathcal{P})$ verifies the hypothesis
\begin{itemize}[label=\textbullet]
\item $\mathcal{R}^{\mathrm{weak}}$ if $X\setminus K$ has a finite number of connected components, $\mathcal{P}$ intersects every connected component of $X\setminus K$, and every point of $K$ is non-thin in $K$.
\item $\mathcal{R}^{\mathrm{strong}}$ if $X\setminus K$ has a finite number of connected components, $\mathcal{P}$ intersects every connected component of $X\setminus K$, and every point of $K$ is non-thin in the \textbf{interior} of $K$.
\end{itemize}
\end{definition}

As the name suggests, the hypothesis $\mathcal{R}^{\mathrm{strong}}$ implies $\mathcal{R}^{\mathrm{weak}}$, and the opposite is not true, as may be seen from considering any connected compact set $K$ with empty interior (such that $X\setminus K$ has a finite number of connected components) that is not reduced to a point. 


\begin{theorem}\label{mr_bernsteinWalsch}
Let $X$ be a compact surface, let $K$ be a compact subset of $X$, $\mathcal{P}$ a finite subset of $X\setminus K$ such that $(X,K,\mathcal{P})$ verify the hypothesis $\mathcal{R}^{\mathrm{weak}}$. Let $h:K\to \R$, then
\begin{itemize}
\item[(a) ]If $h$ extends harmonically to $\left\{\mathcal{G}_{K,\mathcal{P}}<t\right\}$ for some $t>0$, then $$\limsup_{n\to\infty}\inf_{h_n\in\HH_n\left(X\setminus\mathcal{P}\right)}\Vert h-h_n\Vert_{L^\infty(K)}^{1/n}\leq e^{-t}.$$
\item[(b) ]If $(X,K,\mathcal{P})$ verifies the hypothesis $\mathcal{R}^{\mathrm{strong}}$ and $$\limsup_{n\to\infty}\inf_{h_n\in\HH_n\left(X\setminus\mathcal{P}\right)}\Vert h-h_n\Vert_{L^\infty(K)}^{1/n}\leq e^{-t}$$
for some $t>0$, then $h$ extends harmonically to $\left\{\mathcal{G}_{K,\mathcal{P}}<t\right\}$.
\end{itemize}
\end{theorem}
\begin{remark}\label{rem_rstrong}
The hypothesis $\mathcal{R}^{strong}$ holds under weak geometric hypothesis; for instance it is verified under any of the following properties:
\begin{itemize}[label=\textbullet]
\item Every point of $K$ is analytically accessible from $\intK$ i.e. for any $p\in K$ there exists some real analytic function $c:(-1,1)\to X$ such that $c(0)=p$, $c((0,1))\subset\intK$. This is reminiscent of Plesniak's theorem (see \cite[Cor. 3.2]{plesniak1984criterion}, or \cite[Th. 4.3]{BL07}) for harmonic function in any dimension.
\item $K$ is the closure of $\intK$ and $X\setminus \partial K$ has a finite number of connected components.
\end{itemize}
This is discussed at the end of subsection \ref{subsec_holotoharmo}. 
\end{remark}

This spectral speed of approximation is observed when numerically solving the harmonic extension problem (see figure \ref{fig_spec}).

It is sufficient to prove this result for oriented surfaces, and we will suppose in all that follows that $X$ is oriented. While most of the paper is focused on higher genus surfaces, we include the genus $g=1$ in the result and in this case the basis of harmonic function may be replaced with the (analogous) basis from \cite[Th 1.2]{KOO23}. We will use the notations we introduce for genus $g\geq 2$ and will point out the difference between genus $1$ and $2$ in the proof.

\subsection{Organisation of the paper}

In the second section we introduce several well-known concepts in compact Riemann surfaces: the notion of canonical basis of homology, the space of $1$-form and the Abel-Jacobi map, the function $\Theta$ and its application to the construction of quasi-periodic meromorphic functions. Finally, we remind several consequences of the Riemann-Roch theorem and some fact on Weierstrass points.

In the third section, we give a method to effectively construct the Abel-Jacobi map from polynomial approximation, and then the different elements of the basis of harmonic functions for fixed poles.

In the fourth section we prove the main results \ref{mr_finitepoles} and \ref{mr_bernsteinWalsch}, starting from an approximation result for holomorphic functions (see theorem \ref{th_bersnteinwalsch_holo}).

In the fifth section we detail the computational methods used to compute effectively the previously defined basis of harmonic functions with arbitrary precision, and we apply this in the fifth section to the harmonic extension problem in a surface.

In the last section, we give the technical details of the proof of convergence of the least-square method in the approximation of Abel-Jacobi map.

\section{Some facts and notations on compact Riemann surfaces of higher genus}

In this section, we fix the construction of the surface $X$ and remind several classical results on Riemann surfaces, particularly on the space of $1$-forms on a surface, on the Abel-Jacobi map and on theta functions that will be used later to define our basis of harmonic functions. In this section, $X$ is an oriented compact surface of genus $g\geq 2$.

\subsection{Gluing procedure and Fenchel-Nielsen coordinates}\label{sec_gluing}

We first fix a construction of a compact oriented surface $X$ of genus $g\geq 2$. Consider the disk model $\D=\{z\in\C:|z|<1\}$ with the hyperbolic metric $\frac{4}{(1-|z|^2)^2}dz\ov{dz}$. We remind that the matrix group 
\[PSU(1,1)=\left\{\begin{pmatrix}\alpha & \beta \\ \ov{\beta} & \ov{\alpha}\end{pmatrix}\in M_2(\C):|\alpha|^2-|\beta|^2=1\right\}/\{\pm I_2\},\]
describes all direct hyperbolic isometries of $\D$ through the action 
\[\begin{pmatrix}\alpha & \beta \\ \ov{\beta} & \ov{\alpha}\end{pmatrix}\cdot z=\frac{\alpha z+\beta}{\ov{\beta}z+\ov{\alpha}}.\]  
Consider a finite number of hyperbolic polygons in the unit disk $\D=\{z\in\C:|z|<1\}$  (meaning closed simply connected sets such that their boundaries are a union of hyperbolic geodesic i.e. circular arcs and lines that meet $\partial\D$ orthogonally) denoted 
\[(H_p)_{p=1,\dots,m}.\]
We call $\gamma_{p,1},\hdots,\gamma_{p,c(p)}$ the successive sides of $H_p$. We then consider a gluing rule for these polygons: each side $\gamma_{p,i}$ is glued to a unique side $\gamma_{q,j}$ (we write $(p,i)\to (q,j)$, and this is a symmetric relation) in the sense that there exists a disk automorphism  $g_{p,i}\in PSU(1,1)$ such that $g_{p,i}(\gamma_{p,i})=\gamma_{q,j}$ and the sets $g_{p,i}(H_p)$ and $H_q$ have disjoint interior. This implies in particular that glued sides have the same hyperbolic length, and a necessary and sufficient condition for this gluing rule to define a smooth surface is that around each vertex, the sum of the angular opening of all the vertices it is identified with is $2\pi$: we will suppose this condition is verified.

$X$ is then defined as the formal union of the polygons $(H_p)_{p=1,\hdots,m}$ with the identifications induced by the gluing rule. Moreover, we may see the embeddings $H_p\to X$ ($p=1,2,\hdots m$) as a system of charts that define the surface $X$, since these embeddings may be extended to a small neighbourhood of the polygons (so that the range of the charts overlap).

A particular case of this decomposition is when each $H_p$ is a right-angled hexagon. We remind that any three lengths $(l_1,l_2,l_3)\in\R_{>0}$ define (up to isometry) a unique right-angled hexagon where the first, third and fifth sides have respective hyperbolic lengths $l_1/2,l_2/2,l_3/2$.

While this may seem restrictive, any compact oriented surface of genus $g\geq 2$ is conformally equivalent to a hyperbolic surface obtained as the gluing of $4g-4$ right-angled hexagons with the following gluing rules: each hexagon associated to $(l_1,l_2,l_3)$ is paired with a mirrored hexagon (of lengths $(l_1,l_3,l_2)$) by gluing the alternate side (and obtaining a pair of pants i.e. a hyperbolic compact surface with geodesic boundary that is homeomorphic to a sphere with three disks removed).

The $2g-2$ pants are then associated by gluing boundary geodesics of same length, with a possible rotation in the gluing of the geodesics: this leaves a total of $6g-6$ independent parameters with $3g-3$ lengths (denoted $l_i>0$) and $3g-3$ angles (denoted $t_i\in\R/\Z$), that constitute the Fenchel-Nielsen coordinates of the surface. This construction is illustrated in the genus $2$ case in figure \ref{fig_gluing}.

\begin{figure}[!htb]
    \centering
    \begin{tabular}[t]{ccc}
        \begin{tabular}{c}
            \hfill
            \begin{subfigure}[t]{0.3\textwidth}
                \centering
                \includegraphics[width=\linewidth]{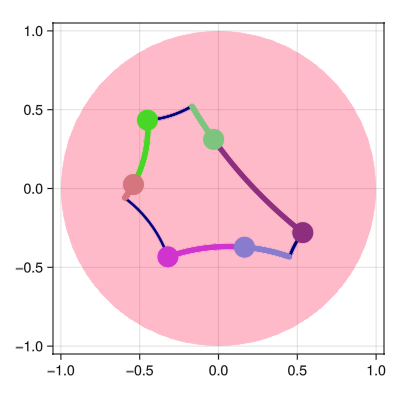}
            \end{subfigure}\\
            \hfill
            \begin{subfigure}[t]{0.3\textwidth}
                \centering
                \includegraphics[width=\linewidth]{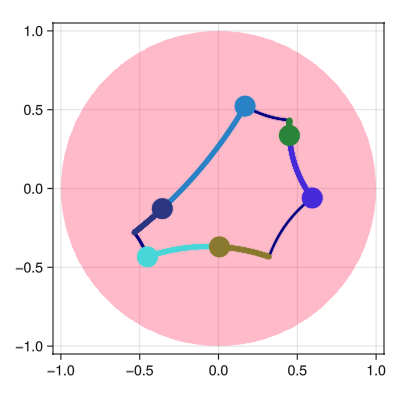}
            \end{subfigure}
        \end{tabular}
        &
        \begin{tabular}{c}
            \hspace{-1cm}
            \begin{subfigure}[t]{0.3\textwidth}
                \centering
                \includegraphics[width=\linewidth]{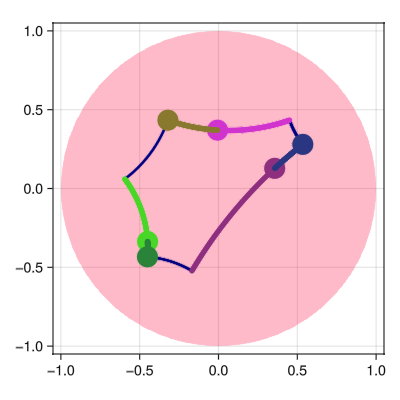}
            \end{subfigure}\\
            \hspace{-1cm}
            \begin{subfigure}[t]{0.3\textwidth}
                \centering
                \includegraphics[width=\linewidth]{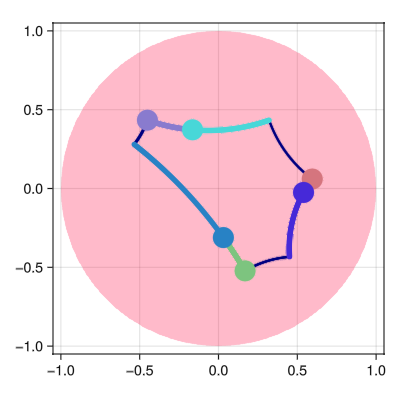}
            \end{subfigure}
        \end{tabular}
    \end{tabular}
    \caption{Gluing conditions associated to Fenchel-Nielsen coordinates: every geodesic is divided in four pieces. In this example
     $(l_1, t_1; l_2, t_2; l_3, t_3) = (\frac32, \frac{1}{10}; 2, \frac{2}{10};\frac52, \frac18)$.}
     \label{fig_gluing}
\end{figure}

\subsection{Canonical basis of homology}

Let $\Gamma$ be the subgroup of $PSU(1,1)$ that is generated from the compositions
\[g_{p_n,i_n}\circ g_{p_{n-1},i_{n-1}}\circ\hdots\circ g_{p_1,i_1}\]
where
\begin{itemize}[label=\textbullet]
\item $p_1=1$.
\item For all $k=1,\hdots,n-1$ there exists $j\in\{1,\hdots,c(p_{k+1})\}$ such that $(p_k,i_k)\to(p_{k+1},j)$.
\item There exists $j\in\{1,\hdots,c(1)\}$ such that $(p_n,i_n)\to (1,j)$.
\end{itemize}

$X$ may then be identified with the quotient $\D/\Gamma$ through the mapping $z\in H_1\mapsto z\text{ mod }\Gamma$, which extends uniquely to a conformal bijection $X\to\D/\Gamma$. This induces a universal covering map $\imath:\D\to X$.

We fix in all that follows a canonical basis of homology of $X$, meaning a set of $2g$ closed curves $a_i:\R/\Z\to X$ and $b_i:\R/\Z\to X$ for $i=1,2,\hdots,g$, such that up to homotopy $(a_i,a_j),(b_i,b_j),(a_i,b_j)$ do not intersect when $i\neq j$, and $(a_i,b_i)$ intersect once with positive orientation.

For each $a_i$ (resp $b_i$), we let $\tilde{a_i}:\R\to\D$ (resp $\tilde{b_i}:\R\to\D$) that maps to $a_i$ (resp $b_i$) through the covering map $\iota:\D\to X$, and there exists a unique element $A_i\in\Gamma$ (resp $B_i\in\Gamma$) such that
\[A_i(\tilde{a_i}(0))=\tilde{a_i}(1),\ B_i(\tilde{b_i}(0))=\tilde{b_i}(1).\]
Up to reordering the curves appropriately then $A_1,B_1,A_2,B_2,\hdots,A_g,B_g$ is a generating family of $\Gamma$ such that
\[A_1B_1A_1^{-1}B_1^{-1}A_2B_2A_2^{-1}B_2^{-1}\hdots A_gB_gA_{g}^{-1}B_{g}^{-1}=I_2,\]
and this relation is sufficient to give a presentation of $\Gamma$. We refer to \cite[I.2.5]{FKH92} for the construction of a canonical $4g$-gon from a general polygonal gluing, which gives this presentation.



\subsection{$1$-forms and Abel-Jacobi maps}\label{sec_abel}
For a subset $E$ of $\C$ we denote by $\Oo(E)$ (resp $M\Oo(E)$) the set of holomorphic (resp meromorphic) functions defined on an open neighbourhood of $E$.

We denote by $\Om^k(X)$ (resp $\Om^k(\D/\Gamma)$) the space of holomorphic $k$-forms on $X$ (resp $\D/\Gamma$), and by $M\Om^k(X)$ (resp $M\Om^k(\D/\Gamma)$) the space of meromorphic $k$-forms on $X$ (resp $\D/\Gamma$).

To be more precise, these are defined as follows: $\Om^k(X)$ is the subset of holomorphic functions 
\[f=(f_{|1},\hdots,f_{|m})\in\Oo(H_1)\times \hdots\times\Oo(H_m)\]
such that for every gluing $(p,i)\to (q,j)$ we have
\[f_{|p}(z)=(g_{p,i}'(z))^kf_{|q}\circ g_{p,i}(z)\text{ for }z\in \gamma_{p,i}.\]
Similarly, $\Om^k(\D/\Gamma)$ is the subset of holomorphic functions $f\in\Oo(\D)$ such that 
\begin{equation}\label{eq_period}
\forall z\in\D,\ \forall \gamma\in\Gamma,\ f(z)=\gamma'(z)^kf(\gamma(z)).
\end{equation}
With our two constructions, the bijection between the spaces $\Om^k(X)$ and $\Om^k(\D/\Gamma)$ is as follows: for any $(f_{|1},\hdots,f_{|m})\in\Om^k(X)$, each $f_{|p}$ extends holomorphically to some function $\tilde{f_{|p}}$ defined on the disk $\D$, and the map
\[\begin{cases}
\Om^k(X)\to \Om^k(\D/\Gamma)\\
(f_{|1},\hdots,f_{|m})\mapsto \tilde{f_{|1}}
\end{cases}\]
is a bijection.
We remind that $\Om^1(\D/\Gamma)$ is a finite dimensional space of dimension $g$, and that we may fix a (uniquely defined) basis $(\om_1,\om_2,\hdots,\om_g)$ of $\Om^1(\D/\Gamma)$ such that for any $j,k\in\{1,\hdots,g\}$:
\begin{equation}\label{eq_canonical_1form}
\int_{0}^{A_j(0)}\om_k=\delta_{j,k}=\begin{cases}1\text{ if }j=k\\0\text{ if }j\neq k\\ \end{cases}.
\end{equation}
In this way we may define the period matrix $\tau\in \text{Sym}_g(\R)$:
\[\tau_{j,k}=\int_{0}^{B_j(0)}\om_k.\]
$\tau$ is a symmetric complex-valued matrix, and $\Im(\tau)$ is positive definite. 
We now define the \textbf{Abel-Jacobi} coordinate maps $u_j:\D\to\C$ by
\[u_j(z)=\int_{0}^{z}\om_j.\]
$u_j$ is well-defined on $\D$ but is not $\Gamma$-periodic: instead it verifies the partial periodicity relations
\[u_j(A_k(z))=u_j(z)+\delta_{j,k},\ u_j(B_k(z))=u_j(z)+\tau_{j,k}.\]
The Abel-Jacobi map based at the point $0\in\D_1$ is then
\[u:\begin{cases}\D\to\C^g\\ z\mapsto (u_1(z),u_2(z),\hdots,u_g(z))\end{cases}\]
It verifies the periodicity relations
\begin{equation}\label{aj_period}
u(A_j(z))=u_j(z)+e_j,\ u(B_j(z))=u(z)+\tau e_j
\end{equation}
for all $z$. In particular, if $z'=\gamma(z)$ for some $\gamma\in\Gamma$, then
\[u(z')-u(z)\in \Z^g+\tau \Z^g,\]
meaning that $u$ defines a map from the surface $X$ to the $2g$-dimensional torus $\C^{2g}/(\Z^g+\tau \Z^g)$ denoted $\ov{u}$. We also introduce the notation
\[u^k:\begin{cases}\D^k\to \C\\ (z_1,\hdots,z_k)\mapsto\sum_{j=1}^{k}u(z_j)\end{cases},\]
and $\ov{u}^k$ the quotient map from $X$ to $\C^g/(\Z^g+\tau\Z^g)$. We denote $\mathcal{C}\subset \C^g$ the critical values of $u^{g}$: it is a set with (complex) dimension at most $g-2$, and it is characterized as the set of points $\sum_{j=1}^{g}u(z_j)$ where
\[\det\left(\left(\om_j(z_k)\right)_{1\leq j,k\leq g}\right)=0.\]


\subsection{Theta functions and quasi-periodic functions with prescribed zeroes}

Let us now define the function
\[\Theta(Z):=\sum_{n\in\Z^g}e^{2\pi i \left(\frac{1}{2}n\cdot \tau n+n\cdot Z\right)},\ \forall Z\in\C^g.\]
In many references $\Theta$ is given by the same formula as a function of $Z$ and $\tau$, but since the matrix $\tau$ is fixed here we leave it out of the arguments. The function $\Theta$ verifies the following periodicity relations, for any $Z\in\C^g,\ n\in\Z^g$:
\begin{equation}\label{eq_period_theta}
\Theta(Z+n)=\Theta(Z),\ \Theta(Z+\tau n)=\Theta(Z)e^{-i\pi\left(n\cdot \tau n+2n\cdot Z\right)},\ \Theta(-Z)=\Theta(Z).
\end{equation}

Let us define the Riemann constant $\K\in\C^g$ by its coordinates
\begin{equation}\label{eq_defRiemann}
\K_j=\frac{1}{2}\tau_{jj}-\sum_{k=1}^{g}\int_{0}^{A_k(0)}u_j(z)u_k'(z)dz.
\end{equation}

\begin{theorem}
$\{Z\in\C^g:\Theta(Z)=0\}=\K+u^{g-1}(\D^{g-1})$. Moreover, for any $\xi\in\C^g\setminus (\K+\mathcal{C})$ written as
\[\xi=\K+\sum_{j=1}^{g}u(w_j),\]
the function
\[z\mapsto \Theta(u(z)-\xi)\]
vanishes with order $1$ at every element of $\{\gamma(w_j),\gamma\in\Gamma,j=1,\hdots,g\}$, and is non-zero elsewhere.
\end{theorem}
This is contained in \cite[VI.3.1,VI.3.2]{FKH92}. Note that the condition on $\xi$ imposes that the $(w_j)s$ must be distinct (modulo $\Gamma$). The periodicity conditions verified by this function are, for any $j$,
\begin{equation}\label{eq_periodicity}
\Theta(u(A_j(z))-\xi)=\Theta(u(z)-\xi),\ \Theta(u(B_j(z))-\xi)=\Theta(u(z)-\xi)e^{2\pi i\left(\xi_j-u_j(z)-\frac{1}{2}\tau_{j,j}\right)}.
\end{equation}

\subsection{Weierstrass points}\label{subsec_intro_weiestrass}

Let $D$ be a divisor on $X$, meaning an element of $X^{\Z}$ represented as formal sum $D=\sum_{p\in X}n_p [p]$ where $n_p\in\Z$ is zero for all but a finite number of points $p\in X$. We denote by

\[\ell(D)=\di\left\{f\in M\Om^0(X):f=0\text{ or }(f)+D\geq 0\right\},\]
where $(f)=\sum_{p\in X}\text{ord}(f,p)[p]$ is the divisor associated to $f$. Let $K$ be the divisor of some holomorphic $1$-form, we remind that by the Riemann-Roch theorem:
\[\ell(D)-\ell(K-D)=\deg(D)-g+1,\]
where $\deg\left(\sum_{p\in X}n_p[p]\right)=\sum_{p\in X}n_p$. We summarize the results we will use as follows:
\begin{proposition}
Let $p\in X$, then the sequence $n\in \N\mapsto \ell(n[p])$ verifies
\[\ell(0[p])=1,\ \ell((n+1)[p])-\ell(n[p])\in\{0,1\},\ \ell(n[p])=n-g+1\text{ for any }n\geq 2g-1.\]
Moreover, there is a finite set of points $W\subset X$ called \textbf{Weierstrass points} such that for any $p\in X\setminus W$ we have
\[\ell(n[p])=1+(n-g)_+.\]
\end{proposition}
$p$ is a Weierstrass point if and only if
\[\det\left(\left(\om_j^{(k-1)}(p)\right)_{1\leq j,k\leq g}\right)= 0,\]
where this quantity is defined up to a (non-zero) multiplicative factor in $\D/\Gamma$.

By \cite[Cor III.5.11]{FKH92}, the cardinal of $W$ is in the interval $[2g+2,g^3-g]$, in particular when $g=2$ then $W$ has exactly $6$ element.

Applying the Riemann-Roch theorem to $D=n[p]$, this may be reformulated as follows: $\ell(K-n[p])$ is the dimension of the space of $1$-forms $\om\in\Om^1(X)$ that admits a zero of order at least $n$ at $p$. So for any $p\in X$,
\[\ell(K-(2g-1)[p])=0,\]
meaning that any $1$-form that has a zero of order larger than $2g-1$ is zero everywhere,
and if $p\in X\setminus W$ then
\[\ell(K-g[p])=0,\]
meaning that any $1$-form that has a zero of order at least $g$ at a non-Weierstrass point is zero everywhere. As a consequence:

\begin{lemma}\label{lemma_wei}
Let $z\in X$ (in some polygon $H_p$), the linear map 
$$\om\in \Om^1(X)\mapsto \left(\om_{|p}(z),\om_{|p}'(z),\hdots,\om_{|p}^{(2g-2)}(z)\right)\in\C^{2g-1}$$
is injective, and if $z\in X\setminus W$ then the map 
$$\om\in\Om^1(X)\mapsto \left(\om_{|p}(z),\om_{|p}'(z),\hdots,\om_{|p}^{(g-1)}(z)\right)$$
is bijective.
\end{lemma}


\section{Construction of a basis} \label{sec_basis} 

We split this construction in several steps: first there is the ``pre-treatment'' that is  independent of the location of the pole, where we compute a polynomial that approximates with high accuracy the Abel-Jacobi coordinate maps $u$. This allows us to compute the Riemann constant $\K$ defined in equation \eqref{eq_defRiemann} as well.

Next we build, for two poles $v,w$, the harmonic function with opposing logarithmic poles at $v,w$. Then for a fixed pole $w$  we build the complex harmonic function with a single pole of order $n$ at $w$: we start by constructing the pole of order $1$ using the composition of the logarithmic derivative of the $\Theta$ function with a suitable translation of the Abel-Jacobi map. We then compute formally the successive differential of this function along the location $w$ of the pole.

Once we've built sufficiently many meromorphic function with a pole at $w$, then the rest of the basis may be built by exponentiation of these functions: depending on whether the pole $w$ is a Weierstrass point or not, we need to build by successive differentiation either $2g+1$ or $4g-1$ functions before all the rest may be built by exponentiation.

We will explain in detail the case where the pole $w$ is \textbf{not} a Weierstrass point. We then explain afterward how to adapt it to this case.

\subsection{Computation of $1$-forms}\label{subsec_aj}

This subsection relies heavily on the representation of $X$ as a gluing of $m$ hyperbolic polygons $(H_p)_{p=1,\hdots,m}$, and we use the notations of section \ref{sec_gluing}. We will suppose additionally that the first polygon $H_1$ contains the center of the disk $0$, and that $0$ is not a Weierstrass point. 

For any $\om\in\Om^1(X)$, let us define the approximation space

\[\A^{N,\om}=\left\{P=(P_{|1},\hdots,P_{|m})\in \C_{N}[z]^m:(P_{|1}(0),P_{|1}'(0),\hdots,P_{|1}^{(g-1)}(0))=(\om_{|1}(0),\om_{|1}'(0),\hdots,\om_{|1}^{(g-1)}(0))\right\}.\]

Let $S(N)=a N$ for some $a\in \N_{\geq 1}$ that will be taken large enough, and for any side $\gamma_{p,i}$ we denote by $\Sa_{p,i}$ a sampling of $\gamma_{p,i}$ by $S(N)$ regularly spaced points. 

For any $P=(P_{|1},\hdots,P_{|m})\in\C_N[z]^{m}$, we let
\[E^N(P)=\sum_{(p,i)\to (q,j)}\frac{1}{S(N)}\sum_{z\in\Sa_{p,i}}\left|P_{|p}(z)-g_{p,i}'(z)P_{|q}(g_{p,i}(z))\right|^2.\]
Finally, for any compact set $K\subset\C$ we let $\mathcal{G}_{ K,w}$ be its potential function defined as in equation \eqref{eq_green} (where the ambient manifold is $\widehat{\C}$ in this case).

The main result of this section is the following:
\begin{theorem}\label{the_conv}
Let $X$ be a surface defined as above, such that $0\in H_1$ is not a Weierstrass point. Let $\rho\in ]0,1[$ be such that 
\[\rho>\exp\left(-\min_{p=1,\hdots,m}\inf_{z\in\partial \D_1}\mathcal{G}_{ H_p,\infty}(z)\right),\]
Then there exists a large enough integer $a>0$ and a constant $C>0$ such that for any $\om\in\Om^1(X)$, any $N\in\N_{\geq g}$, we denote by $Q^{N,\om}$ a minimizer of $E^N$ in $\A^{N,\om}$ and we have: 
\[\sum_{p=1}^{m}\Vert Q_{|p}^{N,\om}-\om_{|p}\Vert_{L^\infty(H_p)}\leq C\rho^N\sum_{j=0}^{g-1}|\om_{|1}^{(j)}(0)|.\]
\end{theorem}

We prove in fact a more precise conclusion: there is a constant $C_a>0$ (depending on the geometry of the hexagons $(H_p)_{p=1,\hdots,m}$, which is supposed to be fixed, and on $a$) such that for any $Q\in \A^{N,\om}$:
\[\sum_{p=1}^{m}\Vert Q_{|p}-\om_{|p}\Vert_{L^\infty(H_p)}^2\leq C_a N^4e^{\frac{C}{\sqrt{a}}N}E^N(Q).\]
Then the result follows from the fact that $\inf_{\A^{N,\om}}\sqrt{E^N}=\Oo\left(\rho^{N}\sum_{j=0}^{g-1}|\om_{|1}^{(j)}(0)|\right)$ for any such $\rho$, and taking $a$ large enough: nowhere do we use the fact that $Q^{N,\om}$ is a minimizer.

For now, we relegate the proof of this theorem to the last section. Let us give some remarks on the optimality and the limitations of this result.

\begin{itemize}
\item[1) ]We may give a more explicit upper bound of $\rho$ as follows: any $\rho$ such that the open disk $\D_{\rho}$ contains all the closed polygons $(H_p)_{p}$ verifies the hypothesis of the theorem. Indeed, by maximum principle we would have $\mathcal{G}_{ H_p,\infty}(z)\geq \log(z/\rho)_+$, so $\inf_{z\in\partial \D}\mathcal{G}_{ H_p}\geq \log(1/\rho)$ (and the inequality would be strict by strong maximum principle).
\item[2) ]The speed of convergence given here is optimal, in the sense that any $\rho$ lower than the bound given in the theorem would lead to a contradiction. This is discussed at the end of the section \ref{sec_leastsquareconv}.
\item[3) ]Each $Q^{N,\om}_{|p}$ converges beyond the polygon $H_p$: in fact we have 
\[\Vert Q^{N,\om}_{|p}-\om\Vert_{L^\infty(H_p\sqcup\{G_{\widehat{\C}\setminus H_p,\infty}<t\})}=\mathcal{O}_{N\to\infty}\left((e^{t}\rho)^N\right)\]
for any $\rho$ as in the statement and any $t$. However, this does not mean that $Q^{N,\om}_{|p}$ converges to $\om_{|p}$ in all of the unit disk $\D_1$. 
\item[4) ]Our proof of the theorem does not give an explicit constant $C$, because one of the bounds (see lemma \ref{lem_est_riemannroch}) includes an argument by contradiction. It is unclear whether this could be made explicit easily.
\item[5) ]$0\in H_1$ may happen to be a Weierstrass point. This happens for instance when $X$ is obtained by gluing the opposite sides of a single regular octogon centered at $0$ with angles $\frac{\pi}{4}$ (this is the Bolza surface): in this case the Weierstrass points are exactly the origin, the vertices and the middle of the edges (which after identifications correspond to a total of $6$ different points). In particular there is a basis of $1$-forms $\alpha,\beta\in \Om^1(X)$ such that
\[\alpha(z)=1+\Oo_{z\to 0}(z),\ \beta(z)=z^2+\Oo_{z\to 0}(z^3).\]
In practice this may be detected by the fact that for some $l=1,\hdots,g$, the quantity
\[\inf\left\{E^N(P),\ P\in \C_N[z]^m:\forall j=1,\hdots,g,\ P_{|1}^{(j-1)}(0)=\delta_{j,l}\right\}\]
does not converge to $0$. To solve this issue we proceed by perturbation: we choose some random point $w_{rand}\in H_1$ and impose successive derivative conditions at $w_{rand}$ instead of $0$. We may also change $H_1$ into $\varphi(H_1)$ for some small perturbation of the identity $\varphi\in PSU(1,1)$, which amounts to changing $\gamma_{1,i}$ into $\varphi(\gamma_{1,i})$ and composing the transition functions $g_{p,i}$ accordingly.
\end{itemize}

Overall this least square method allows us to compute a good approximation of the canonical basis of $1$-forms (verifying \eqref{eq_canonical_1form}). Moreover, polynomial expression can be explicitly integrated, so this also give us an approximate period matrix $\tau$, Abel-Jacobi map $u$ and Riemann constant $\K$. The details of these computations are explained in sections \ref{sec_comp} and \ref{sec_comp_period}.

In the next subsections we give a construction of the harmonic basis with explicit formula given by compositions of the Abel-Jacobi map $u$ and $\Theta$ functions (and its derivatives), the period matrix $\tau$ and the Riemann constant $\K$: in practice these will be computed using our approximations.

\subsection{Logarithmic singularities}\label{subsec_log}
In this subsection and all the following we work on the surface $\D/\Gamma$, instead of the (equivalent) surface $X$ obtained by gluing, for the convenience of notations. How these two points of view reconcile in practice is explained in section \ref{sec_num}.

Let $v,w$ be two distinct points and

\[\xi\in (\K+u^{g-1}(\D^{g-1}))\setminus \left[(\mathcal{C}+\K-u(v))\cup (\mathcal{C}+\K-u(w))\right],\]
where $\mathcal{C}$ is the set defined in subsection \ref{subsec_intro_weiestrass} as the image of the critical points of $u^g$.

In other words $\xi$ is chosen as

\[\xi=\K+\sum_{j=1}^{g-1}u(w_j)\]
where $w_1,\hdots,w_{g-1}$ are generic: since $\mathcal{C}$ has dimension at most $g-2$, generic points will give a suitable $\xi$. We let
\begin{equation}\label{Bjperiod_sigma}
\sigma_{v,w}(z)=\frac{\Theta(u(z)-u(v)-\xi)}{\Theta(u(z)-u(w)-\xi)}.
\end{equation}

Then $\sigma_{v,w}$ vanishes exactly with order $1$ at $\Gamma\cdot v$, and has one pole (modulo $\Gamma$) of order $1$ at $\Gamma\cdot w$. Moreover, $\sigma_{v,w}$ verifies $\sigma_{v,w}(A_j(z))=\sigma_{v,w}(z)$ and
\[\sigma_{v,w}(B_j(z))=\sigma_{v,w}(z)e^{2\pi i\left(u_j(v)-u_j(w)\right)}.\]
Let $(c_k)_k\in\C^g$ to be fixed later, and
\[\widehat{\sigma}_{v,w}(z)=\sigma_{v,w}(z)e^{2\pi i \sum_{k=1}^{g}c_k \Im(u_k(z))}.\]
Then $\widehat{\sigma}_{w_1,w_2}$ is $A_j$-periodic since $\sigma_{v,w}$ is too and $\Im(u_k(A_j z))=\Im(u_k(z)+\delta_{jk})=\Im(u_k(z))$. Next,
\begin{align*}
\frac{\widehat{\sigma}_{v,w}(B_jz)}{\widehat{\sigma}_{v,w}(z)}&=\exp\left(2\pi i\left[u_j(v)-u_j(w)+\sum_{k=1}^{g}c_k \Im(\tau)_{jk}\right]\right).
\end{align*}
Since $\Im(\tau)$ is invertible, we may choose the coefficients $(c_j)_{j}$ such that $\widehat{\sigma}_{v,w}$ is $\Gamma$-periodic: it  is sufficient to check that for  every $j$ we have
\[\sum_{k=1}^{g}c_k\Im(\tau)_{jk}=u_j(w)-u_j(v),\]
which amounts to choosing
\begin{equation}\label{def_sigma}
\widehat{\sigma}_{v,w}(z)=\sigma_{v,w}(z)\exp\left(-2\pi i\sum_{1\leq k,l\leq g}(\Im(\tau)^{-1})_{k,l}\Im(u_k(z))(u_l(v)-u_l(w))\right).
\end{equation}
 The function 
\begin{equation}\label{eq_logsigma}
\log|\widehat{\sigma}_{v,w}(z)|=\log|\sigma_{v,w}(z)|+2\pi\sum_{1\leq k,l\leq g}(\Im(\tau)^{-1})_{k,l}\Im(u_k(z))\Im(u_l(v)-u_l(w))
\end{equation}
is thus a periodic harmonic function with singularities equal to $\log(|z-v|)+\mathcal{O}_{z\to v}(1)$ near $v$, and to $-\log(|z-w|)+\mathcal{O}_{v\to w}(1)$ near $w$.

Note that the function $\widehat{\sigma}_{v,w}(z)$ might look like it depends on the choice of the generic point $\xi\in u^{g-1}(\D^{g-1})$. However, changing $\xi$ only changes $\widehat{\sigma}_{v,w}(z)$ by a non-zero constant factor. 

\subsection{Pole of order $1$}\label{subsec_p1}

We use the previously defined functions, where $w$ is the location of the pole and $v$ is seen as a variable. Let
\begin{align*}
\wp_1(z,w)&=2\partial_v|_{v=w}\log|\sigma_{v,w}(z)|,\\
\widehat{\wp}_1(z,w)&=2\partial_v|_{v=w}\log|\widehat{\sigma}_{v,w}(z)|,\\
\end{align*}
where $\partial_v|_{v=w}$ is the conformal Wirtinger derivative along the variable $v$, taken at the value $v=w$. In other words
\begin{align}\label{eq_orderone}
\wp_1(z,w)&=\Theta(u(z)-u(w)-\xi)^{-1}\partial_v|_{v=w}\Theta(u(z)-u(w)-\xi)\\
&=-\frac{u'(w)\cdot\nabla\Theta(u(z)-u(w)-\xi)}{\Theta(u(z)-u(w)-\xi)}\\
\widehat{\wp}_1(z,w)&=\wp_1(z,w)-2\pi i \sum_{1\leq k,l\leq g}(\Im(\tau)^{-1})_{k,l}\Im(u_k(z))u_l'(w).
\end{align}

The function $\widehat{\wp}_1(z,w)$ is still $\Gamma$-periodic along the variable $z$: by applying the Wirtinger derivative in $v$ at $v=w$ to the periodicity relation
\[\forall\gamma\in\Gamma,\ \log|\widehat{\sigma}_{v,w}(z)|=\log|\widehat{\sigma}_{v,w}(\gamma(z))|\]
we obtain
\[\forall\gamma\in\Gamma,\ \widehat{\wp}_1(z,w)=\widehat{\wp}_1(\gamma(z),w).\]

\subsection{Poles of order $2$ to $g$}\label{subsec_phat}

We define by induction
\begin{align*}
\wp_{n}(z,w)&=\partial_w^{(n-1)}\wp_1(z,w)\\
\widehat{\wp}_{n},(z,w)&=\partial_w^{(n-1)}\widehat{\wp}_1(z,w).
\end{align*}

More precisely,
\begin{equation}\label{eq_wpnfullexpression}
\widehat{\wp}_n(z,w)=\wp_n(z,w)-2\pi i\sum_{1\leq k,l\leq g}(\Im(\tau)^{-1})_{k,l}\Im(u_k(z))u_l^{(n)}(w).
\end{equation}

As before, $\widehat{\wp}_n(\gamma(z),w)=\widehat{\wp}_n(z,w)$ for any $\gamma\in\Gamma$. While the computation of $u_l^{(n)}(w)$ is straightforward using its polynomial approximation, the function $\wp_n$ will be computed by an analytic finite difference method explained in section \ref{sec_num}. 

Note that the function $\wp_n(z,w)$ by itself is holomorphic with respect to $z$ but only verifies the periodicity relation
\begin{equation}\label{eq_periodwpn}
\wp_n(A_j(z),w)=\wp_n(z,w),\ \wp_n(B_j(z),w)=\wp_n(z,w)+2\pi i u_j^{(m)}(w)
\end{equation}


\subsection{Poles of order $g+1$ to $2g+1$}\label{subsec_pbar}

The fact that $w$ is not a Weierstrass point appears in this construction. We first define $(\wp_n)_{g+1\leq n\leq 2g+1}$ and $\widehat{\wp}_n(z,w)$ in the same way as earlier, and we let $C(w)\in M_{g\times g}(\C)$ to be defined as
\[C(w)_{k,n}=u_k^{(n)}(w),\ 1\leq k\leq g,\ 1\leq n\leq g\]
such that for any $n\in\{1,\hdots,g\}$:
\[\widehat{\wp}_n(z,w)=\wp_n(z,w)-2\pi i\sum_{k=1}^{g}\left(\Im(\tau)^{-1}C(w)\right)_{k,n}\Im(u_k(z)).\]
Since $w$ is not a Weierstrass point, then $C(w)$ is invertible (as a consequence of lemma \ref{lemma_wei}).

For any $n=g+1,\hdots,2g+1$, we then let
\[d_n(w)=C(w)^{-1}u^{(n)}(w)\in\C^g\]
be the unique vector that verifies for any $k=1,\hdots,g$:
\[u_k^{(n)}(w)=\sum_{p=1}^{g}d_{n,p}(w)u^{(p)}_k(w).\]
Letting
\[\widecheck{\wp}_n(z,w)=\wp_n(z,w)-\sum_{p=1}^{g}d_{n,p}(w)\wp_p(z,w),\]
this is a meromorphic function with a pole of order $n$ at $z=w$, and it is $\Gamma$-periodic. Indeed, it is equal by construction to

\[\widehat{\wp}_n(z,w)-\sum_{p=1}^{g}d_{n,p}(w)\widehat{\wp}_p(z,w)\]
which is periodic.

\subsection{Poles of order $2g+2$ and more}\label{subsec_ptilde}

We now build a periodic meromorphic function with a pole of order $n$ at $w$ for $n\geq 2g+2$. For any such $n$ there exists $m\in\N^*$, $q\in\{g+1,g+2,\hdots,2g+1\}$ such that
\[n=(g+1)m+q.\]
We then define
\[\tilde{\wp}_{n}(z,w)= \widecheck{\wp}_q(z,w)\widecheck{\wp}_{g+1}(z,w)^m.\]
Note that we \textit{could} follow the previous construction: the functions $\widecheck{\wp}_n$ and $\widetilde{\wp}_n$ differ (up to a nonzero multiplicative constant) by a combination of $\widecheck{\wp}_m$ (or $\tilde{\wp}_m$) for $m<n$. The same can be said between $\widecheck{\wp}_n$ and $\widehat{\wp}_n$; the construction of $\tilde{\wp}_n$ is only for computational purposes.

\subsection{Adaptation to Weierstrass points}\label{subsec_weiestrass}

Let us now sketch how each of the previous results apply to Weierstrass points. First, as in section \ref{subsec_aj} we compute (using the least square method from theorem \ref{the_conv}) an approximation of a basis of $\Om^1(X)$ and of the Abel-Jacobi map.

For the construction of the basis of harmonic functions with prescribed poles, for a general Weierstrass point the matrix $C(w)$ defined in subsection \ref{subsec_pbar} is not invertible in general: we need to compute more of the functions $(\wp_n)_n$ to get a full-rank matrix, as follows:

\begin{itemize}
\item[1) ]We define $\log|\widehat{\sigma}_{v,w}|$ as previously.
\item[2) ]We define similarly $\wp_1(z,w)$, $\widehat{\wp}_1(z,w)$.
\item[3) ]By differentiation in $w$ we build the functions $\widehat{\wp}_{n}$ as previously, this time for the values $n=2,3,\hdots,2g-1$.
\item[4) ]We define the non-square matrix $\tilde{C}(w)\in M_{g\times (2g-1)}(\C)$ by:
\[\tilde{C}(w)_{k,n}=u_k^{(n)}(w),\ 1\leq k\leq g,\ 1\leq n\leq 2g-1.\]
By lemma \ref{lemma_wei}, $\tilde{C}(w)$ has rank $g$, so we find $d_n(w)\in\C^{2g-1}$ (which may be defined uniquely, up to restricting $\tilde{C}(w)$ to the right subspace of $\C^{2g-1})$t such that
\[\tilde{C}(w)d_n(w)=c^{(n-1)}(w),\]
and we may then define $\widecheck{\wp}_n(z,w)$ as previously, for $n=2g,2g+1,\hdots,4g-1$.
\item[5) ]For $n\geq 4g$, we make the same Euclidean division $n=2g m+q$ where $2g\leq q<4g$, and we let
\[\tilde{\wp}_{n}(z,w)= \widecheck{\wp}_q(z,w)\widecheck{\wp}_{2g}(z,w)^m.\]
\end{itemize}


\section{Proof of the main approximation results}\label{sec_proofmain}

Here we give the proof of the two main results \ref{mr_finitepoles} and \ref{mr_bernsteinWalsch}. 

While the first result is a simple consequence of our construction and of the maximum principle, the second one is more involved: we have to adapt to general surfaces the classical proof of Bernstein-Walsch theorem in $\C$ (see \cite[Ch. VII]{W35} or \cite[Th 6.3.1]{R95} for a more recent reference, to which we will refer for intermediate results) to prove the approximation theorem on holomorphic functions (see theorem \ref{th_bersnteinwalsch_holo}). Then transmit this result to harmonic functions, by relying on an interior approximation of the set $K$.

\begin{proof}[Proof of \ref{mr_finitepoles}]

Let $h$ be a function as in the statement and let $w_j$ be one of the poles. We may then develop $h$ near this point (here we have fixed a chart that induced by the gluing of polygons)

\[h(z)=a_{j,0}\log(|z-w_j|)+\sum_{k=1}^{n}\Re\left[a_{j,k}(z-w_j)^{-k}\right]+\mathcal{O}_{z\to w_j}(1),\]
where $a_{j,0}$ is real and $a_{j,k}$ is complex. We remind that this is obtained by integrating the Laurent series decomposition of $\partial_{z}h(z)$, which is meromorphic with a pole at $w_j$. By construction, there exists a constant $c\neq 0$ such that
\[\widehat{\wp}_k(z,w_j)=c(-1)^k (k-1)!(z-w_j)^{-k}+\Oo_{z\to w_j}\left((z-w_j)^{-(k-1)}\right).\]
So there exists $b_{j,0},b_{j,1},b_{j,2},\hdots,b_{j,n}$ such that the function
\[h(z)-\sum_{j=1}^{c-1}b_{j,0}\log\left|\sigma_{w_j,w_{j+1}}(z)\right|-\Re\sum_{j=1}^{c}\left(\sum_{k=1}^{g}b_{j,n}\widehat{\wp}_k(z,w_j)+\sum_{k=g+1}^{2g+1}b_{j,k}\widecheck{\wp}_k(z,w_j)+\sum_{k=2g+1}^{n}b_{j,k}\tilde{\wp}_k(z,w_j)\right)\]
is harmonic in $X\setminus\{w_j,j=1,\hdots,N\}$ and bounded near each pole, thus it is a constant function, from which we get the theorem.
\end{proof}

\subsection{Approximation of holomorphic functions}


Given an oriented compact surface $X$ and a finite set $\mathcal{P}\subset X$, we let $\Oo_n(X\setminus\mathcal{P})$ be the (finite dimensional) space of meromorphic functions $f\in M\Oo(X)$ such that
$$(f)+n\sum_{w\in\mathcal{P}}[w]\geq 0$$
or, said differently, the space of functions $f\in\Oo(X\setminus \mathcal{P})$ such that for some constant $C>0$, and for any $z\in X\setminus \mathcal{P}$, we have $|f(z)|\leq C d_X(z,\mathcal{P})^{-n}$.

Before stating the first approximation result, we prove several preliminary lemmas. We let $\HH:X^2\to \R\cup\{-\infty\}$ be the Green function on $X$, defined by
\[\forall w\in X,\ \Delta_z \HH(\cdot,w)=2\pi\left(\delta_w-\frac{\mu_X}{\mu_X(X)}\right),\ \int_{X}\HH(z,w)d\mu_X(z)=0\]
where $\mu_X$ is the area measure induced by the hyperbolic metric. Note that $\HH(x,y)=\HH(y,x)$ with this normalization, as may be seen from integrating $z\mapsto \HH(z,x)\Delta_z\HH(z,y)-\HH(z,y)\Delta_z\HH(z,x)$ on $X\setminus (\D_{X}(x,\eps)\cup\D_{X}(y,\eps))$ for $\eps\to 0$ (where the disks are taken with respect to the hyperbolic metric).

\begin{lemma}\label{est_green}
Let $X$ be an oriented compact surface. There exists $c_1,c_2>0$ only depending on $X$, such that for any $(z,w)\in X^2$ we have
$$c_1+\log d_X(z,w)\leq \HH(z,w)\leq c_2.$$
\end{lemma}
\begin{proof}
Let $\varphi:\D_1\to X$ be a chart of $X$, and $\chi\in\mathcal{C}^\infty_c(\D_1,\R)$ such that $\chi=1$ on $\D_{\frac{3}{4}}$. Then for any $w\in \varphi(\D_{\frac{1}{2}})$, we may write
$$\mathcal{H}(z,w)=\chi\circ \varphi^{-1}(z)\log\left|\varphi^{-1}(z)-\varphi^{-1}(w)\right|+g_w(z)$$
where the first term vanishes outside $\varphi(\D_1)$ by convention, and $\Vert \Delta g_w\Vert_{L^\infty(X)}$, $\int_X g_w$ are bounded independently of $w$. By elliptic regularity we then have a uniform bound of $\Vert g_w\Vert_{L^\infty(X)}$ for any $w\in \varphi(\D_{\frac{1}{2}})$, and the result follows by covering $X$ with a finite number of such charts. 
\end{proof}

\begin{lemma}\label{lem_omz}
Let $X$ be an oriented compact surface, $w\in X$, then for any $z\in X\setminus\{w\}$ there exists a meromorphic $1$-form $\om_z$ with exactly a pole of order $1$ at $z$ and $w$, with residues $1,-1$ respectively, such that $z\mapsto \om_z$ is itself meromorphic with respect to $z$, with a pole at $z=w$ of order at most $2g-1$.
\end{lemma}

\begin{proof}
In genus $0$ ($X=\widehat{\C}$), this is the role played by $\zeta\mapsto\frac{1}{\zeta-z}-\frac{1}{\zeta-w}$ (when $w\in \C$) or  $\frac{1}{\zeta-z}$ (when $w=\infty$).

In genus $1$, suppose (without loss of generality) $X$ is of the form $\C/L$ for some lattice $L=\Z+\tau\Z$ where $\Im(\tau)>0$. Let us consider the zeta function
\[\mathcal{Z}(z)=\frac{1}{z}+\sum_{\ell\in L\setminus\{0\}}\left\{\frac{1}{z-\ell}+\frac{1}{\ell}+\frac{z}{\ell^2}\right\}.\]
We remind the periodicity of $\mathcal{Z}$, due to Eisenstein:
\[z\mapsto \widehat{\mathcal{Z}}(z):=\mathcal{Z}(z)-\gamma_2 z-\frac{\pi}{\Im(\tau)}\ov{z}\text{ is }L\text{-periodic},\]
where $\gamma_2=\sum_{\ell\in L\setminus\{0\}}\frac{1}{\ell^4}$. Then we let
\[\om_z(\zeta)=\left(\mathcal{Z}(\zeta-z)+\mathcal{Z}(z-w)-\mathcal{Z}(\zeta-w)\right).\]
$\om_z$ is meromorphic in $\C$ with respect to $\zeta$, $z$ with the appropriate poles. Moreover, since $\mathcal{Z}(\zeta-z)+\mathcal{Z}(z-w)-\mathcal{Z}(\zeta-w)=\widehat{\mathcal{Z}}(\zeta-z)+\widehat{\mathcal{Z}}(z-w)-\widehat{\mathcal{Z}}(\zeta-w)$ then it is also $L$-periodic with respect to $\zeta,z$.

We now suppose $g\geq 2$. To build this $1$-form we work from the disk model $\D/\Gamma$. As in section \ref{subsec_log}, we let $\xi$ be a generic point of $\K+u^{g-1}(\D^{g-1})$, and
\[\sigma_{z,w}(\zeta)=\frac{\Theta(u(\zeta)-u(z)-\xi)}{\Theta(u(\zeta)-u(w)-\xi)}.\]
We let
\[\alpha_z(\zeta)=2\partial_{\zeta}\log\left|\sigma_{z,w}(\zeta)\right|=\frac{\sigma_{z,w}'(\zeta)}{\sigma_{z,w}(\zeta)}.\]
Then $\alpha_z\in M\Om^1(\D/\Gamma)$ for every $z$, with the poles exactly as in the statement of the lemma when $z\neq w$. Moreover, it is holomorphic in $\D$ with respect to $z$ (in particular $\alpha_w=0$). However, it is not periodic with respect to $z$, indeed a direct computation gives
\[\alpha_{A_j(z)}(\zeta)=\alpha_{z}(\zeta),\ \alpha_{B_j(z)}(\zeta)=\alpha_{z}(\zeta)+2\pi i \om_j(\zeta).\]
We define, for some functions $c_k(z)$ to be fixed,
\[\om_z(\zeta)=\alpha_z(\zeta)- \sum_{k=1}^{g}c_k(z)\om_k(\zeta),\]
where $c_k\in M\mathcal{O}(\D)$ has poles at $\Gamma\cdot w$ and must verify the periodicity relations
\[\forall j,k\in\{1,\hdots,g\},\ c_k(A_j(z))=c_k(z),\ c_k(B_j(z))=c_k(z)+2\pi i\delta_{j,k}.\]
We build $c_k$ as a combination of the functions $(\wp_m)_{m=1,2,\hdots,2g-1}$ (which could be replaced by $m=1,2,\hdots,g$ when $w$ is not a Weierstrass point) built in subsection (\ref{subsec_p1},\ref{subsec_phat}). We remind the periodicity of $\wp_m$:
\begin{equation}\label{eq_periodwpm}
\wp_{m}(A_j(z),w)=\wp_{m}(z,w),\ \wp_{m}(B_j(z),w)=\wp_{m}(z,w)+2\pi i \om_j^{(m-1)}(w).
\end{equation}

Letting (as in subsection \ref{subsec_weiestrass}) $\tilde{C}(w)=(\om_j^{(m-1)}(w))_{1\leq j\leq g,\ 1\leq m\leq 2g-1}$, then $\tilde{C}(w)$ has rank $g$ according to lemma \ref{lemma_wei} so for each $k$ there exists some $b^k=(b^k_m)_{0\leq m\leq 2g-2}\in\C^{2g-1}$ such that $\tilde{C}(w)b^k=e_k$. We then define
\[c_k(z)=\sum_{m=0}^{2g-2}b^k_m \wp_m(z,w).\]
This way $\om_z$ is periodic with respect to $z$.
\end{proof}

We are now ready to prove a first version of the approximation result for holomorphic function. Note that this version does not require the stronger regularity hypothesis $\mathcal{R}^{\mathrm{strong}}$.

\begin{theorem}\label{th_bersnteinwalsch_holo}
Let $X$ be an oriented compact surface, $K$ a compact subset of $X$, $\mathcal{P}$ a finite subset of $X\setminus K$ such that $(X,K,\mathcal{P})$ verify the hypothesis $\mathcal{R}^{\mathrm{weak}}$. 

Let $f:K\to\C$. Then for any $t>0$ the two following properties are equivalent:
\begin{itemize}
\item[(a) ]$f$ extends to a holomorphic function defined on $\left\{\mathcal{G}_{K,\mathcal{P}}<t\right\}$.
\item[(b) ]$\limsup_{n\to\infty}\inf_{f_n\in\Oo_n\left(X\setminus\mathcal{P}\right)}\Vert f-f_n\Vert_{L^\infty(K)}^{1/n}\leq e^{-t}$.
\end{itemize}
\end{theorem}

\begin{proof}
We start with $(b\Rightarrow a)$. For any $n\geq 1$, denote $g_n=f_n-f_{n-1}$, such that
\[\limsup_{n\to\infty}\Vert g_n\Vert_{L^\infty(K)}^{1/n}\leq e^{-t}.\]
$\log|g_n|$ is subharmonic in $X\setminus \mathcal{P}$ and more precisely $\Delta\log|g_n|\geq -2\pi n\sum_{w\in\mathcal{P}}\delta_{w}$, so the function
$$z\in X\setminus K\mapsto \log\frac{|g_n(z)|}{\Vert g_n\Vert_{L^\infty(K)}}-\mathcal{G}_{K,\mathcal{P}}(z)$$ is subharmonic in $X\setminus K$ and nonpositive in $\partial K$. By the maximum principle (that we apply in each connected component of $X\setminus K$) it is nonpositive in $X\setminus K$. This implies that for any $s\in (0,t)$, and any $z\in  \{\mathcal{G}_{K,\mathcal{P}}<s\}$:
\[|g_n(z)|\leq \Vert g_n\Vert_{L^\infty(K)}e^{ns}.\]
So $\limsup_{n\to\infty}\Vert g_n\Vert_{L^\infty(\{\mathcal{G}_{K,\mathcal{P}}<s\})}^\frac{1}{n}\leq e^{s-t}(<1)$, meaning $\Vert g_n\Vert_{L^\infty( \{\mathcal{G}_{K,\mathcal{P}}<s\})}$ is summable for all $s<t$. Then $f$  extends holomorphically to $\left\{\mathcal{G}_{K,\mathcal{P}}<t\right\}$ as $f_0+\sum_{n\geq 1}g_n$.

We now prove to the other implication $(a\Rightarrow b)$. By our regularity hypothesis on $K$, $\mathcal{G}_{K,\mathcal{P}}$ is continuous and equal to $0$ on $\partial K$. We claim it is sufficient to prove the result on the set 
$$K_\eps=\left\{\mathcal{G}_{K,\mathcal{P}}\leq \eps\right\}$$
for arbitrarily small $\eps\in (0,t)$: indeed $\mathcal{G}_{K_\eps,\mathcal{P}}=\left(\mathcal{G}_{K,\mathcal{P}}-\eps\right)_+$ (since the difference of the two functions is harmonic in $X\setminus K_\eps$ and vanishes on $\partial K_\eps$), so proving the result on $K_\eps$ implies 
$$\limsup_{n\to\infty}\inf_{f_n\in\Oo_n(X\setminus\{w\})}\Vert f-f_n\Vert_{L^\infty(K)}^{1/n}\leq \limsup_{n\to\infty}\inf_{f_n\in\Oo_n(X\setminus\{w\})}\Vert f-f_n\Vert_{L^\infty(K_\eps)}^{1/n}\leq e^{\eps-t}.$$
Thus, in the rest of the proof, up to replacing $K$ with $K_\eps$ we assume without loss of generality that $\partial K$ is a finite union of smooth curves (since it is the case for $K_\eps$, for almost every $\eps$).

We denote by $\nu$ the measure supported on $K$ such that we have, in the sense of distributions,
$$\Delta \mathcal{G}_{K,\mathcal{P}}=2\pi \left(\nu-\sum_{w\in\mathcal{P}}\delta_{w}\right) \text{ in }\mathcal{D}'(X).$$
In particular $\nu(X)=|\mathcal{P}|$ (the cardinal of $\mathcal{P}$). We claim there exists a sequence of points $(z_{j,n})_{1\leq j\leq n,\ n\in\N^*}$ such that the associated measure
$$\nu_n=\frac{|\mathcal{P}|}{n}\sum_{j=1}^{n}\delta_{z_{j,n}}$$
converges to $\nu$, and such that for any large enough $n$ there exists a meromorphic function $F_n\in M\Oo(X)$ that verifies
\begin{equation}\label{eq_divFn}
(F_n)=|\mathcal{P}|\sum_{j=1}^{n}[z_{j,n}]-n\sum_{w\in\mathcal{P}}[w].
\end{equation}

\begin{proof}[Proof of the claim]
We remind that we denote by $\ov{u}:X\to \C^g/(\Z^g+\tau\Z^g)$ the quotiented Abel-Jacobi map, and
$$\ov{u}^{k}(z_1,\hdots,z_k)=\sum_{j=1}^{k}\ov{u}(z_j).$$
$\ov{u}^g$ is a diffeomorphism outside a codimension $1$ set, and $K$ has non-empty interior (since we assumed $K$ to be smooth without loss of generality) so $\ov{u}^g(K^g)$ contains an open set of $\C^g/(\Z^g+\tau\Z^g)$. As a consequence, since $\C^g /(\Z^g+\tau\Z^g)$ is compact, for some large enough $m\in\N$ we have 
\[\ov{u}^{mg}(K^{mg})=\underbrace{\ov{u}^g(K)+\hdots+\ov{u}^g(K)}_{m\text{ times}}=\C^g/(\Z^g+\tau\Z^g).\]
We fix $q=mg$. Let $z_{1,n},\hdots,z_{n-q,n}$ be points of $\partial K$ chosen such that
$$\frac{|\mathcal{P}|}{n}\sum_{j=1}^{n-q}\delta_{z_{j,n}}\underset{n\to +\infty}{\rightharpoonup^*}\nu$$
Since $\ov{u}^q$ is surjective from $K$ to $\C^g/(\Z^g+\tau\Z^g)$, we may choose $z_{n-q+1,n},z_{n-q+2,n},\hdots,z_{n,n}\in K$ that verify
$$|\mathcal{P}|\sum_{j=n-q+1}^{n}\ov{u}(z_{j,n})=n \sum_{w\in\mathcal{P}}\ov{u}(w)-|\mathcal{P}|\sum_{j=1}^{n-q}\ov{u}(z_{j,n})$$
Then $(z_{1,n},\hdots,z_{n,n})$ verifies the claim: indeed we have $\nu_n\rightharpoonup^*\nu$ and
$$|\mathcal{P}|\sum_{j=1}^{n}\ov{u}(z_{j,n})-n\sum_{w\in\mathcal{P}}\ov{u}(w)=0$$
so by Abel's theorem (see for instance \cite[Th. III.6.3]{FKH92}), there exists a function $F_n\in M\Oo(X)$ verifying \eqref{eq_divFn}.
\end{proof}
Let
$$G_n(z):=\frac{|\mathcal{P}|}{n}\sum_{j=1}^{n}\mathcal{H}(z,z_{j,n})-\sum_{w\in\mathcal{P}}\mathcal{H}(z,w).$$
By lemma \ref{est_green} we have $\sup_{n}\int_{X}\left|G_n(z)\right|^2d\mu_X(z)<+\infty$. Let $(n_k)_{k\in\N}$ be some extraction of $\N$, then there is some $G\in L^2(X)$ and some subsequence $(n_{k_i})_i$ such that we have the weak convergence 
$$G_{n_{k_i}}\underset{L^2(X)}{\rightharpoonup}G\text{ as }i\to +\infty.$$
For any $\varphi\in \mathcal{C}^2(X)$ we have $\int_{X}G_{n}\Delta\varphi d\mu_X=2\pi \int_{X}\varphi d\left(\nu_{n}-\sum_{w\in\mathcal{P}}\delta_{w}\right)$ which passes to the limit as
$$\Delta G=2\pi (\nu-\delta_{w})\text{ in }\mathcal{D}'(X).$$
Since $\int_{X}Gd\mu_X=0$, we uniquely identify $G=\mathcal{G}_{K,\mathcal{P}}-\fint_{X}\mathcal{G}_{K,\mathcal{P}}d\mu_X$. Thus, the full sequence $G_n$ converges weakly in $L^2(X)$ to $\mathcal{G}_{K,\mathcal{P}}-\fint_{X}\mathcal{G}_{K,\mathcal{P}}d\mu_X$. Since each $G_n$ is harmonic in $X\setminus (K\cup \mathcal{P})$, the convergence is locally uniform in $X\setminus (K\cup \mathcal{P})$.

We remind that the points $(z_{j,n})$ were chosen so that there exists some meromorphic $F_n$ verifying \eqref{eq_divFn}. Up to multiplying $F_n$ by a scalar factor, we have
$$\frac{1}{n}\log\left|F_n(z)\right|=\fint_{X}\mathcal{G}_{K,\mathcal{P}}d\mu_X+G_n(z).$$
So the left-hand side converges locally uniformly in $X\setminus (K\cup\mathcal{P})$ to $\mathcal{G}_{K,\mathcal{P}}$. Since $\log|F_n|$ is subharmonic in $X\setminus \mathcal{P}$ we obtain that:
\begin{equation}\label{eq_convergenceFn}
\lim_{n\to +\infty}\Vert F_n\Vert_{L^\infty(K)}^{1/n}\leq 1 
\end{equation}

Consider now $K_t=\{\mathcal{G}_{K,\mathcal{P}}\leq t\}$. Up to a small downward perturbation of $t$ (i.e. replacing $t$ with $t-\delta$ for some generic arbitrarily small $\delta>0$) we may suppose that $\partial K_t$ is smooth. Let $w_0\in\mathcal{P}$, and $\om_z$ be the $1$-form built in lemma \ref{lem_omz} from the base point $w_0$. Let then
\[f_n(z)=\frac{1}{2\pi i}\oint_{\partial K_t}\frac{F_n(\zeta)-F_n(z)}{F_n(\zeta)}f(\zeta)\om_z(\zeta)d\zeta,\ \forall z\in K_t\]
By residue theorem,
\[f_n(z)=-\sum_{j=1}^{n}f(z_{j,n})F_n(z)\text{res}\left(\zeta\mapsto \frac{\om_z(\zeta)d\zeta}{F_n(\zeta)},z_{j,n}\right),\]
so $f_n$ is meromorphic with
$$(f_n)\geq -(n+2g-1)[w_0]-n\sum_{w\in\mathcal{P}\setminus \{w_0\}}[w]$$
since $F_n$ has order $n$ at the points of $\mathcal{P}$, and $\om_z$ has order at most $2g-1$ at $w_0$. Again by residue theorem $f(z)=\frac{1}{2\pi i}\oint_{\partial K_t}f(\zeta)\om_z(\zeta)d\zeta$ for any $z\in K_t$, so
\begin{align*}
f(z)-f_n(z)=\frac{1}{2\pi i}\oint_{\partial K_t}\frac{F_n(z)}{F_n(\zeta)}f(\zeta)\om_z(\zeta)d\zeta.
\end{align*}
For some constant $C>0$, and for any $z\in K$, we obtain:
\[|f(z)-f_n(z)|^\frac{1}{n+2g-1}\leq \left(C\frac{\Vert F_n\Vert_{L^\infty(K)}}{\inf_{\partial K_t}| F_n|}\right)^\frac{1}{n+2g-1}\underset{n\to\infty}{\longrightarrow}e^{-t}.\]
\end{proof}

\subsection{From holomorphic to harmonic functions}\label{subsec_holotoharmo}

The growth of harmonic function outside a compact set depending on their degree is less straightforward than it is for holomorphic function: as an example (which is  discussed already in \cite{W29}) we may take a segment $[-1,1]$ in the surface $\widehat{\C}$ (so that harmonic functions with a pole at infinity are exactly the harmonic polynomials). Then there is no ``potential function'' $V:\C\to\R_+$ such that for any harmonic polynomial $h$ of degree $n$ we would have $|h|\leq Ce^{nV}\Vert h\Vert_{[-1,1]}$. Indeed, this is immediately contradicted by considering $h(z)=\lambda\Im(z)$ for $\lambda\to \infty$.

While this counterexample is a clear consequence of the non-injectivity of the restriction map $h\in\HH_{n}(\C)\to h|_{[-1,1]}\in L^\infty([-1,1])$, and may be adjusted by choosing a ``correct'' right-inverse to this restriction, whether this can be done for more general compacts set is unclear.

For this reason, we will assume a stronger regularity hypothesis on $K$, that guarantees a fixed growth of harmonic function as described above. In particular, this hypothesis forces $K$ to have non-empty interior, which implies that $h\mapsto h|_{K}$ is injective. 
A higher dimensional version of this is Plesniak's theorem that may be found in \cite[Cor. 3.2]{plesniak1984criterion}, see also \cite{BL07} for a self-contained exposition). The method of Plesniak relies on writing harmonic polynomials as restrictions of holomorphic functions in $z,\ov{z}$, i.e. by writing
$$\Re(f(z))=\frac{f(z)+g(\ov{z})}{2}$$
where $g(z)=\ov{f(\ov{z})}$ is holomorphic. Then they establish a limit growth for the function $f(z)+g(\zeta)$ for any $z,\zeta\in \C^2$. In our case several elements of the basis cannot be written as the real part of some meromorphic function, so this method does not seem to extend : instead we estimate the Wirtinger derivative of our function in a sufficiently large compact subset of the interior of $K$, and use the regularity hypothesis $\mathcal{R}^{strong}$ to prove that it implies the same bound in a neighbourhood of $K$.

\begin{lemma}\label{lem_constructionom}
Let $X$ be a compact oriented surface of genus $g\geq 0$, let $w\in X$ and let $D$ be some open neighbourhood of $w$. Then there exists $\om\in M\Om^1(X)$ such that $\om$ has no zeroes or poles outside $D$, and $\om$ may vanish only at $w$.
\end{lemma}
\begin{proof}
This is obvious in genus $0$, so we consider only genus $g\geq 1$. Let $$\sum_{j=1}^{2g-2}[\zeta_j]$$
be the divisor of some holomorphic $1$-form denoted $\om_0$. Let $\ov{u}:X\to \C^g/(\Z^g+\tau\Z^g)$ be the Abel-Jacobi map. $\ov{u}^g$ is open, so there exists some large enough $q\in \N$ such that $\ov{u}^{q}(D^q)=\C^g/(\Z^g+\tau\Z^g)$.

As a consequence, There exists $z_1,\hdots,z_q\in D$ such that
$$\sum_{j=1}^{q}\ov{u}(z_j)=(q-2g+2)\ov{u}(w)+\sum_{j=1}^{2g-2}\ov{u}(\zeta_j)$$
By Abel's theorem, there exists $f\in M\Oo(X)$ such that
$$(f)=(q-2g+2)[w]+\sum_{j=1}^{2g-2}[\zeta_j]-\sum_{j=1}^{q}[z_j]$$
and $\om=f\om_0$ is a meromorphic $1$-form that satisfies the conclusion of the lemma.
\end{proof}

\begin{lemma}\label{lem_harmonicgrowth}
Let $X$ be an oriented compact surface, $K$ a compact subset of $X$, $\mathcal{P}$ a finite subset of $X\setminus K$ such that $(X,K,\mathcal{P})$ verify the hypothesis $\mathcal{R}^{\mathrm{strong}}$. Then for any $t>0$, $\delta>0$, there exists $C>0$ such that for any $n\in\N$, $h\in\mathcal{H}_n(X\setminus\mathcal{P})$:
$$\Vert h\Vert_{L^\infty(\{\mathcal{G}_{K,\mathcal{P}}<t\})}\leq Ce^{(t+\delta)n}\Vert h\Vert_{L^\infty(K)}.$$
\end{lemma}
\begin{proof}
For any $\eps>0$, let
$$K_{-\eps}=\left\{z\in K:d_X(z,\partial K)\geq\eps\right\}.$$
Since $K$ has non-empty interior, for a small enough $\eps$ the compact set $K_{-\eps}$ has non-empty interior, in particular it is non-polar so its Green function  $\mathcal{G}_{K_{-\eps},\mathcal{P}}$ is well-defined.

We claim that $\mathcal{G}_{K_{-\eps},\mathcal{P}}$ converges locally uniformly in $X\setminus (K\cup\mathcal{P})$ to $\mathcal{G}_{K,\mathcal{P}}$.

\begin{proof}[Proof of the claim]

By maximum principle, for any $0<\eps<\eps'$ small enough such that $\mathcal{G}_{K_{-\eps'},\mathcal{P}}$ exists, we have
$$\mathcal{G}_{K,\mathcal{P}}\leq\mathcal{G}_{K_{-\eps},\mathcal{P}}\leq \mathcal{G}_{K_{-\eps'},\mathcal{P}}.$$
Let $\mathcal{G}(z)$ be the limit of $\mathcal{G}_{K_{-\eps},\mathcal{P}}(z)$.  Since $\mathcal{G}_{K_{-\eps},\mathcal{P}}$ is harmonic in $X\setminus (K\cup\mathcal{P})$, so is $\mathcal{G}$ and the convergence $\mathcal{G}_{K_\eps,\mathcal{P}}\to \mathcal{G}$ is local uniform in $X\setminus (K\cup\mathcal{P})$.

By \cite[Th. 2.4.6]{R95}, $\mathcal{G}$ is subharmonic on $X\setminus \mathcal{P}$ and for every (small enough) $\eps>0$ we have
$$\mathcal{G}_{K,\mathcal{P}}(z)\leq \mathcal{G}(z)\leq \mathcal{G}_{K_{-\eps},\mathcal{P}}(z).$$
This implies that $\mathcal{G}=0$ in the interior of $K$ and
$$\mathcal{G}(z)=-\log d_X(z,w)+\mathcal{O}_{z\to w}(1)$$
for any $w\in\mathcal{P}$. 

To prove that $\mathcal{G}=\mathcal{G}_{K,\mathcal{P}}$, it is enough to prove that $\mathcal{G}(z_0)=0$ for every $z_0\in K$. By our hypothesis $\mathcal{R}^{strong}$, the interior of $K$ is non-thin at $z_0$, and $\mathcal{G}$ is zero is the interior of $K$, thus by the definition of non-thinness we have
$$\mathcal{G}(z_0)\leq \limsup_{z\to z_0,z\in\intK}\mathcal{G}(z)=0.$$
Thus, $\mathcal{G}=\mathcal{G}_{K,\mathcal{P}}$, and this concludes the claim.
\end{proof}

Let now $t,\delta>0$ be as in the statement. We lose no generality in supposing that $\{\mathcal{G}_{K,\mathcal{P}}<t\}$ is smooth.
We fix a small enough $\eps>0$ such that
$$\{\mathcal{G}_{K,\mathcal{P}}<t\}\subset \{\mathcal{G}_{K_{-\eps},\mathcal{P}}<t+\delta\}.$$
Consider $\om$ the $1$-form given by lemma \ref{lem_constructionom} applied to some $w_0\in\mathcal{P}$, with the neighbourhood $D=\left\{\mathcal{G}_{K_{-\eps},\mathcal{P}}>t+\delta \right\}$; in particular, $\om$ has no zeroes or poles in $\left\{\mathcal{G}_{K_{-\eps},\mathcal{P}}<t+\delta \right\}$. 
We remind the following estimate for holomorphic functions in the unit disk $\D_1\subset\C$: for some constant $C>0$ and for any $g\in\Oo(\D_1)$, we have $\Vert g'\Vert_{L^\infty(\D_{1/2})}\leq C\Vert\Re(g)\Vert_{L^\infty(\D_1)}$. 

Let $h\in\mathcal{H}_n(X\setminus\mathcal{P})$ (for some $n\geq 1$). By covering $K_{-\eps}$ with a finite number of disk of radius $\eps$ (so included in $K$), we find that for some constant $C_{\eps}>0$:
$$\Vert\partial_zh/\om\Vert_{L^\infty(K_{-\eps})}\leq C_{\eps}\Vert h\Vert_{L^\infty(K)}$$
Observe that the function
$\partial_zh/\om$ 
(which is well-defined as a ratio of $1$-forms) belongs to $\Oo_{n+N+1}(X\setminus\mathcal{P})$ (where $N$ is the order of vanishing of $\om$ at $w_0$). By maximum principle on the subharmonic function
$$z\in X\setminus K_{-\eps}\mapsto \frac{1}{n+N+1}\log\left|\frac{\partial_z h}{\om}\right|-\mathcal{G}_{K_{-\eps},\mathcal{P}}(z)$$
we obtain that for every $z\in \left\{\mathcal{G}_{K,\mathcal{P}}<t \right\}(\subset \left\{\mathcal{G}_{K_{-\eps},\mathcal{P}}<t+\delta \right\})$, we have
$$\left|\frac{\partial_z h}{\om}(z)\right|\leq C_\eps e^{(t+\delta)n}\Vert h\Vert_{L^\infty(K)}$$
Since $\left\{\mathcal{G}_{K,\mathcal{P}}<t \right\}$ is assumed to be smooth, for every $z\in \left\{\mathcal{G}_{K,\mathcal{P}}<t \right\}$ there is a smooth curve $$c_z:[0,1]\to \left\{\mathcal{G}_{K,\mathcal{P}}<t \right\}$$
such that $c_z(0)\in K$, $c_z(1)=z$, and the length of $c_z$ is bounded independently of $z$ (by some positive constant $L>0$). Then for any $z\in \left\{\mathcal{G}_{K,\mathcal{P}}<t \right\}$ we have $h(z)=h(c_z(0))+2\Re\int_{c_z}\frac{\partial_z h}{\om}\om$ so for some constant $C>0$ (depending on $t,\eps,\delta$), we get:
$$|h(z)|\leq Ce^{n(t+\delta)}\Vert h\Vert_{L^\infty(K)}.$$
\end{proof}

\begin{proof}[Proof of \ref{mr_bernsteinWalsch}]
We suppose first that $X$ is oriented.
We start from the implication $(a)\Rightarrow(b)$. Consider  $(h_n)_{n\in\N}$ a sequence with $h_n\in\HH_n(X\setminus\mathcal{P})$ such that
$$\limsup_{n\to +\infty}\Vert h_n-h\Vert_{L^\infty(K)}^{1/n}\leq e^{-t}.$$
Let $s\in (0,t)$, $\delta\in (0,t-s)$. We apply lemma \ref{lem_harmonicgrowth} to the sequence $(h_n-h_{n-1})_n$: for some constant $C>0$ (that is independent of $n$) we have
$$\forall z\in\{\mathcal{G}_{K,\mathcal{P}}<s\},\ |h_n(z)-h_{n-1}(z)|\leq C e^{n(s+\delta)}\Vert h_n-h_{n-1}\Vert_{L^\infty(K)}$$
The right-hand side is summable, and we let
$$\tilde{h}(z)=h_0(z)+\sum_{n\geq 1}(h_n(z)-h_{n-1}(z))$$
be its limit. Then $\tilde{h}$ is harmonic in $\{\mathcal{G}_{K,\mathcal{P}}<t\}$ (since we may consider $s$ arbitrarily close to $t$) and coincides with $h$ on $K$.\bigbreak

We then prove the implication $(b)\Rightarrow(a)$. Consider $h$ a harmonic function as in the statement of the theorem, defined in some neighbourhood of $K$ denoted $$K_t=\left\{\mathcal{G}_{K,\mathcal{P}}\leq t\right\}.$$
Up to a small perturbation of $t$ we may suppose $\partial K_t$ is smooth. We let $D_1,\hdots,D_c$ be the connected components of $X\setminus K_t$, and $w_1,\hdots,w_c$ is an (arbitrarily) choice of elements of $\mathcal{P}$ such that $w_j\in D_j$. Let
$$r_j(h)=\frac{1}{2\pi i}\oint_{D_j}\partial_z h.$$
By residue theorem in $K_t$ we have $\sum_{j=1}^{c}r_j=0$. Let
$$h^{res}=\sum_{j=1}^{c}r_j(h)\mathcal{H}(z,w_j).$$
Then $h^{res}$ is harmonic such that $r_j(h^{res})=r_j(h)$ for every $j$, and the periods of $\partial_z(h-h^{res})$ are well-defined in the following sense: for any smooth open set $E\subset X$ such that $\partial E\subset K_t$, we have
$$\oint_{\partial E}\partial_z(h-h^{res})=0.$$
Indeed, let $J=\{j=1,\hdots,c\text{ s.t. }D_j\subset E\}$, then since $r_j(h)=r_j(h^{res})$ for every $j\in J$ we have
$$\oint_{\partial E}\partial_z(h-h^{res})=\oint_{\partial (E\setminus \cup_{j\in J}D_j)}\partial_z(h-h^{res})=0.$$
Let now $V$ be the subspace of the homology space $H_1(X,\C)$ that may be represented by a linear combination of curves in $K_t$: by the above consideration, $\partial_z(h-h^{res})$ defines an element of $V^*$ by through the linear application
\begin{equation}\label{def_linearmap}
L:\mathcal{C}\in V\mapsto \int_{\mathcal{C}}\partial_z(h-h^{res})
\end{equation}
Let now
$$G(w_1)=\{n\in\N^*:\ell(n[w_1])=\ell((n-1)[w_1])\}.$$
We remind that $G(w_1)=\{1,\hdots,g\}$ when $w_1$ is not a Weierstrass point, and in general it is a subset of $\{1,\hdots,2g-1\}$ of cardinal $g$ such that there is no meromorphic function on $X$ with only a pole of order $m\in G(w_1)$ at $w_1$. Denote $W$ the function space spanned by
$$\left(\widehat{\wp}_m(\cdot,w_1)\right)_{m\in G(w_1)},\ (\ov{\widehat{\wp}_m(\cdot,w_1)})_{m\in G(w_1)}.$$
We claim the bilinear form
$$
B:\begin{cases}
H_1(X,\C)\times W&\to \C\\
\mathcal{C},\rho&\mapsto \oint_{\mathcal{C}}\partial_z\rho
\end{cases}
$$
is nondegenerate.
\begin{proof}[Proof of the claim]
Since $H_1(X,\C)$ and $W$ have the same dimension, it is sufficient to prove that if for some $\rho\in W$ and for every $\mathcal{C}\in H_1(X,\C)$, we have $\oint_{\mathcal{C}}\partial_z\rho=0$, then necessarily $\rho=0$.
Under this condition there exists some $f\in M\Oo(X)$ with only a pole at $w_1$ such that $\partial_z f=\partial_z\rho$. Let $(\lambda_m)_{m\in G(w_1)}$, $(\mu_m)_{m\in G(w_1)}$ be coefficients such that
$$\rho=\sum_{m\in G(w_1)}\left(\lambda_m\widehat{\wp}_m(\cdot,w_1)+\mu_m\ov{\widehat{\wp}_m(\cdot,w_1)}\right)$$
Suppose there is some nonzero coefficient $\lambda_m$. Let $m_0$ be the maximal index $m$ such that $\lambda_m\neq 0$, then $\partial_z \rho$ has a pole of order $m_0+1$ at $w_1$ (because $\partial_z\ov{\widehat{\wp}_m(z,w_1)}$ is holomorphic for any $m$).
Thus, $f$ has a pole of order $m_0$ at $w_1$, which is in contradiction with the fact that $m_0$ is in $G(w_1)$.

As a consequence $\lambda_m=0$ for every $m$, so $\partial_z f\in \Om^1(X)$. In particular, $f$ has no pole at $w_1$, so it must be constant, meaning $\partial_z \rho=0$, which implies $\rho=0$.
\end{proof}

We now conclude the proof: let $L'\in H_1(X,\C)^*$ be an extension of $L$ (defined in \eqref{def_linearmap}), and $h^{per}\in W$ such that $B(\cdot,h^{per})=L'(\cdot)$. Then for any closed loop $\mathcal{C}$ in $K_t$, we have
$$\int_{\mathcal{C}}\partial_z(h-h^{res}-h^{per})=L(\mathcal{C})-B(\mathcal{C},h^{per})=0.$$
Thus, there exists $f\in\Oo(K_t)$ such that
$$h=h^{res}+h^{per}+\Re(f).$$
By theorem \ref{th_bersnteinwalsch_holo}, there exists a sequence $(f_n)_n$ such that $f_n\in\Oo_n(X\setminus \mathcal{P})$ and
$$\limsup_{n\to+\infty}\Vert f-f_n\Vert_{L^\infty(K)}^{1/n}\leq e^{-t}.$$
Then $h_n=h^{res}+h^{per}+\Re(f_n)$ verifies the conclusion $(a)$.

When $X$ is not oriented, consider $\pi:X^{or}\to X$ the orientable double cover of $X$, with the associated involution $s:X^{or}\to X^{or}$. 
We identify
$$\mathcal{G}_{\pi^{-1}(K),\pi^{-1}(\mathcal{P})}=\mathcal{G}_{K,\mathcal{P}}\circ\pi.$$
Implication $(a)\Rightarrow(b)$: if $(a)$ holds then $\limsup_{n\to +\infty}\Vert h\circ \pi-h_n\circ\pi\Vert_{L^\infty(\pi^{-1}(K))}^{1/n}\leq e^{-t}$, so $h\circ \pi$ extends to some harmonic function $H$ in $\left\{\mathcal{G}_{\pi^{-1}(K),\pi^{-1}(\mathcal{P})}<t\right\}$. Since $H\circ s=H$, then $h$ extends harmonically to $\left\{\mathcal{G}_{K,\mathcal{P}}<t\right\}$ by $h\circ\pi=H$.\\
Implication $(b)\Rightarrow(a)$: if $(b)$ holds then $h\circ \pi$ is harmonic on $\left\{\mathcal{G}_{\pi^{-1}(K),\pi^{-1}(\mathcal{P})}<t\right\}$. So there exists some sequence $(H_n)_n$ with $H_n\in \HH_n(X^{or}\setminus\pi^{-1}(\mathcal{P}))$ such that
$$\limsup_{n\to +\infty}\Vert H_n-h\circ\pi\Vert_{L^\infty(\pi^{-1}(K))}^{1/n}\leq e^{-t}.$$
Define the sequence $h_n\in\HH_n(X\setminus\mathcal{P})$ by $h_n\circ\pi=\frac{1}{2}(H_n+H_n\circ s)$, then $\Vert h_n-h\Vert_{L^\infty(K)}\leq \Vert H_n-h\circ\pi\Vert_{L^\infty(\pi^{-1}(K))}$, so the conclusion $(a)$ holds.
\end{proof}

Let us now discuss the geometrical conditions under which $\mathcal{R}^{strong}$ holds, as described in remark \ref{rem_rstrong}. Assume $h$ is a subharmonic function defined in a neighbourhood $U$ of some point $p\in K$, such that $h$ is nonpositive on the interior of $K$. We suppose without loss of generality that $U=\varphi(\D_1)$ for some chart $\varphi:\D_1\to X$ such that $\varphi(0)=p$. 

Assume the first condition of remark \ref{rem_rstrong}. Let $c:(-1,1)\to X$ be analytic such that $c(0)=p$, $c((0,1))\subset \intK$. There exists some $r\in (0,1]$ such that $c((-r,r))\subset U$. Then $\varphi^{-1}\circ c_{|(-r,r)}$ is real-analytic from $(-r,r)$ to $\D_1$, so there exists some $s\in (0,r]$ and some holomorphic function $g:\D_s\to \D_1$ such that $g(t)=\varphi^{-1}\circ c(t)$ for any $t\in (-s,s)$. Now, $h\circ\varphi\circ g$ is subharmonic on $\D_s$ and nonpositive on $(0,s)$, so by \cite[Lem. 3.8.4]{R95} we have $h\circ g(0)=h(p)\leq 0$.

Assume now the second condition of remark \ref{rem_rstrong}. In particular there is a connected component $V$ of $X\setminus K$ such that $p\in\ov{V}$, and for some small enough $r_0\in (0,1]$ we have $\varphi^{-1}(U\cap V)\cap\partial\D_r\neq\emptyset$, for every $r\in (0,r_0)$. By \cite[Th. 5.4.2]{R95}, $\varphi^{-1}(U\cap V)$ is not thin at $0$, so $h(\varphi^{-1}(0))=h(p)\leq 0$.

In both cases, we obtain that every point of $K$ is non-thin in $\intK$, so the hypothesis $\mathcal{R}^{strong}$ holds.

\section{Computational approach} \label{sec_num}

We describe in this section the several steps required to compute a numerical  approximation of the harmonic basis introduced in section \ref{sec_basis}.  Following \cite{frauendiener2015computational}, we illustrate our method on three distinct surfaces of genus 2 described by their Fenchel-Nielsen coordinates (see also section \ref{sec_gluing}):

\begin{itemize}
    \item {\bf The surface with symmetry group} $D_6 \times \mathbb{Z}_2$ for which
$$(l_1, t_1; l_2, t_2; l_3, t_3) = (2 \acosh(2), 0 ; 2 \acosh(2), 0; 2 \acosh(2), 0)$$
which coordinates correspond to four regular orthogonal hexagons with no twists.
    \item {\bf The Bolza surface} which can either be parametrized by Fenchel-Nielsen  coordinates with one single nonzero twist
$$(l_1, t_1; l_2, t_2; l_3, t_3) = (2 \acosh(3 + 2 \sqrt2), \frac12 ; 2 \acosh(1 + \sqrt2), 0; 2 \acosh(1 + \sqrt2), 0)$$
or equivalently by the uniform representation
$$(l_1, t_1; l_2, t_2; l_3, t_3) = (l,t;l,t;l,t)$$
where $l = \acosh(1 + \sqrt2)$ and $ t = \frac1l \acosh\left(\frac{2\sqrt{3 + \sqrt2}}{7}\right)$.

\item {\bf The Gutzwiller octagon} of coordinates $(l_1, t_1; l_2, t_2; l_3, t_3)$ given by 
$$  \left(
2\acosh\left(\frac{\sqrt2 + 1}{\sqrt2}\right), \frac12; 
4 \acosh\left(\frac{\sqrt2 + 1}{\sqrt2}\right), \frac14; 
2\acosh\left(\frac{\sqrt2 + 1}{\sqrt2}\right), \frac12\right).$$

\end{itemize} 
Every triple of alternate edge lengths $(l_1, l_2, l_3)$ determines a unique orthogonal hexagon in Poincare's disk up to automorphisms of the disk. The practical construction of these hexagons for three alternate lengths is described in \cite{hubbard2016teichmuller} chapter 3. We plot  the tilings of the hyperbolic disk obtained by the four orthogonal hexagons of every surface in figure \ref{fig:tilings}.

\begin{figure}[!htb]  
    \centering
    \begin{tabular}[t]{ccc}
        \includegraphics[width=0.3\textwidth]{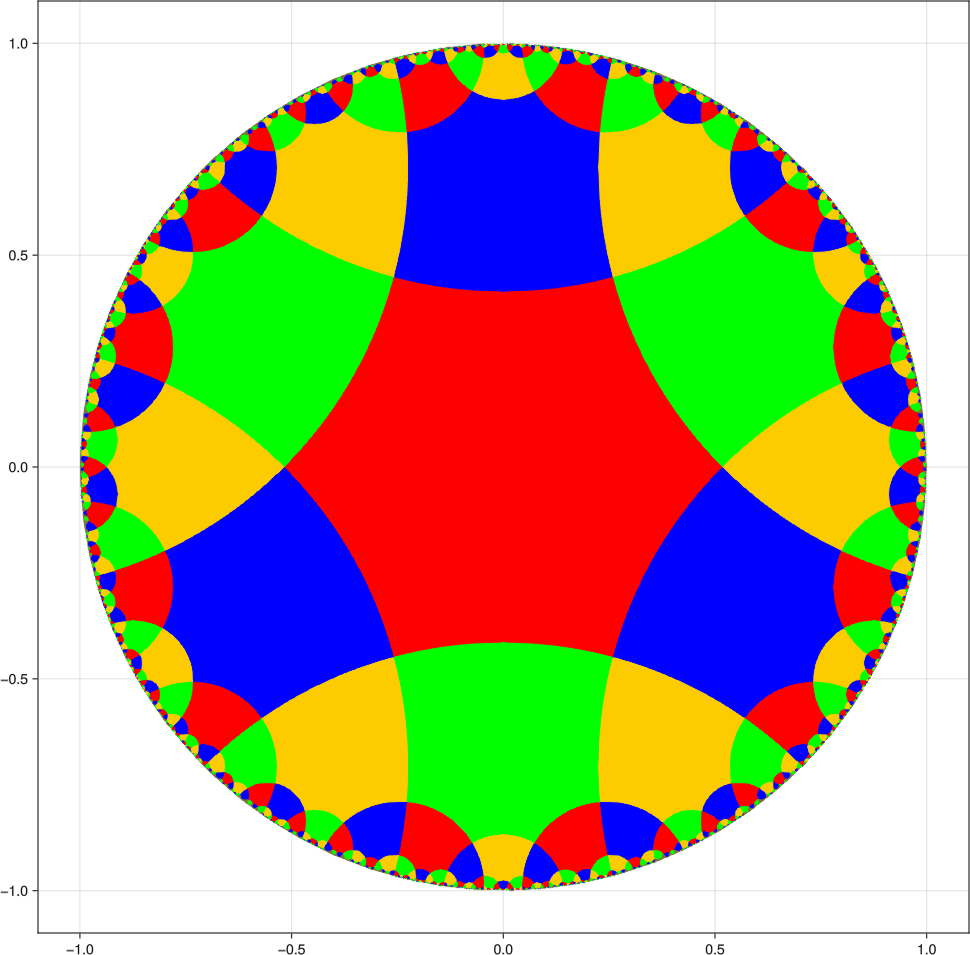}
        &
        \includegraphics[width=0.3\textwidth]{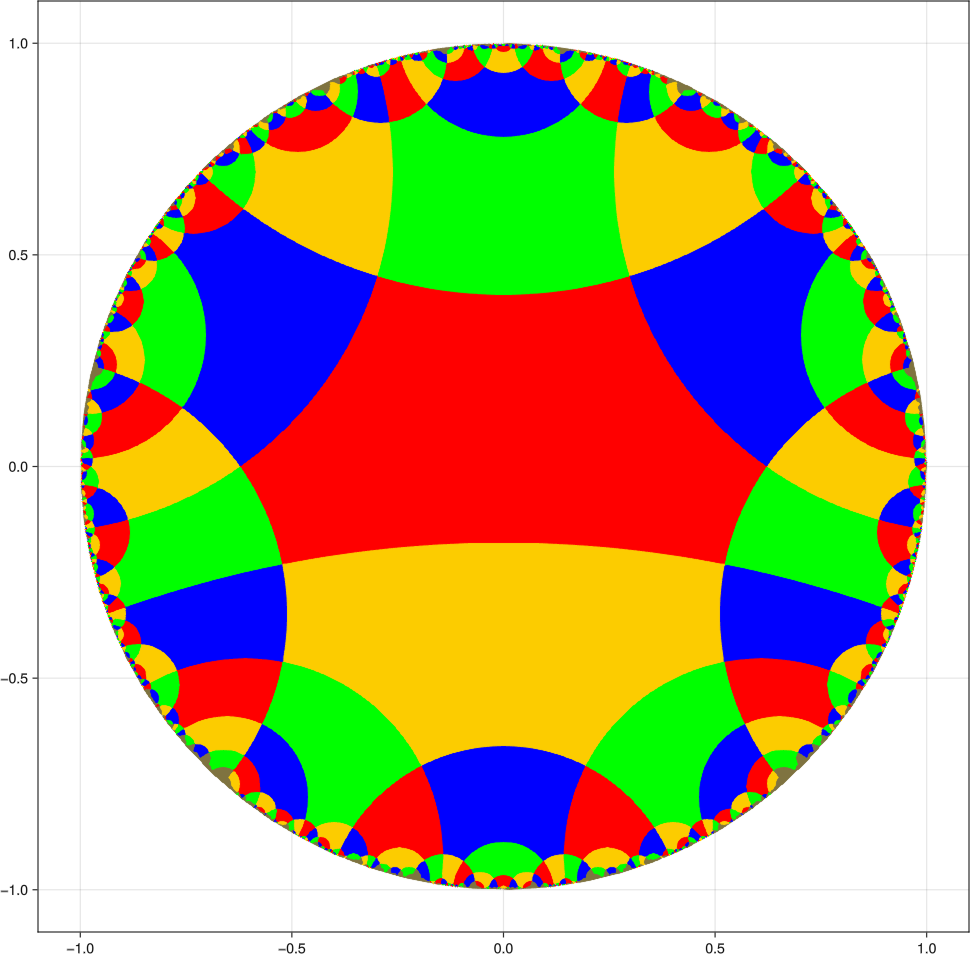}
        &
        \includegraphics[width=0.3\textwidth]{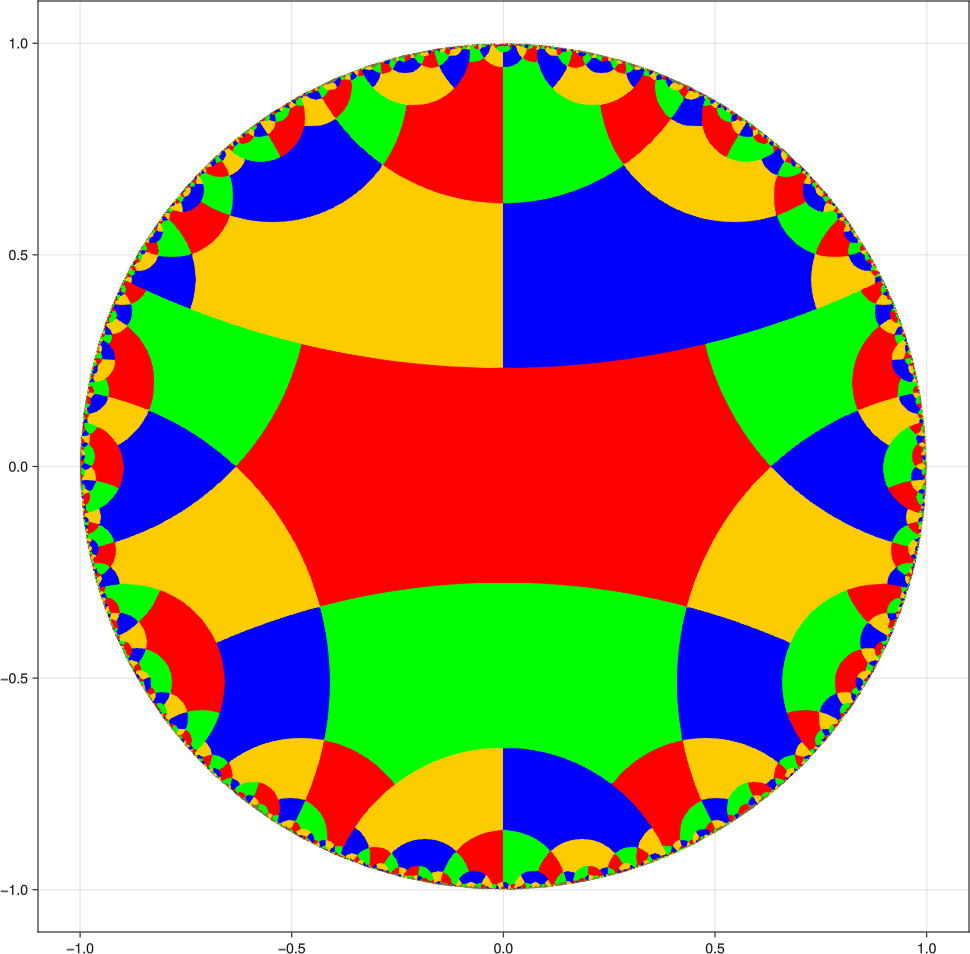}
    \end{tabular}
    \caption{$D_6 \times \mathbb{Z}_2$, Bolza and Gutzwiller hexagonal tilings (left to right).}
    \label{fig:tilings}
\end{figure}

\subsection{Approximation of 1-forms} \label{sec_comp}

Our construction of a basis of harmonic functions strongly relies on the approximation of a basis of one 1-forms. We already described in the introduction the least square approximation procedure that we implemented to obtain a polynomial approximation of this basis. We prove in section \ref{sec_leastsquareconv} the spectral convergence in infinity norm of the approximation with respect to the degree of the polynomials.

Assuming that the base point is not a Weierstrass point (see remark 4 below theorem \ref{the_conv} for a discussion of the detection and correction when the center is a Weierstrass point), we compute for some large $N$ the optimal polynomials $Q^N_l=(Q^{N}_{l|p})_{p=1,\hdots,m}\in \C_{N}[z]^m$ defined as
\[\text{argmin}\left\{E^N(P),\ P\in \C_N[z]^m:\forall j=1,\hdots,g,\ P_{|1}^{(j-1)}(0)=\delta_{j,l}\right\}.\]
These act as an approximate basis of the $1$-forms $\Om^1(X)$.  To illustrate this convergence, we list in table \ref{tab_oneform} the a posteriori errors obtained with several degrees $N$ of approximation. The number of sampling points $S(N)$ to enforce the periodicity conditions on every circular edge is fixed to be $3N$. To avoid round-off errors, we implemented the method using arbitrary precision with a number of $256$ bits which correspond to a maximal floating point precision of $154$ digits. The computation in the required precision of the polynomial approximations required from several seconds to a few minutes (for $N=10$ to $N=200$). Observe that these computations have to be considered as pre-treatement: to evaluate our full basis of harmonic functions we only need to perform these computations once.
The error values given in table  \ref{tab_oneform} correspond to the maximal error on periodicity conditions 
\[P_{|p}(z)=g_{p,i}'(z)P_{|q}(g_{p,i}(z))\]
numerically evaluated on $6N$ random sampling points $z$ on every circular edge. As expected, we recover a spectral precision with respect to $N$ for all three test cases.
%

\begin{center}
    \begin{table}[ht]
        \begin{center}
            \begin{tabular}{|l|c|c|c|c|c|c|} 
                \hline
                & $N=10$ & $N=20$ & $N=50$ &  $N=100$ &  $N=200$ \\ 
                \hline
                $D_6 \times \mathbb{Z}_2$ & 6.93e-04 & 1.58e-06 & 6.89e-16 & 
                1.20e-31 &  6.29e-59\\ 
                \hline
                Bolza surface & 6.15e-03 & 1.52e-05 & 9.24e-14 &
                6.06e-27 &  1.79e-52\\ 
                \hline
                Gutzwiller octagon & 1.95e-02 & 2.61e-05 & 1.04e-13&
                7.76e-27 & 1.15e-51\\ 
                \hline
            \end{tabular}
        \end{center}
        \vspace{-0.3cm}
        \caption{Spectral convergence of numerical a posteriori periodicity errors}
        \label{tab_oneform} 
    \end{table}
\end{center}

\subsection{Approximation of the period matrix {and the Abel-Jacobi map}}\label{sec_comp_period}
The period matrix of a surface is a fundamental data required in our construction to build meromorphic functions of prescribed orders. We explained in previous section how to compute an approximation of a basis of holomorphic  $1$-forms. The computation of a period matrix reduces to the evaluation of path integrals of these forms along the canonical basis $(a_j, b_j)_{j=1,...,g}$ introduced in \ref{sec_abel}.  Starting from the polynomials $Q^N_l=(Q^{N}_{l|p})_{p=1,\hdots,m}\in \C_{N}[z]^m$ defined previously, we first let

\begin{align*}
A^N=\left(\left(\int_{a_k}Q^{N}_{l}\right)_{1\leq k,l\leq g}\right),\end{align*}
where the integral of an element $Q=(Q_{|1},\hdots,Q_{|m})\in\C_{N}[z]^m$ along a smooth (say analytic by part) loop $c:[0,1]\to X$ is defined as follows: $[0,1]$ may be partitioned with intervals
\[0=t_0<t_1<t_2<\hdots<t_r=1,\]
such that for any $i\in \{0,\hdots,r-1\}$, $t\in [t_i,t_{i+1}]\mapsto c(t)$ lies in a fixed polygon $H_{p_i}$, and $c(t_{i})$ is in the boundary of $H_{p_i}$ for $0<i<r$. We then let
\[\int_{c}Q=\sum_{i=0}^{r-1}\int_{t_i}^{t_{i+1}}Q_{|p_i}(c(t))c'(t)dt.\]
While this definition may seem standard, we remind that in this case the functions $(Q_{|p})_{p=1,\hdots,m}$ do not verify exactly the periodicity condition.

The matrix $A^N$ is invertible at least for a large enough $N$, and we make the change of variable
\[\tilde{Q}^{N}_k=\sum_{l=1}^{g}((A^{N})^{-1})_{l,k}Q^{N}_l\]
such that $\int_{a_k}\tilde{Q}_l=\delta_{k,l}$. This way $\tilde{Q}^{N}_l$ is an approximation of the $l$-th canonical $1$-form $\om_l$ verifying condition \eqref{eq_canonical_1form}.
We then let
\begin{align*}
\tau^N&=\left(\left(\frac{1}{2}\int_{b_k}\tilde{Q}^{N}_l+\frac{1}{2}\int_{b_l}\tilde{Q}^{N}_k\right)_{1\leq k,l\leq g}\right)
\end{align*}
the approximate period matrix. Note that we force its symmetry by construction: we could define $\tau^N_{k,l}=\int_{b_k}\tilde{Q}^{N}_l$, and since $\tilde{Q}^{N}_l$ is close to $\om^l$, $\tau^N$ would be close to being symmetric. In our simulations, we used the very efficient  \emph{Arb} C library \cite{Johansson2017arb, arbtheta} which has been now incorporated in the \emph{Flint} library. The functions available in the library may provide results at a precision depending on the precision of the input. In particular, the call of the multidimensional theta function requires the input period matrix to be symmetric, which motivates this forced symmetry.

We evaluate the numerical precision of our approximations of period matrices on the well known case of Bolza surface. O. Bolza computed analytically the associated period matrix in  \cite{bolza1887binary}:

$$\tau = \begin{pmatrix}
\frac{-1+i\sqrt2}{2} & \frac12 \\
\frac12 & \frac{-1+i\sqrt2}{2} 
\end{pmatrix}.
$$
Moreover, the Siegel reduction of $\tau$ is given by 
$$\tau_{Siegel} = (A \tau + B) (C \tau + D) ^{-1},$$
where
$$
\begin{pmatrix}
   A & B\\  
   C & D
\end{pmatrix}
=
\begin{pmatrix}
   -1  &0   &-1  &0 \\ 
    0  &1   &0  &0  \\ 
    1  &0   &0  &0  \\ 
    0  &0   &0  &1
 \end{pmatrix}.
$$
Using previous formula we obtain
$$\tau_{Siegel} = - \frac13 \begin{pmatrix}
     1 &  1\\  
     1 &  1
 \end{pmatrix}
 + i 
 \begin{pmatrix}
    \frac{2\sqrt2}3   &  -\frac{\sqrt2}{3}\\
    -\frac{\sqrt2}{3} &   \frac{2\sqrt2}3
 \end{pmatrix}.
 $$
 We reproduce below in table \ref{tab_periodmatrix} the infinity norm error between $\tau_{Siegel}$ and the approximations $\tau_{Siegel}^N$. As expected, we recover the same order of convergence as the one associated to the periodicity errors of previous section. We observed the same qualitative behavior when approximating the period matrices of the  $D_6 \times \mathbb{Z}_2$ surface and Gutzwiller octagon.

\begin{center}
    \begin{table}[ht]
        \begin{center}
            \begin{tabular}{|l|c|c|c|c|c|c|} 
                \hline
                & $N=10$ & $N=20$ & $N=50$ &  $N=100$ & $N=200$ \\ 
                \hline
                $||\tau_{Siegel} - \tau_{Siegel}^N||_\infty$ & 8.95e-05 &5.69e-08 & 4.51e-17&
                6.33e-30 & 1.53e-55\\ 
                \hline
            \end{tabular}
        \end{center}
        \vspace{-0.3cm}
        \caption{Spectral convergence of the approximation of Bolza period matrix.}
        \label{tab_periodmatrix} 
    \end{table}
\end{center}

Finally, the approximation of the Abel-Jacobi map $u=(u_1,\hdots,u_g)$ is defined from taking a primitive of $\tilde{Q}^{N}$ on each polygon and adjusting it with a constant: we fix a path $c:[0,1]\to X$ that meet every polygon exactly once and define on each polygon $H_p$ the vector-valued polynomial $U^{N}_{|p}=(U^{N}_{1|p},\hdots,U^{N}_{g|p})\in\C_{N+1}[z]^g$ such that for every $l=1,\hdots,g$
\[(U^{N}_{l|p})'=\tilde{Q}^{N}_{l|p}\text{ and }t\mapsto U^{N}_{l}(c(t))\text{ is continuous.}\]
While this choice of path may seem arbitrary, we remind that any different choice would only differ by an element of $\Z^g+\tau \Z^g$ on each polygon, which will be transparent after a composition of $\Theta$.

Finally, we compute the approximate Riemann constant
\[\K^N_j=\frac{1}{2}\tau^N_{jj}-\sum_{k=1}^{g}\int_{a_k}U_j^N(z)(U_k^N)'(z)dz,\]
where the last term may be computed exactly since $U_j^N(z)(U_k^N)'(z)\in \C_{2N+1}[z]^m$.

\subsection{Approximation of meromorphic function of logarithmic order and of order $1$}

We described in sections \ref{subsec_log} and \ref{subsec_p1}, that the evaluation of functions $\log|\widehat{\sigma}_{v,w}|$ and $\widehat{\wp}_1(z,w)$ requires the knowledge of the period matrix $\tau$ of the surface (to evaluate the $\Theta$ function) and a basis of $1$-forms. We already discussed  the accurate approximation of these data in the two previous sections.

Several libraries provide implementations of the vector valued Riemann theta function for a given $\tau$ matrix in the Siegel upper half-space $\mathbb{H}_g$ (the space of symmetric complex matrices with positive definite imaginary part). 

All $\Theta$ function evaluations of this section were performed with a precision of $512$ using formula \ref{eq_logsigma} and \ref{eq_orderone}. We a posteriori evaluated the numerical precision computing the periodicity errors of our approximations of the $\log|\widehat{\sigma}|$ and $\widehat{\wp}_1$ function. Using the same criteria as in section \ref{sec_comp}, we observed  periodicity errors of the same order with respect to $N$ as the one given in table \ref{tab_oneform}. We plot in figures \ref{fig_logD6}, \ref{fig_logBolza} and \ref{fig_logGutzwiller} the graphs of the  $\log|\widehat{\sigma}|$ function for every three surfaces. The left four graphs represent the restriction of the function to the Fenchel hexagons (with colored edges periodicity). The right plots represent the full graph of the associated function on Poincare disk. 

As an example, we also provide plots of the real and imaginary part of $\widehat{\wp}_1$ of the Bolza surface in the appendix (see figures \ref{fig_BolzaOrder1R} and \ref{fig_BolzaOrder1I}).

\begin{figure}[!htb]
    \centering
    \begin{tabular}[t]{ccc}
        \begin{tabular}{c}
            \hfill
            \begin{subfigure}[t]{0.22\textwidth}
                \centering
                \includegraphics[width=\linewidth]{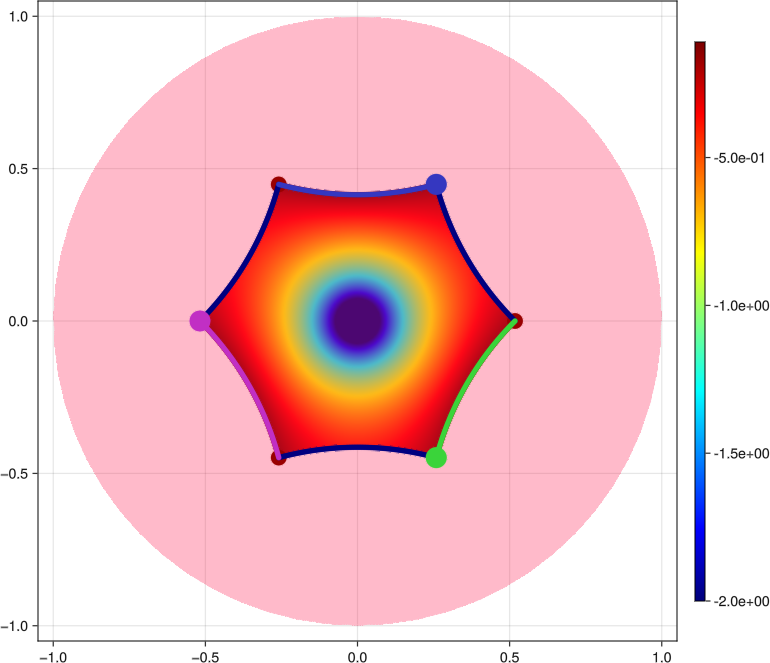}
            \end{subfigure}\\
            \hfill
            \begin{subfigure}[t]{0.22\textwidth}
                \centering
                \includegraphics[width=\linewidth]{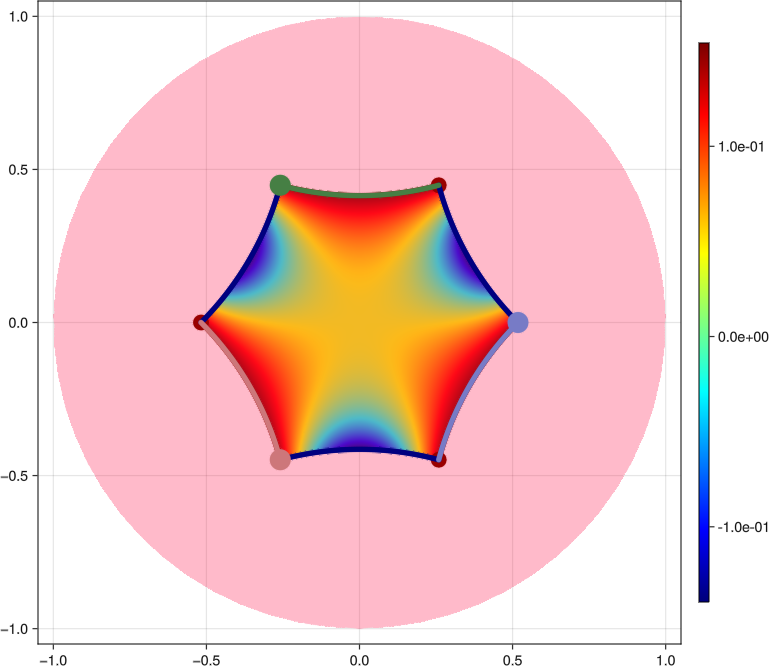}
            \end{subfigure}
        \end{tabular}
        &
        \begin{tabular}{c}
            \hspace{-1cm}
            \begin{subfigure}[t]{0.22\textwidth}
                \centering
                \includegraphics[width=\linewidth]{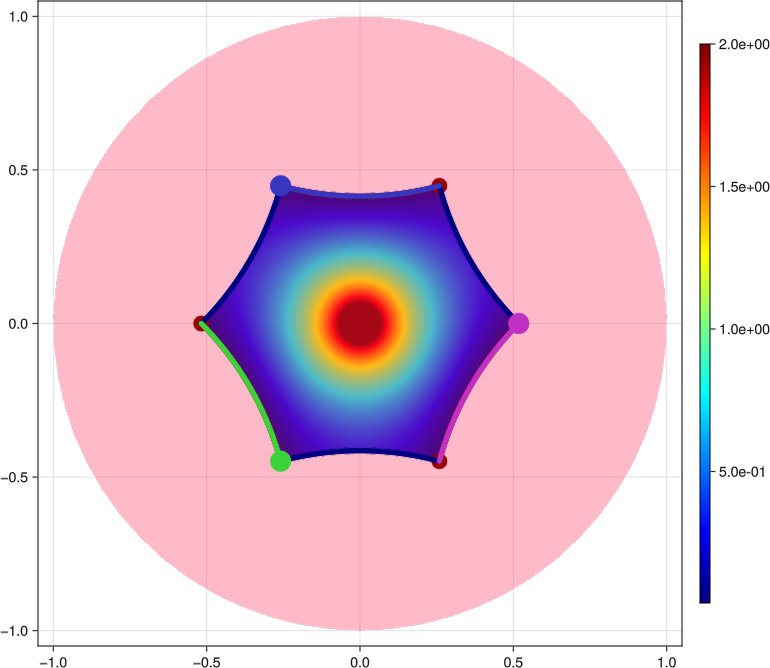}
            \end{subfigure}\\
            \hspace{-1cm}
            \begin{subfigure}[t]{0.22\textwidth}
                \centering
                \includegraphics[width=\linewidth]{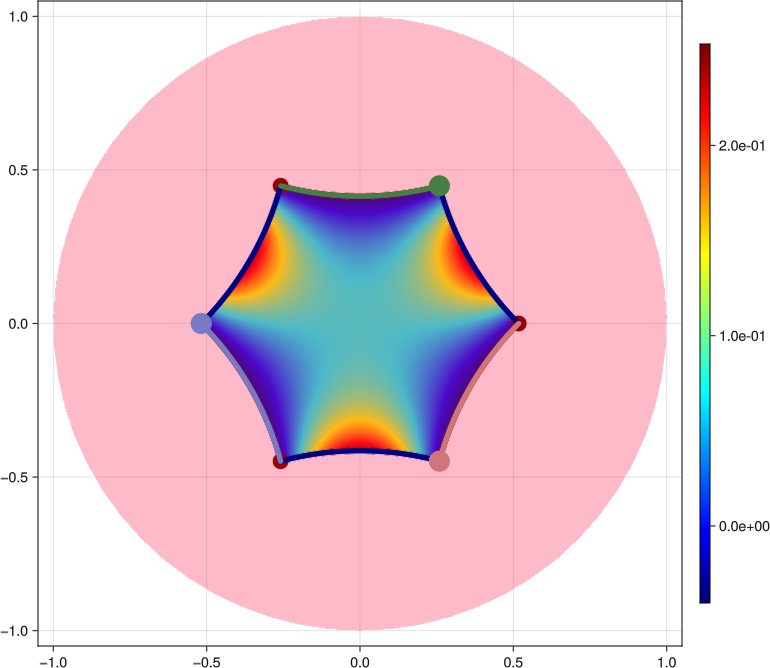}
            \end{subfigure}
        \end{tabular}
        &
        \begin{tabular}{c}
            \hspace{-0.75cm}
            \begin{subfigure}[t]{0.42\textwidth}
                \centering
                \includegraphics[width=\linewidth]{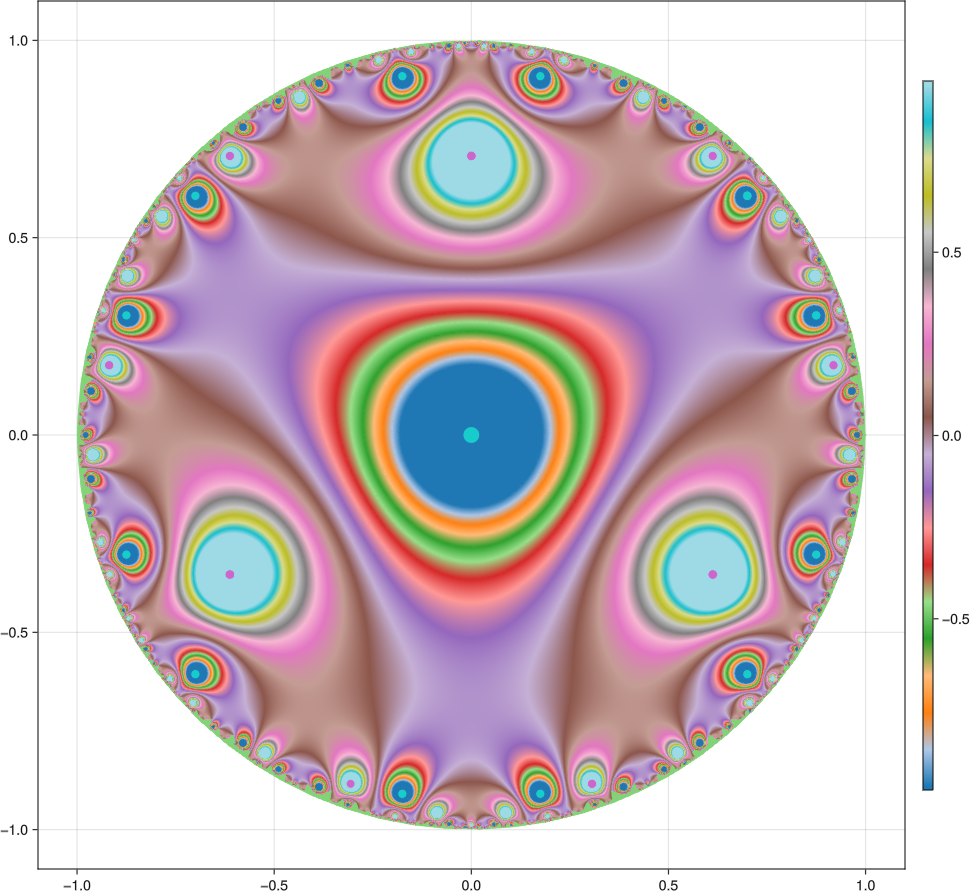}
            \end{subfigure}
        \end{tabular}
    \end{tabular}
    \caption{Logarithmic poles on  $D_6 \times \mathbb{Z}_2$'s surface.}
    \label{fig_logD6}
\end{figure}

\begin{figure}[!htb]
    \centering
    \begin{tabular}[t]{ccc}
        \begin{tabular}{c}
            \hfill
            \begin{subfigure}[t]{0.22\textwidth}
                \centering
                \includegraphics[width=\linewidth]{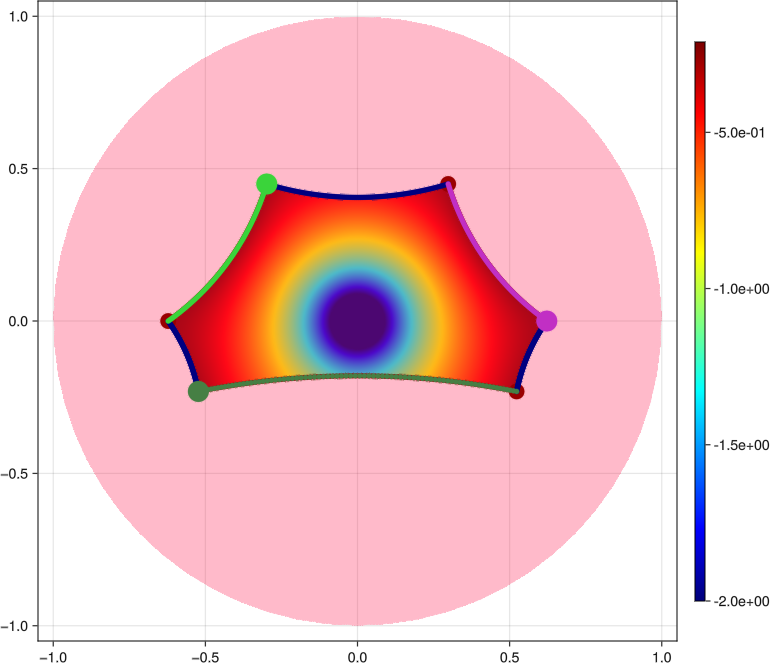}
            \end{subfigure}\\
            \hfill
            \begin{subfigure}[t]{0.22\textwidth}
                \centering
                \includegraphics[width=\linewidth]{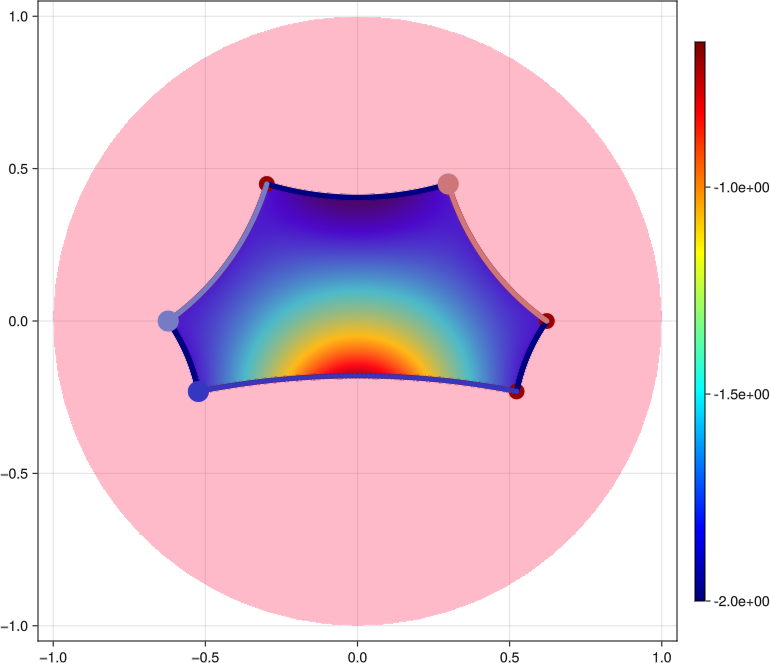}
            \end{subfigure}
        \end{tabular}
        &
        \begin{tabular}{c}
            \hspace{-1cm}
            \begin{subfigure}[t]{0.22\textwidth}
                \centering
                \includegraphics[width=\linewidth]{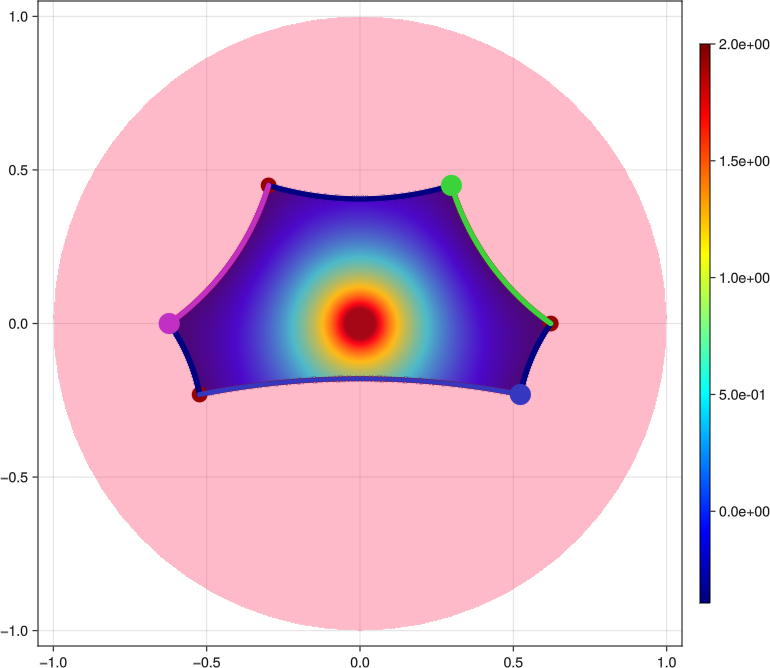}
            \end{subfigure}\\
            \hspace{-1cm}
            \begin{subfigure}[t]{0.22\textwidth}
                \centering
                \includegraphics[width=\linewidth]{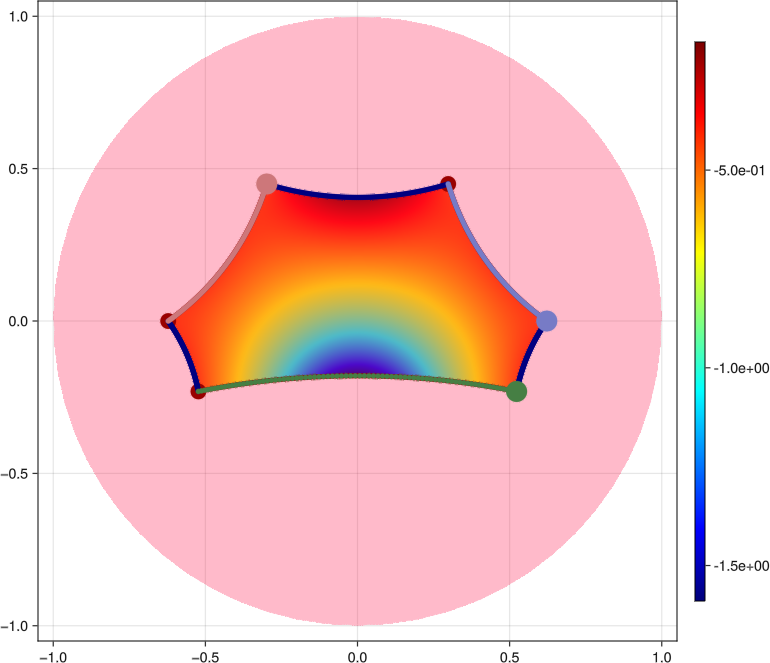}
            \end{subfigure}
        \end{tabular}
        &
        \begin{tabular}{c}
            \hspace{-0.75cm}
            \begin{subfigure}[t]{0.42\textwidth}
                \centering
                \includegraphics[width=\linewidth]{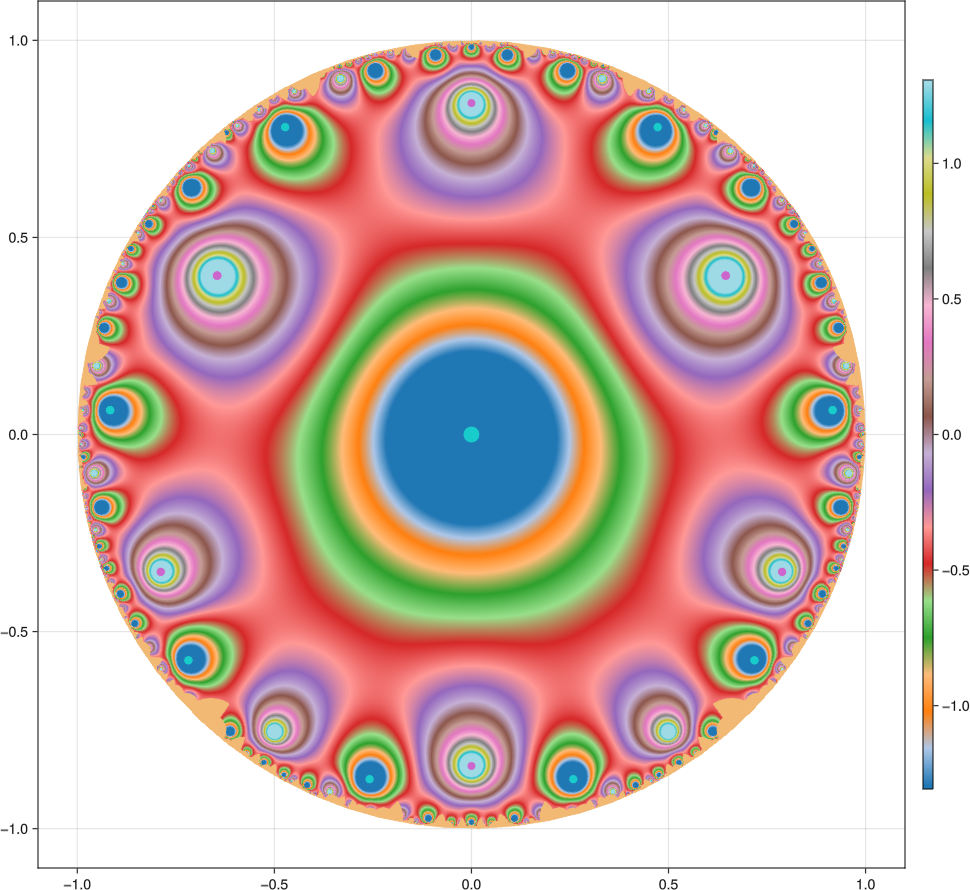}
            \end{subfigure}
        \end{tabular}
    \end{tabular}
    \caption{Logarithmic poles on Bolza's surface.}
    \label{fig_logBolza}
\end{figure}

\begin{figure}[!htb]
    \centering
    \begin{tabular}[t]{ccc}
        \begin{tabular}{c}
            \hfill
            \begin{subfigure}[t]{0.22\textwidth}
                \centering
                \includegraphics[width=\linewidth]{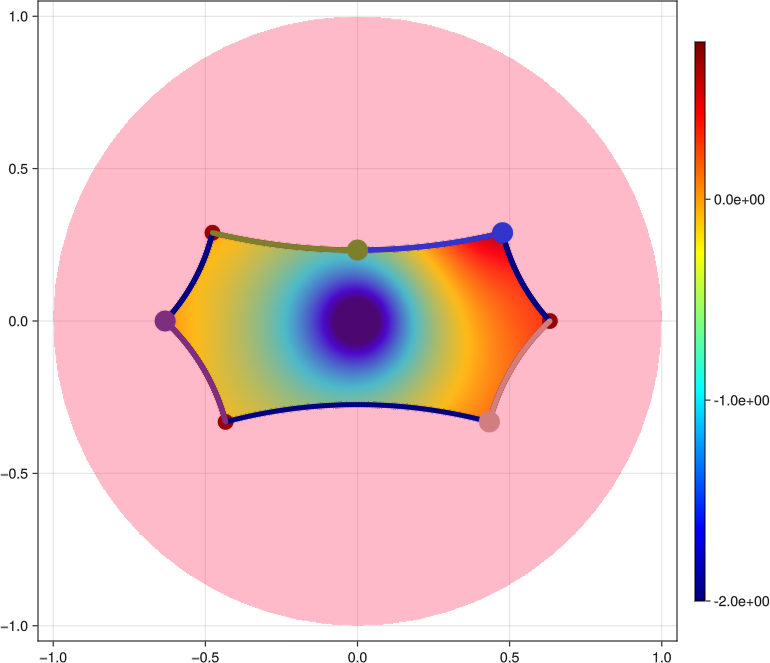}
            \end{subfigure}\\
            \hfill
            \begin{subfigure}[t]{0.22\textwidth}
                \centering
                \includegraphics[width=\linewidth]{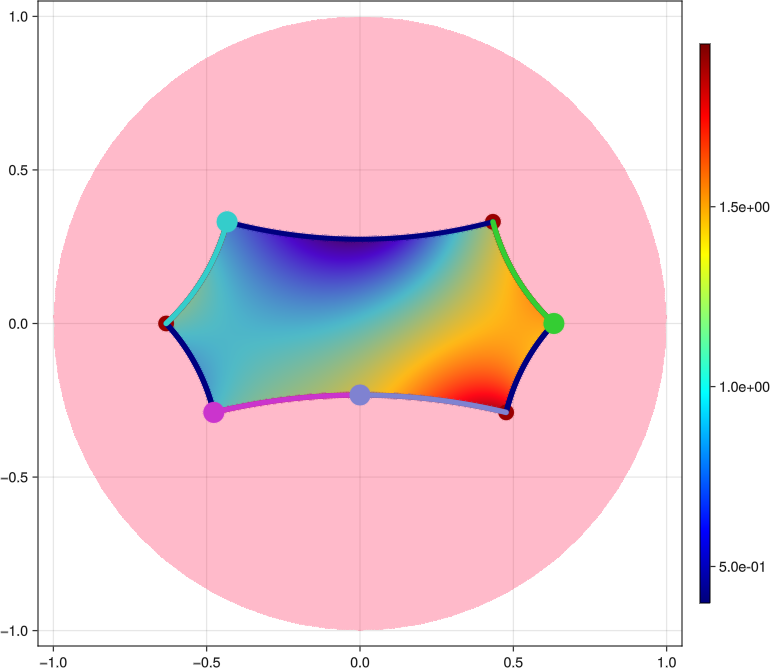}
            \end{subfigure}
        \end{tabular}
        &
        \begin{tabular}{c}
            \hspace{-1cm}
            \begin{subfigure}[t]{0.22\textwidth}
                \centering
                \includegraphics[width=\linewidth]{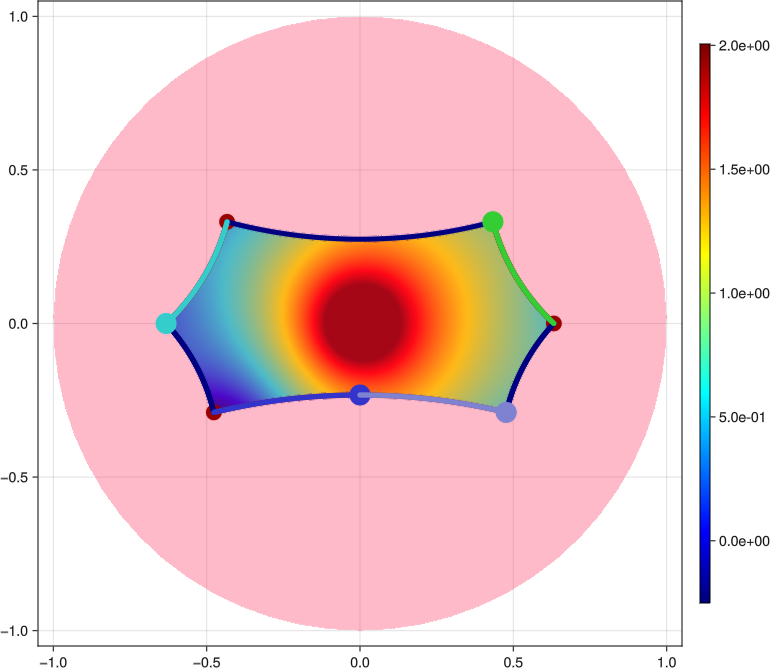}
            \end{subfigure}\\
            \hspace{-1cm}
            \begin{subfigure}[t]{0.22\textwidth}
                \centering
                \includegraphics[width=\linewidth]{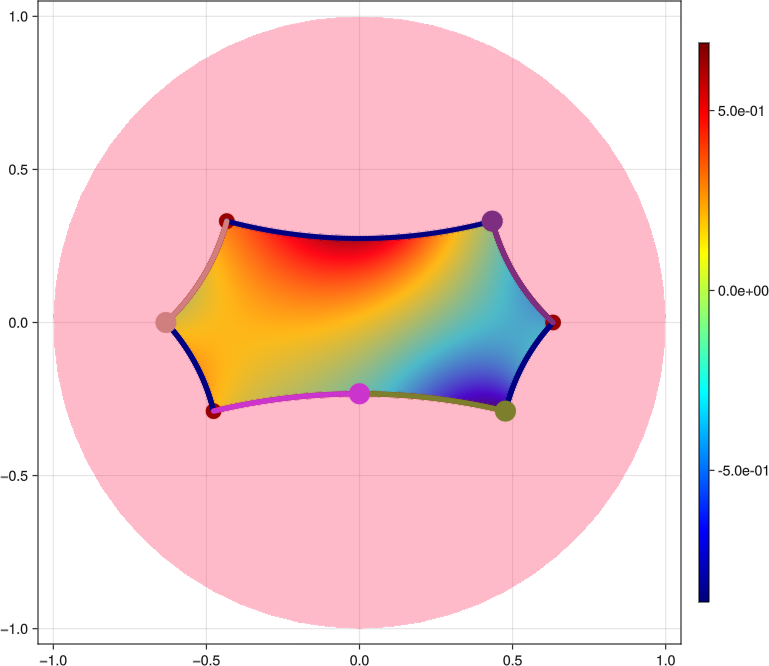}
            \end{subfigure}
        \end{tabular}
        &
        \begin{tabular}{c}
            \hspace{-0.75cm}
            \begin{subfigure}[t]{0.42\textwidth}
                \centering
                \includegraphics[width=\linewidth]{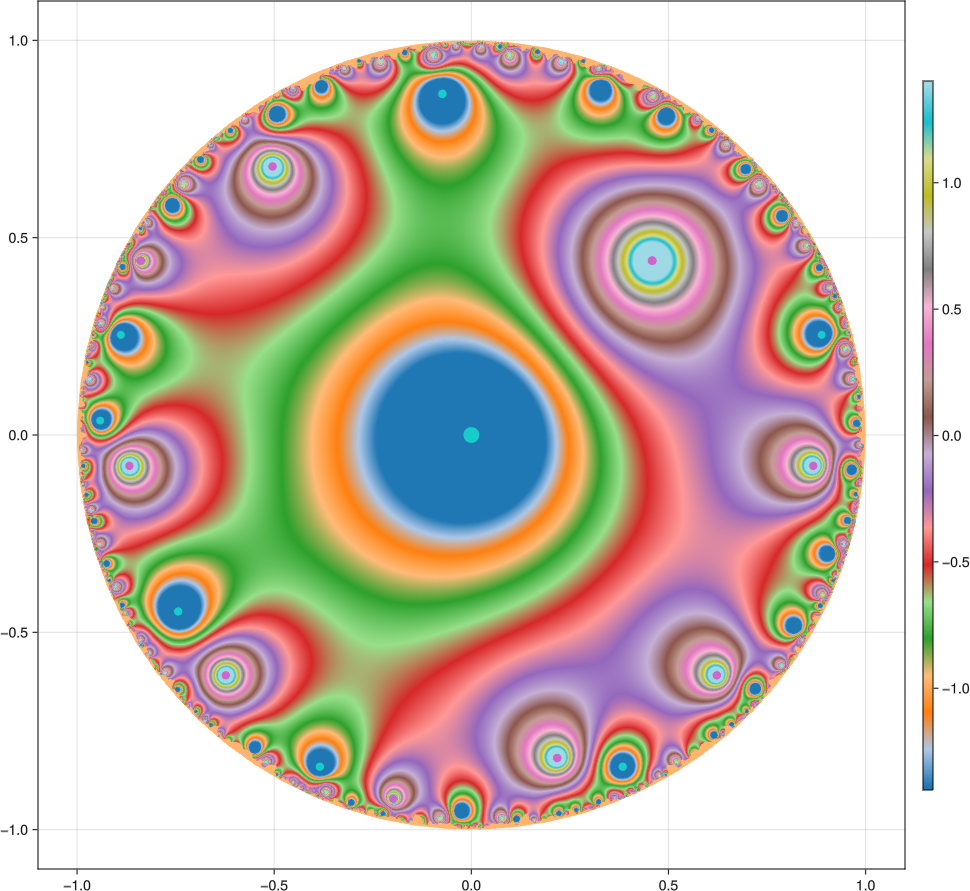}
            \end{subfigure}
        \end{tabular}
    \end{tabular}
    \caption{Logarithmic poles on Gutzwiller's surface.}
    \label{fig_logGutzwiller}
\end{figure}

\subsection{Approximation of meromorphic function of higher order by complex differentiation}

The definitions of harmonic functions of order $2$ to $2g+1$ introduced in sections \ref{subsec_ptilde} and \ref{subsec_pbar} involve a new difficulty from a computational point of view: The derivatives of a rational function of $\Theta$ (composed with anti-derivatives of pre-computed  $1$-forms) have to be evaluated.
For this task, we implemented a finite difference procedure well adapted to analytic functions. In this context, finite differences avoid round off errors coming from high order algebraic derivatives  minimizing the number of function evaluations. We recall a simple process to approximate the derivative of order $m \in \mathbb{N}^*$ of  a function $f$ of a single complex variable $z$. The generalization of the method to complex vector variable is straightforward.

Let $\varepsilon >0$ and $\zeta = e^{\frac{2 i \pi}m}$. Assume $f(z +h) = \sum_{k \in \mathbb{N}} a_k h^k$ for a small complex perturbation $h \in \mathbb{C}$. For any fixed $1 \leq p < m$, we have using algebraic simplifications
$$\sum_{n = 0}^{m-1} \zeta^{-pn}f(z +h_n) = m \sum_{k \in \mathbb{N}, \, k = p (\text{mod}\, m)} a_k \varepsilon^k$$ 
where $h_n = \varepsilon \zeta^n$. Thus, previous equality gives an approximation of $a_p$ with an error term of order $\varepsilon^m$. Consequently, $m$ evaluations of the function $f$ lead to an approximation of all its derivatives (deduced from the $a_k$) up to order $m-1$. 
In order to avoid round off errors in the finite difference process, we systematically call the  $\emph{Flint}$ library (to evaluate Riemann $\Theta$ function) with a (doubled) accuracy of $512$ bits.

Finally,  meromorphic functions of order greater than $2g+2$ can be easily obtained as powers and product of previous functions as detailed in section \ref{subsec_ptilde}. Similarly to the approximation of the $\log|\widehat{\sigma}_{v,w}|$ and $\widehat{\wp}_1(z,w)$ function of previous section, we observe no loss of accuracy with respect to the periodicity. We obtained once again the same order of convergence as the one obtained in the approximation of  $1$-forms. As an illustration, we plot in the appendix the real and imaginary part of a meromorphic function of the surface with symmetry group $D_6 \times \mathbb{Z}_2$ of order $3$ (see figures \ref{fig_Weierstrass3D6R} and \ref{fig_Weierstrass3D6I}) and  of order $6$ (see figures \ref{fig_Weierstrass6D6R} and \ref{fig_Weierstrass6D6I}).

\subsection{Algorithmic complexity and higher genus} The previous steps describe how to construct a complete basis for approximating harmonic functions on surfaces of arbitrary genus. As emphasized in the introduction, the most computationally demanding part of the method is the least squares approximation of $g$ independent $1$-forms: For a fixed genus $g$, we need to solve $g$ least square problems which enforce the $2g-2$ periodicity conditions. In contrast, the computation of the period matrix and the associated Riemann constant relies solely on fast algebraic evaluations. Once computed, these 1-forms serve to parametrize basis of harmonic functions associated with any finite number of holes on the surface.

\section{Numerical harmonic extension in genus $2$}\label{sec_exp}

In the same spirit as the numerical experiments of \cite{T18, KOO23} we illustrate in this section the use of our basis of harmonic functions to approximate harmonic extension by the method of particular solutions. Namely, we consider the Gutzwiller surface $X$  to solve for a given $g$ the Laplace equation
\begin{equation} \label{eq_harmext}
    u: X\setminus (B_1\cup B_2)\to\R\text{ such that }\begin{cases}\Delta u=0&\text{ in } X \setminus ( B_1 \cup B_2) \\ u=g & \text{ in }\partial B_1 \cup \partial B_2
    \end{cases},
\end{equation}
where $B_1$ and $B_2$ are two closed geodesic balls.

Theorem \ref{mr_finitepoles} is the key ingredient of the method of particular solution to solve \eqref{eq_harmext} for a surface of genus greater than 1. For a finite set of harmonic functions, the idea is to identify the linear combination which minimizes the boundary error in the least square sense. To simplify the notations, we denote by $(\phi_i)_{1\leq u \leq M}$ a finite set of harmonic functions in  $X \setminus ( B_1 \cup B_2)$ provided by theorem \ref{mr_finitepoles} which is dense in the space of harmonic functions when $M$ tends to infinity.
We thus look for the unknown coefficients $(v_i)_{1\leq u \leq M}$ such that the function 
\begin{equation}
    \label{e:truncatedSeriesSolution}
    u(z) = \sum_{i = 1}^M v_i \phi_i(z)
\end{equation}
is close to satisfy the boundary condition imposed by $g$.

To identify the optimal coefficients, we sample uniformly the boundary $\partial B_1 \cup \partial B_2$  with respect to arclength and denote the collection of all sampled points by $(p_\ell)_{ \ell=1,\hdots, S}$. We define the matrix $B \in \mathbb R^{S\times M}$ to be
\begin{align*}
    B_{\ell,i} &= \phi_i(p_l). 
\end{align*}
The least-squares solution is  found by solving the normal equations \begin{equation}
    \label{e:normalEq}
    B^t B v = B^t b
\end{equation}
where the vector $b = (g(p_\ell))_{ \ell \in S}$ is the evaluation of the boundary condition at the sampling points.

We implemented the numerical method in Julia using arbitrary precision provided by the packages \emph{GenericLinearAlgebra.jl} and \emph{ArbNumerics.jl} (a wrapper of the \emph{Arb} C library we already introduced). All computational experiments were performed with a precision of $512$ bits which corresponds to a machine epsilon approximately equal to $10^{-154}$. 

In our example, we choose the two disks $B_1$ and $B_2$ to be centered circular holes of radii $r=0.1$ of the first and third hexagons representing Gutzwiller surface. We systematically fixed the number of sampling points per circular edge to be equal to $3M$. We define the boundary condition as $g(\theta) = \sin(3\theta)$ and  $g(\theta) = \sin(7\theta)$ respectively on $\partial B_1$ and $\partial B_2$, where $\theta$ are the polar angles with respect to the centers of the holes. By the maximum principle, the accuracy of the solution can be computed by looking at the error on the boundary that is $\sup_{x\in \partial B_1 \cup \partial B_2} |u(x) - f(x)|$.

We plot in figure \ref{fig_spec} the graph of the (base 10) logarithm of the boundary  error a posteriori evaluated on $3S$ random points with respect to the number $M$ of basis elements. As expected we recover a spectral convergence as the results obtained on similar test cases in \cite{T18, KOO23}. Finally, we plot in figure \ref{fig_harmonic} the full graph of the harmonic extension of $g$ on Poincare disk. To visualize the small variations of the function $u$, we apply a threshold to $u$ and plot the function $\max(\min(u,t), -t)$ where $t = 0.1$.


\begin{figure}[!htb]
    \centering
    \begin{tabular}[t]{ccc}
        \begin{tabular}{c}
            \hfill
            \begin{subfigure}[t]{0.22\textwidth}
                \centering
\includegraphics[width=\linewidth]{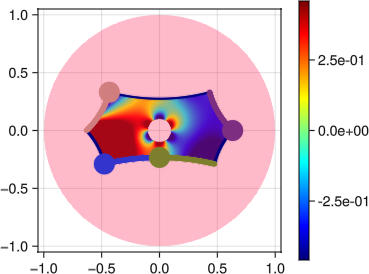}
            \end{subfigure}\\
            \hfill
            \begin{subfigure}[t]{0.22\textwidth}
                \centering
\includegraphics[width=\linewidth]{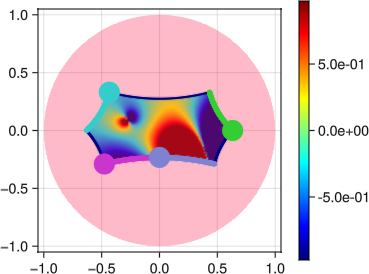}

            \end{subfigure}
        \end{tabular}
        &
        \begin{tabular}{c}
            \hspace{-1cm}
            \begin{subfigure}[t]{0.22\textwidth}
                \centering
\includegraphics[width=\linewidth]{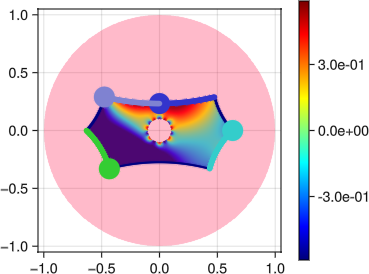}

            \end{subfigure}\\
            \hspace{-1cm}
            \begin{subfigure}[t]{0.22\textwidth}
                \centering
\includegraphics[width=\linewidth]{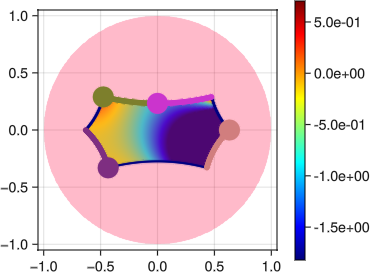}
            \end{subfigure}
        \end{tabular}
        &
        \begin{tabular}{c}
            \hspace{-0.75cm}
            \begin{subfigure}[t]{0.42\textwidth}
                \centering
\includegraphics[width=\linewidth]{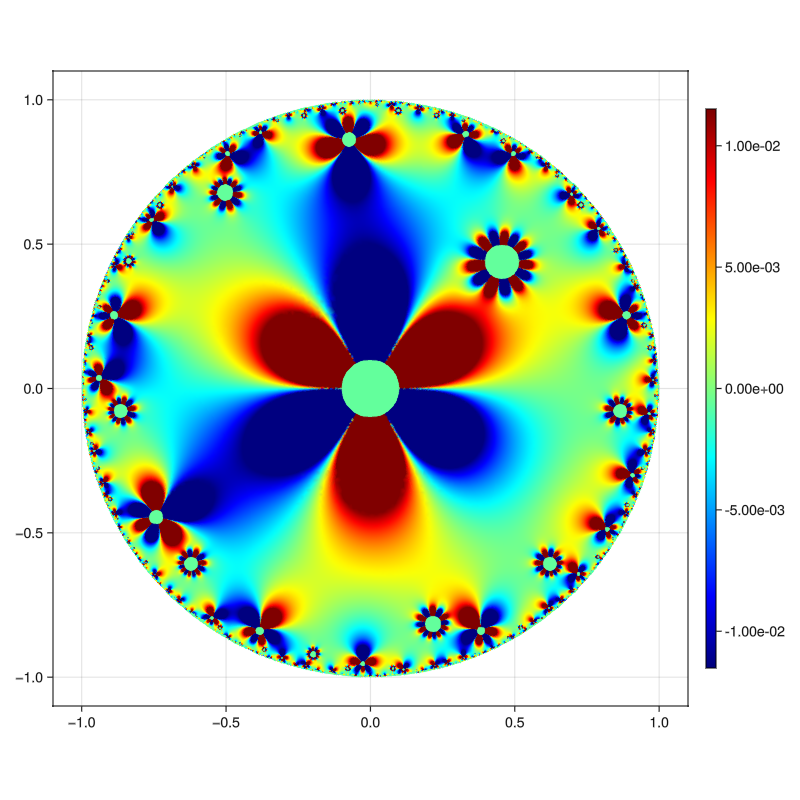}
            \end{subfigure}
        \end{tabular}
    \end{tabular}
    \caption{Harmonic extension on Gutzwiller surface by the method of particular solution.}
    \label{fig_harmonic}
\end{figure}

\begin{figure}[!htb]
    \centering
    \includegraphics[width=0.5\textwidth]{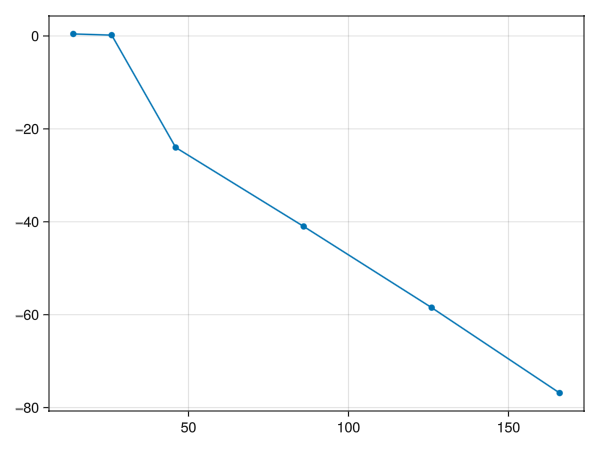}
    \caption{Spectral convergence with respect to the number of degrees of freedom: plot of the logarithm of the error  in the boundary condition of \eqref{eq_harmext} with respect to the number of degrees of freedom.} 
    \label{fig_spec}
\end{figure}

\section{Convergence of the least square approximation}\label{sec_leastsquareconv}

In this section we prove the theorem \ref{the_conv}. We fix $X$ as in the statement of the result (in particular $0\in H_1$ is not a Weierstrass point), and we divide the proof in several intermediate results. We will use the following abuse of notation: the edge $\gamma_{p,i}$ is seen as a parametrizing function
\[\gamma_{p,i}:[0,1]\to \D\]
with constant speed (equal to the length of $\gamma_{p,i}$ denoted $L(\gamma_{p,i})$). This way the sampling set is simply 
\[\Sa_{p,i}=\{\gamma_{p,i}(k/aN),\ k=0,1,2,\hdots,aN\}.\]
We start by introducing the notation\[\gamma_{p,i}^a=\gamma_{p,i}([1/a,1-1/a]).\]
For any $K\subset \C$, we write $\mathcal{N}^\eps(K)=\{z\in\C:\text{dist}(z,K)\leq \eps\}$. We introduce the following auxiliary functionals; for any $P=(P_{|1},\hdots,P_{|m})\in\C[z]^m$, we let:
\[F(P)=\sum_{(p,i)\to (q,j)}\Vert P_{|p}-g_{p,i}'P_{|q}\circ g_{p,i}\Vert_{L^\infty(\gamma_{p,i}^a)}^2\]
and
\[G(P)=\sum_{(p,i)\to (q,j)}\Vert P_{|p}-g_{p,i}'P_{|q}\circ g_{p,i}\Vert_{L^\infty\left(\mathcal{N}^{\frac{2L(\gamma_{p,i})}{a}}(\gamma_{p,i}^a)\right)}^2.\]
Note that $\mathcal{N}^{\frac{2L(\gamma_{p,i})}{a}}(\gamma_{p,i}^a)$ contains $\mathcal{N}^{\frac{L(\gamma_{p,i})}{a}}(\gamma_{p,i}([0,1]))$.

\begin{lemma}\label{lem_est_comparaison}
Let $\rho$ be as in theorem \ref{the_conv}, then for any large enough $N$ and any $\om\in\Om^1(X)$:
\[\inf_{\A^{N,\om}}\sqrt{E^N}< \rho^{N}\sum_{j=0}^{g-1}|\om_{|1}^{(j)}(0)|.\]
\end{lemma}
\begin{proof}
This is a consequence of the Bernstein-Walsch theorem (see \cite[Ch. VII]{W35} or \cite[Th 6.3.1]{R95} for a more recent reference): since $\om_{|p}$ extends as a holomorphic function on $\D$ (for any $p$) and no further, then
\[\limsup_{N\to\infty}\inf_{Q\in\C^N[z]}\Vert Q-\om_{|p}\Vert_{L^\infty(H_p)}^{1/N}=r,\]
where $r\in ]0,1[$ is the smallest value such that $\om_{|p}$ extends holomorphically to $H_p\sqcup \{\mathcal{G}_{ H_p,\infty}<\log(1/r)\}$, meaning the smallest value such that $\{\mathcal{G}_{ H_p,\infty}<\log(1/r)\}$ is included in the disk $\D$. By definition of $\rho$ we have, for any large enough $N$ and any $p$:
\[\inf_{Q\in\C^N[z]}\Vert Q-\om_{|p}\Vert_{L^\infty(H_p)}< \rho^N.\]
The result is then obtained from three facts: $\Om^{1}(X)$ has finite dimension, $E^N$ is a quadratic form and $\om\in\Om^1(X)\mapsto \sum_{j=0}^{g-1}|\om_{|1}^{(j)}(0)|$ is a norm (since the origin of $H_1$ is supposed to be generic).
\end{proof}

\begin{lemma}\label{lem_est_sample}
There exists a geometric constant $C>0$ such that for any $P=(P_{|1},\hdots,P_{|m})\in\C_N[z]^m$ we have

\[F(P)\lesssim C aN^4e^{NC/a} E^N(P).\]

\end{lemma}

\begin{proof}
Let $P=(P_{|1},\hdots,P_{|m})\in\C_N[z]^m$, let $(p,i)$ and $(q,j)$ be the index of two edges such that $(p,i)\to (q,j)$ and the gluing is made by some $g_{p,i}\in PSU(1,1)$ of the form
\[g_{p,i}(z)=\frac{az+b}{z-\om}\]
where $\om:=g_{p,i}^{-1}(\infty)\in\C\setminus \ov{\D}$. We define
\[f(z)=P_{|p}(z)-g_{p,i}'(z)P_{|q}(g_{p,i}(z))\in\frac{\C_{2N+2}[z]}{(z-\om)^{N+2}}.\]
Our goal is to give a bound of $\Vert f\Vert_{L^\infty(\gamma_{p,i}^a)}^2$ by $\frac{1}{aN}\sum_{k=0}^{aN-1}\left|f(\gamma_{p,i}(k/aN))\right|^2$ (since this implies the lemma). Let $z=\gamma_{p,i}(t)$ for some $t\in [\frac{1}{a},1-\frac{1}{a}]$. There exists some integer interval $\{s,s+1,\hdots,s+2N+2\}\subset \{0,1,2,\hdots,aN\}$ such that, denoting $t_k=\frac{s+k}{aN}$, we have
\[t_{N}\leq t\leq t_{N+1}.\]
Denote also $z_k=\gamma_{p,i}(t_k)$. For some constant $C>0$ that only depend on the geometry of $\gamma_{p,i}$, we have for any $s\neq t$ in $[t_0,t_{2N+2}]$:
\begin{equation}\label{eq_distorsionest}
e^{-C/a}\leq \frac{|\gamma_{p,i}(t)-\gamma_{p,i}(s)|}{L(\gamma_{p,i})|t-s|}\leq e^{C/a},
\end{equation}
where $L(\gamma_{p,i})$ is the length of $\gamma_{p,i}$. Since $(z-\om)^{N+2}f(z)$ is a polynomial of degree $2N+2$ at most, then by the classical Lagrange interpolation formula we have
\[f(z)=\sum_{k=0}^{2N+2}\left(\frac{z_k-\om}{z-\om}\right)^{N+2}f(z_k)\prod_{l=0,\ l\neq k}^{2N+2}\frac{z-z_l}{z_k-z_l},\]
and so by Cauchy-Schwarz inequality:
\begin{equation}\label{est_csf}
|f(z)|^2\leq \left(\frac{1}{aN}\sum_{k=0}^{aN}|f(\gamma_{p,i}(k/aN))|^2\right)\left(aN\sum_{k=0}^{2N+2}\left|\frac{z_k-\om}{z-\om}\right|^{2N+4}\prod_{l=0,\ l\neq k}^{2N+2}\left|\frac{z-z_l}{z_k-z_l}\right|^2\right).
\end{equation}
We now bound each term separately:
\begin{itemize}[label=\textbullet]
\item $\left|\frac{z_k-\om}{z-\om}\right|^{2N+4}\leq e^{NC/a}$ for some geometric constant $C>0$ (since $|z_k-\om|\leq e^{c/a}|z-\om|$ for some $c>0$).
\item $\prod_{l=0,\ l\neq k}^{2N+2}\left|z_k-z_l\right|\geq e^{-(2N+2)C/a}L(\gamma_{p,i})^{2N+2}(aN)^{-2N-2}N!(N+1)!$. Indeed, by the estimate \eqref{eq_distorsionest} we have $$|z_k-z_l|>e^{-C/a}L(\gamma_{p,i})\frac{|k-l|}{aN}.$$
The product $\prod_{l=0,\ l\neq k}^{2N+2}\left|k-l\right|=(k-1)!(2N+2-k)!$ is minimal at $k=N+1$ with value $N!(N+1)!$, so we get
\begin{align*}
\prod_{l=0,\ l\neq k}^{2N+2}\left|z_k-z_l\right|&\geq e^{-(2N+2)C/a}L(\gamma_{p,i})^{2N+2}(aN)^{-2N-2}\prod_{l=0,\ l\neq k}^{2N+2}|k-l|\\
&\geq e^{-(2N+2)C/a}L(\gamma_{p,i})^{2N+2}(aN)^{-2N-2}N!(N+1)!.
\end{align*}
\item $\prod_{l=0,\ l\neq k}^{2N+2}\left|z-z_l\right|\leq Ce^{(2N+2)C/a}L(\gamma_{p,i})^{2N+2}(aN)^{-2N-2}(N+1)!^2$. Indeed, let $r:=aNt-s\in [N,N+1]$ (by the definition of $s$ above), then using again estimate \eqref{eq_distorsionest} we get
\begin{align*}
\prod_{l=0,\ l\neq k}^{2N+2}\left|z-z_l\right|&\leq e^{(2N+2)C/a}L(\gamma_{p,i})^{2N+2}(aN)^{-2N-2}\prod_{l=0,l\neq k}^{2N+2}|r - k|\\
&\leq e^{(2N+2)C/a}L(\gamma_{p,i})^{2N+2}(aN)^{-2N-2}(N+1)!^2
\end{align*}
\end{itemize}
Applying these three estimates to the last factor of \eqref{est_csf} we obtain (for to a possibly different geometric constant $C>0$):
\begin{align*}
|f(z)|^2&\leq CaN^4e^{NC/a} \left(\frac{1}{aN}\sum_{k=0}^{aN}|f(\gamma_{p,i}(k/aN))|^2\right),
\end{align*}
which implies the result.
\end{proof}

\begin{lemma}\label{lem_est_bernstein}
There exists a geometric constant $\beta>0$ such that for any $P=(P_{|1},\hdots,P_{|m})\in\C_N[z]^m$, we have the inequality
\[G(P)\leq e^{N\beta/\sqrt{a}}F(P).\]
\end{lemma}
\begin{proof}
For any $t\in [0,1/4]$, there is a continuously-defined homography $h_t$ that maps the circle arc $\gamma_{p,i}([t,1-t])$ to the segment $[-1,1]$, then we have
\[\mathcal{G}_{\gamma_{p,i}([t,1-t]),\om}(z)=\mathcal{G}_{ [-1,1],h_t(\om)}(h_t(z)).\]
The function $\mathcal{G}_{ [-1,1]}$ is $\frac{1}{2}$-Hölder, so the constant
\[\beta=\max_{\text{edge index }(p,i)}\sqrt{2L(\gamma_{p,i})}\left[\Vert \mathcal{G}_{\gamma_{p,i}([1/4,3/4]),\infty}\Vert_{\mathcal{C}^{0,\frac{1}{2}}(\D_1)}+\Vert \mathcal{G}_{\gamma_{p,i}([1/4,3/4]),g_{p,i}^{-1}(\infty)}\Vert_{\mathcal{C}^{0,\frac{1}{2}}(\D_1)}\right],\]
is finite. Define
\begin{align*}
\mathcal{G}(z)&=\mathcal{G}_{\gamma_{p,i}([1/a,1-1/a]),\infty}(z)+\mathcal{G}_{\gamma_{p,i}^a,g_{p,i}^{-1}(\infty)}(z),\\
f(z)&=P_{|p}(z)-g_{p,i}'(z)P_{|q}(g_{p,i}(z))\in\frac{\C_{2N+2}[z]}{(z-\om)^{N+2}},\\
h(z)&=\frac{1}{N}\log\frac{|f(z)|}{\Vert f\Vert_{L^\infty(\gamma_{p,i}^a)}}-\mathcal{G}(z).
\end{align*}
We see $h$ as a function of $\widehat{\C}\setminus\gamma_{p,i}^a$ with a (finite) number of logarithmic singularities. Without loss of generality we may slightly perturb $f$ so that it does not vanish on $\gamma_{p,i}$. Notice then that:
\begin{itemize}
\item[1) ]$h\leq 0$ in $\gamma_{p,i}^a$.
\item[2) ]$h$ is subharmonic in $\widehat{\C}\setminus\gamma_{p,i}^a$: in fact $h$ is harmonic outside a finite number of logarithmic singularities at the zeroes of $f$ (where $h$ is then subharmonic) at $\infty$ and $g_{p,i}^{-1}(\infty)$ (where $h$ is also subharmonic due to the degree condition of $f$).
\end{itemize}
Thus, by maximum principle, $h\leq 0$ everywhere, so
\[|f(z)|\leq e^{N \mathcal{G}(z)}\Vert f\Vert_{L^\infty(\gamma_{p,i}^a)}.\]
Since every point of $\mathcal{N}^{\frac{2L(\gamma_{p,i})}{a}}(\gamma_{p,i}^a)$ has distance $\frac{2L(\gamma_{p,i})}{a}$ to some point of $\gamma_{p,i}^a$, then by the definition of the constant $\beta$ we have for any $z\in \mathcal{N}^{\frac{2L(\gamma_{p,i})}{a}}(\gamma_{p,i}^a)$
\[\mathcal{G}(z)\leq \beta/\sqrt{a},\]
which gives the expected result.
\end{proof}
\begin{lemma}\label{lem_est_riemannroch}
There exists a constant $c_a>0$ such that, for any 
\[f=(f_{|1},\hdots,f_{|m})\in\Oo(\mathcal{N}^{\frac{L(\gamma_{p,i})}{a}}(H_1))\times \Oo(\mathcal{N}^{\frac{L(\gamma_{p,i})}{a}}(H_2))\times\hdots\times \Oo(\mathcal{N}^{\frac{L(\gamma_{p,i})}{a}}(H_m)),\]
we have
\[c_a\sum_{p=1}^{m}\Vert f_{|p}\Vert_{L^\infty(H_p)}\leq G(f)^{1/2}+\sum_{j=0}^{g-1}|f_{|1}^{(j)}(0)|.\]
\end{lemma}

The constant $c_a$ that we obtain here is not explicit due to the fact that our proof is done by contradiction and compactness: it depends on the geometry and on $a$ in a non-trivial way.

\begin{proof}
Suppose that this statement is false. Then there exists a sequence $f^k=(f_{|1}^k,\hdots,f_{|m}^k)$ verifying
\[G(f^k)^{1/2}+\sum_{j=0}^{g-1}|f_1^{k(j)}(0)|\underset{k\to\infty}{\longrightarrow}0\]
and for any $k$:
\[\sum_{p=1}^{m}\Vert f_{|p}^k\Vert_{L^\infty(H_p)}=1.\]
\textbf{Claim}: there exists a constant $C>0$ and some neighbourhoods $(\tilde{H}_p)_p$ of $(H_p)_p$ such that
\[\sum_{p=1}^{m}\Vert f_{|p}^k\Vert_{L^\infty(\tilde{H_p})}\leq C.\]
This is obtained by induction. Let us first remind that, denoting $L$ the minimal length among all sides, we have
\[G(f^k)\geq \sum_{(p,i)\to (q,j)}\Vert f^k_{|p}-g_{p,i}'f^k_{|q}\circ g_{p,i}\Vert_{L^\infty\left(\mathcal{N}^{\frac{L}{a}}(\gamma_{p,i})\right)}^2.\]
\begin{itemize}[label=\textbullet]
\item For every $p=1,\hdots,m$, we have
\[\Vert f^k_{|p}\Vert_{L^\infty(H_p)}\leq 1.\]
\item For every edge index $(p,i)$, denote $(q_{p,i},j_{p,i})$ the corresponding edge. Then for every $z\in \mathcal{N}^{\frac{L}{a}}(\gamma_{p,i})\cap g_{p,i}^{-1}(H_{q_{p,i}})$, we have
\begin{align*}
|f^k_{|p}(z)|&\leq |f^k_{|p}(z)- g_{p,i}'(z)f^k_{|q_{p,i}}\circ g_{q_{p,i},j_{p,i}}(z)|+| g_{p,i}'(z)f^k_{|q_{p,i}}\circ  g_{p,i}(z)|\\
&\leq G(f^k)^\frac{1}{2}+\Vert g_{p,i}\Vert_{L^\infty(\D)}.
\end{align*}
\item For every $p$, and every $n\in\N$, let $H_p^{(n)}$ be the set defined inductively by
\[\begin{cases}
H_p^{(0)}=H_p\\
H_p^{(n)}=H_p^{(n-1)}\cup\bigcup_{i=1}^{c(p)}\left(\mathcal{N}^{\frac{L}{a}}(\gamma_{p,i})\cap g_{p,i}^{-1}(H_{q_{p,i}}^{(n-1)})\right).
\end{cases}\]
Then by the same computation as above, each $f_{|p}$ is bounded independently of $k$ in $L^\infty(H_{p}^{(n)})$ by
\[\sup_{p=1,\hdots,m}\Vert f^k_{|p}\Vert_{L^\infty(H^{(n)}_p)}\leq G(f^k)^\frac{1}{2}+\left(\sup_{(p,i)}\Vert g_{p,i}'\Vert_{L^\infty(\D)}\right)\left(\sup_{p=1,\hdots,m} \Vert f^{k}_{|p}\Vert_{H_p^{(n-1)}}\right).\]
\item Taking $n$ to be larger than the maximal number of vertices of the polygons $(H_p)$ that are identified to a single point (we may take the worst case with $n>c(1)+c(2)+\hdots+c(m)$), then $H_p^{(n)}$ is a neighbourhood of $H_p$ in $\D$, that is denoted $\tilde{H}_p$ in the claim.
\end{itemize}

In particular, up to extraction we may suppose that each $f^k_{|p}$ converges locally uniformly in a (possibly smaller) neighbourhood of $H_p$ to some limit $f_{|p}$ when $k\to\infty$, which verify the periodicity relation
\[f_{|p}=g_{p,i}'f_{|q}\circ g_{p,i}\]
when $(p,i)\to (q,j)$. As a consequence, $f=(f_{|1},\hdots,f_{|m})$ is an element of $\Om^1(X)$, that is non-zero since $\sum_{p=1}^{m}\Vert f_{|p}\Vert_{L^\infty(H_p)}=1$ (because this is verified by $f^{k}_{|p}$, which converges to $f_{|p}$ up to the boundary), and such that $\sum_{j=0}^{g-1}|f^{(j)}_{|1}(0)|=0$: this is in contradiction with the fact that $0$ is not a Weierstrass point, meaning that $\om\in\Om^1(X)\mapsto (\om_{|1}(0),\hdots,\om_{|1}^{(g-1)}(0))$ must be injective.
\end{proof}
We may now prove the theorem \ref{the_conv}:
\begin{proof}
Let now $\rho$ be as in the statement of theorem \ref{the_conv}, let $\rho'<\rho$ also verifying the same lower bound. We may choose $a$ large enough such that
\[\rho^{'2} e^{\frac{C}{a}+\frac{\beta}{\sqrt{a}}}<\rho^2.\]
Then applying successively the lemmas \ref{lem_est_comparaison}, \ref{lem_est_sample}, \ref{lem_est_bernstein}, we have:
\begin{align*}
G(P^{N,\om})&\leq e^{\frac{N\beta}{\sqrt{a}}}F(P^{N,\om})\text{ by lemma \ref{lem_est_bernstein}}\\
&\leq CaN^4e^{\frac{N\beta}{\sqrt{a}}+\frac{NC}{a}}E^N(P^{N,\om})\text{ by lemma \ref{lem_est_sample}}\\
& \leq \rho^{'2N}CaN^4e^{\frac{N\beta}{\sqrt{a}}+\frac{NC}{a}}\left(\sum_{j=0}^{g-1}\left|\om^{(j)}_{|1}(0)\right|\right)^2\text{ by lemma \ref{lem_est_comparaison}}.\\
&\leq \rho^{2N}\left(\sum_{j=0}^{g-1}\left|\om^{(j)}_{|1}(0)\right|\right)^2\text{ for large enough }N
\end{align*}
We apply now lemma \ref{lem_est_riemannroch} to $P^{N,\om}-\om$ ; note that $G(P^{N,\om}-\om)=G(P^{N,\om})$ by the periodicity of $\om$, and the first $g-1$ derivatives of $P^{N,\om}_{|1}-\om_{|1}$ at $0\in H_1$ are zero, so:
\begin{align*}
c_a\sum_{p=1}^{m}\Vert P^{N,\om}_{|p}-\om_{|p}\Vert_{L^\infty(H_p)}&\leq G(P^{N,\om}-\om)^\frac{1}{2}\leq \rho^N\sum_{j=0}^{g-1}\left|\om^{(j)}_{|1}(0)\right|.
\end{align*}
\end{proof}

Let us briefly comment on the fact that the lower bound on $\rho$ given in theorem \ref{the_conv} is optimal: suppose that the convergence is still valid for some $\rho$ strictly lower than this bound, meaning that for some $p$ the set
\[\{z\in\C:\mathcal{G}_{ H_p}(z,\infty)<\log(1/\rho)\}\]
is not included in the disk $\D$. Then by Bernstein-Walsh theorem this gives us an extension of $\om_{|p}$ to the set $\{z\in\C:\mathcal{G}_{ H_p}(z,\infty)<\log(1/\rho)\}$ which contains a point of $\partial \D$: this is a contradiction, as the functions $\om_{|p}$ cannot be extended beyond the boundary $\partial \D$.

\appendix

\section{Some meromorphic functions on genus 2 surfaces}

\begin{figure}[!htb]
    \centering
    \begin{tabular}[t]{ccc}
        \begin{tabular}{c}
            \hfill
            \begin{subfigure}[t]{0.22\textwidth}
                \centering
                \includegraphics[width=\linewidth]{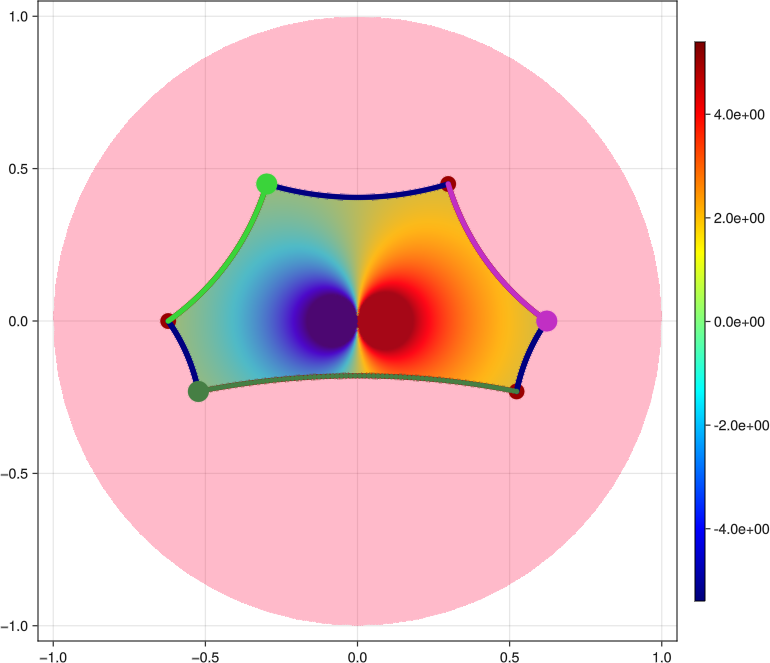}
            \end{subfigure}\\
            \hfill
            \begin{subfigure}[t]{0.22\textwidth}
                \centering
                \includegraphics[width=\linewidth]{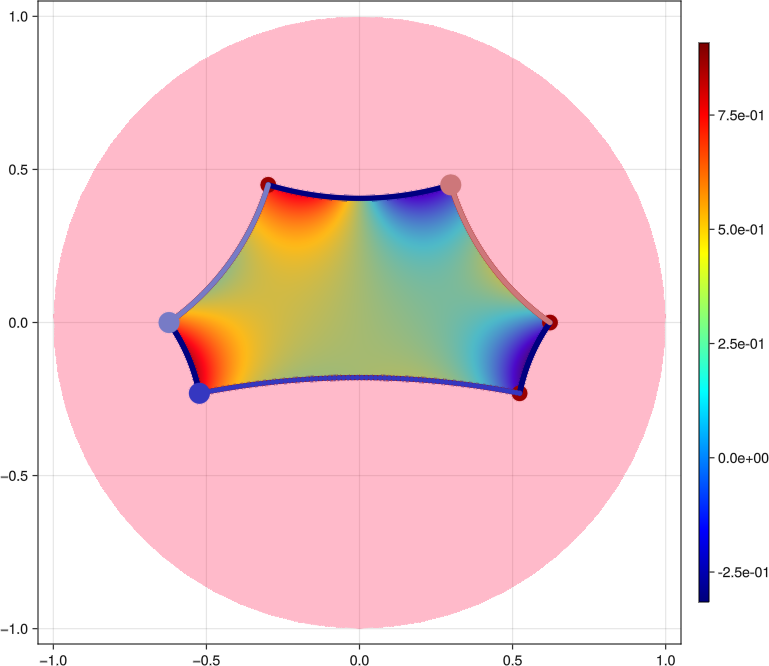}
            \end{subfigure}
        \end{tabular}
        &
        \begin{tabular}{c}
            \hspace{-1cm}
            \begin{subfigure}[t]{0.22\textwidth}
                \centering
                \includegraphics[width=\linewidth]{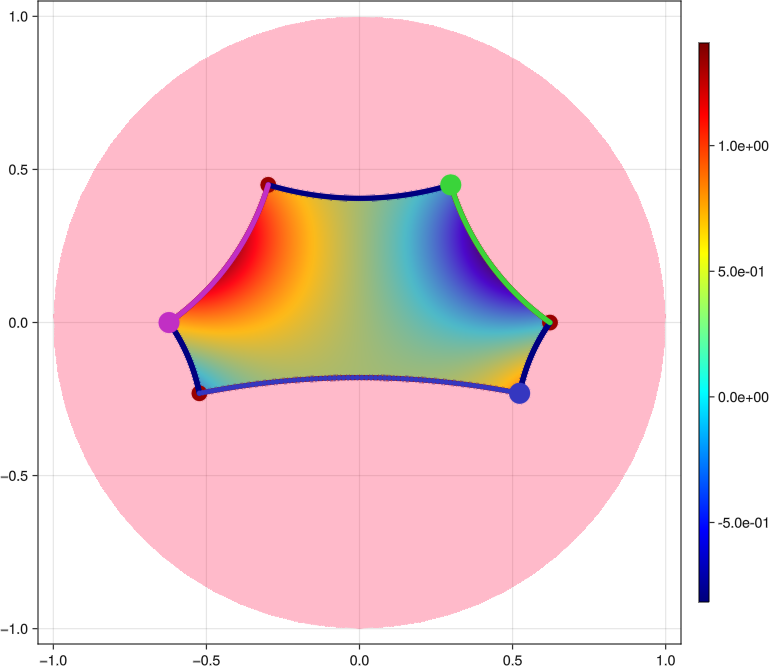}
            \end{subfigure}\\
            \hspace{-1cm}
            \begin{subfigure}[t]{0.22\textwidth}
                \centering
                \includegraphics[width=\linewidth]{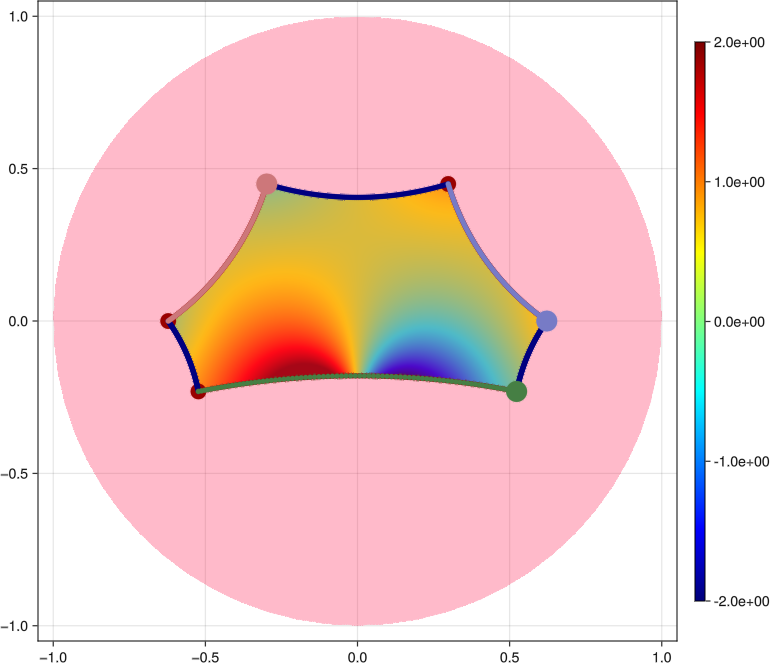}
            \end{subfigure}
        \end{tabular}
        &
        \begin{tabular}{c}
            \hspace{-0.75cm}
            \begin{subfigure}[t]{0.42\textwidth}
                \centering
                \includegraphics[width=\linewidth]{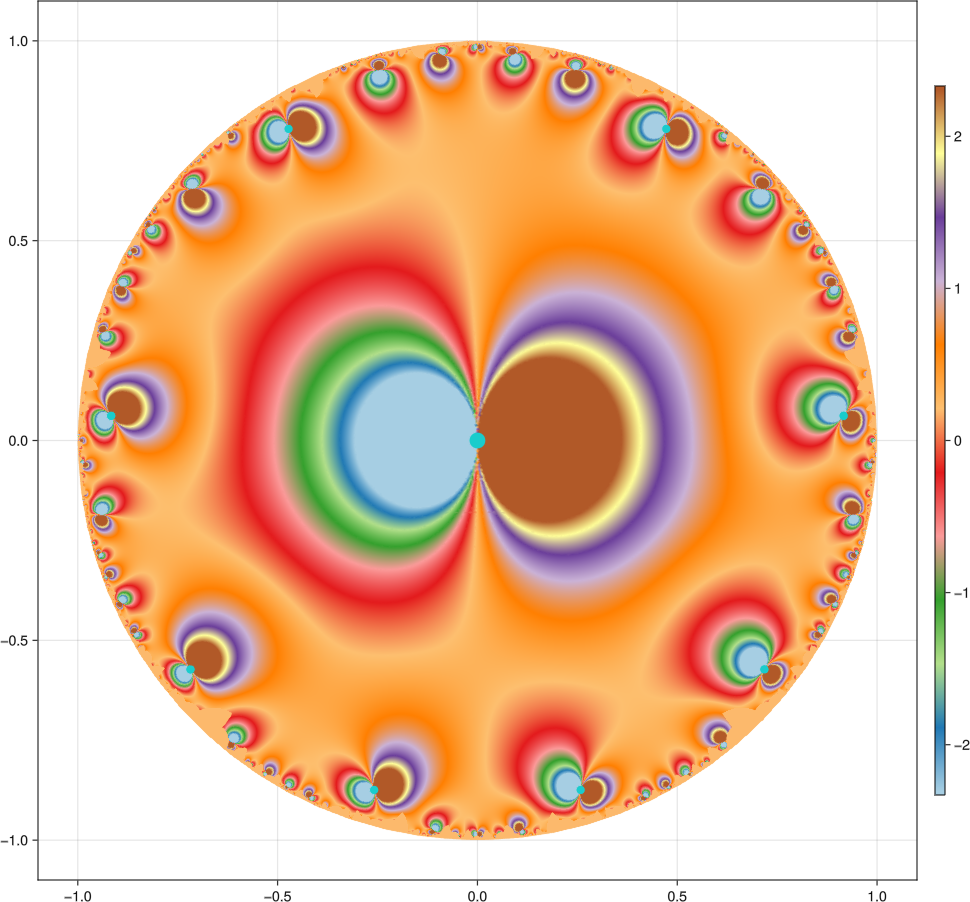}
            \end{subfigure}
        \end{tabular}
    \end{tabular}
    \caption{Pole of order $1$ on Bolza's surface. The Real part.}
    \label{fig_BolzaOrder1R}
\end{figure}

\begin{figure}[!htb]
    \centering
    \begin{tabular}[t]{ccc}
        \begin{tabular}{c}
            \hfill
            \begin{subfigure}[t]{0.22\textwidth}
                \centering
                \includegraphics[width=\linewidth]{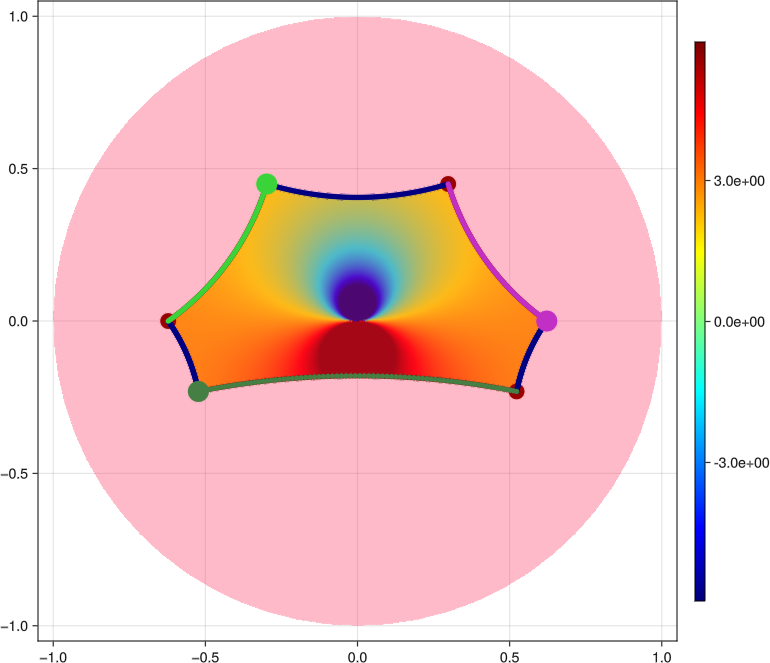}
            \end{subfigure}\\
            \hfill
            \begin{subfigure}[t]{0.22\textwidth}
                \centering
                \includegraphics[width=\linewidth]{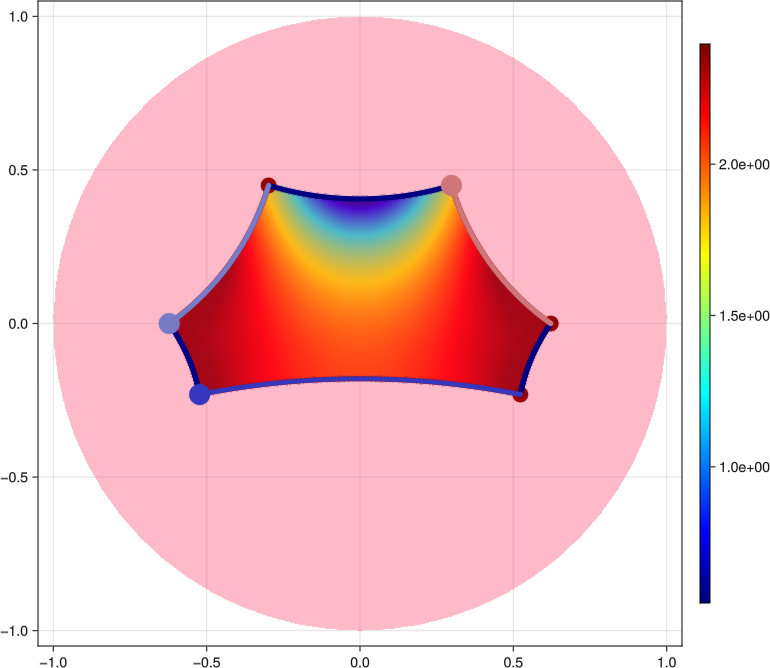}
            \end{subfigure}
        \end{tabular}
        &
        \begin{tabular}{c}
            \hspace{-1cm}
            \begin{subfigure}[t]{0.22\textwidth}
                \centering
                \includegraphics[width=\linewidth]{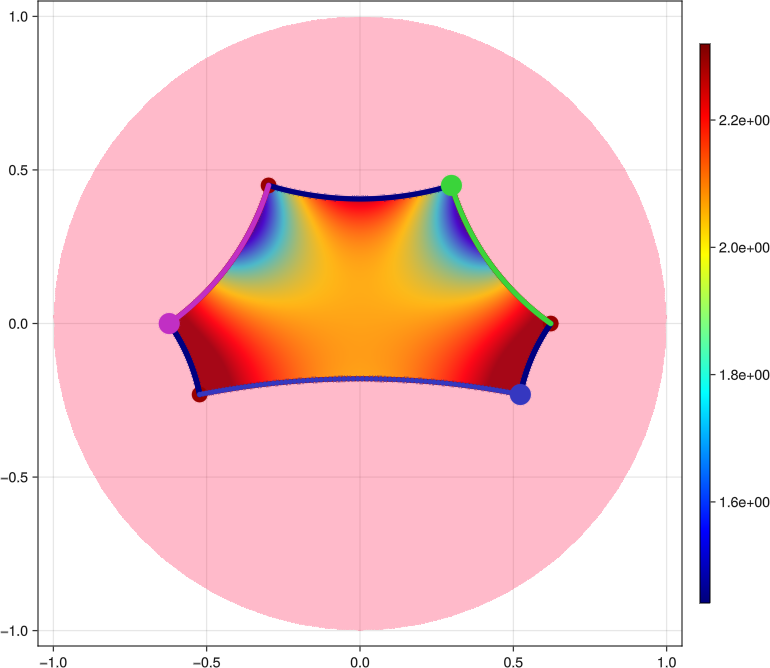}
            \end{subfigure}\\
            \hspace{-1cm}
            \begin{subfigure}[t]{0.22\textwidth}
                \centering
                \includegraphics[width=\linewidth]{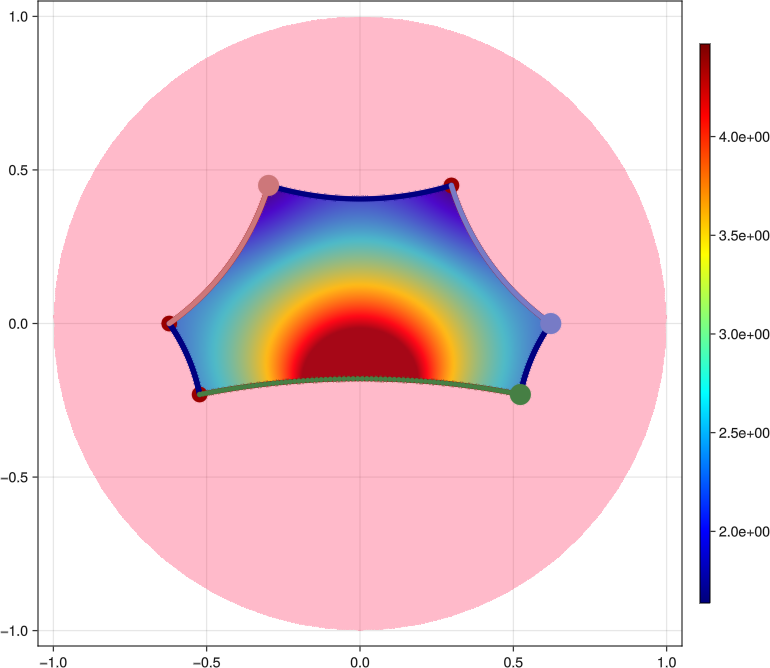}
            \end{subfigure}
        \end{tabular}
        &
        \begin{tabular}{c}
            \hspace{-0.75cm}
            \begin{subfigure}[t]{0.42\textwidth}
                \centering
                \includegraphics[width=\linewidth]{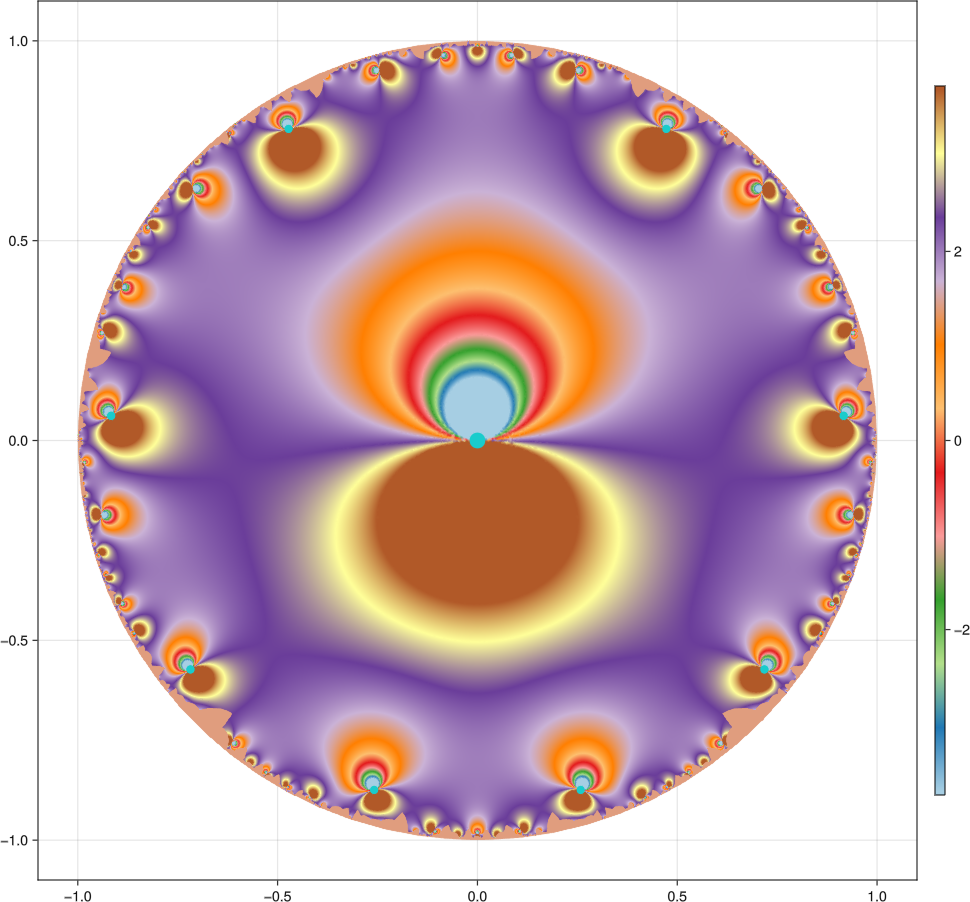}
            \end{subfigure}
        \end{tabular}
    \end{tabular}
    \caption{Pole of order $1$ on Bolza's surface. The Imaginary part.}
    \label{fig_BolzaOrder1I}
\end{figure}

\begin{figure}[!htb]
    \centering
    \begin{tabular}[t]{ccc}
        \begin{tabular}{c}
            \hfill
            \begin{subfigure}[t]{0.22\textwidth}
                \centering
                \includegraphics[width=\linewidth]{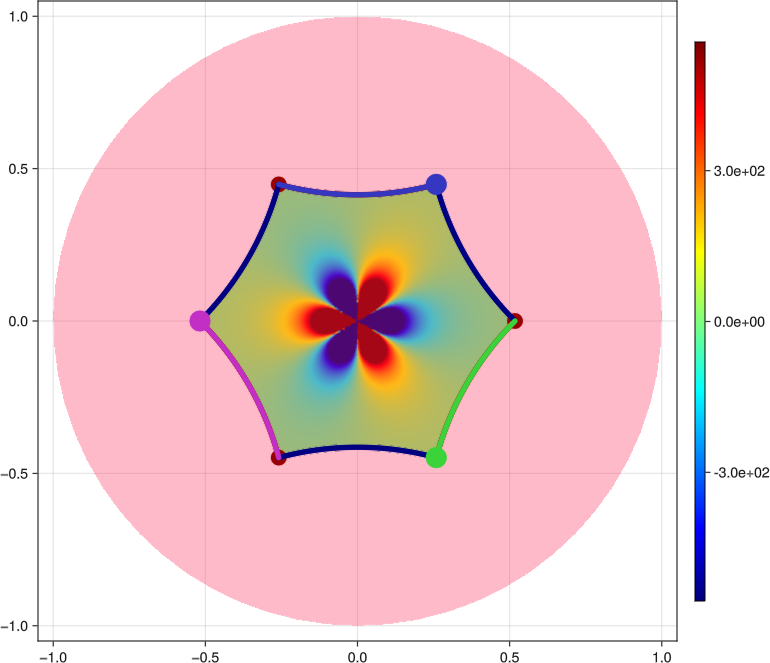}
            \end{subfigure}\\
            \hfill
            \begin{subfigure}[t]{0.22\textwidth}
                \centering
                \includegraphics[width=\linewidth]{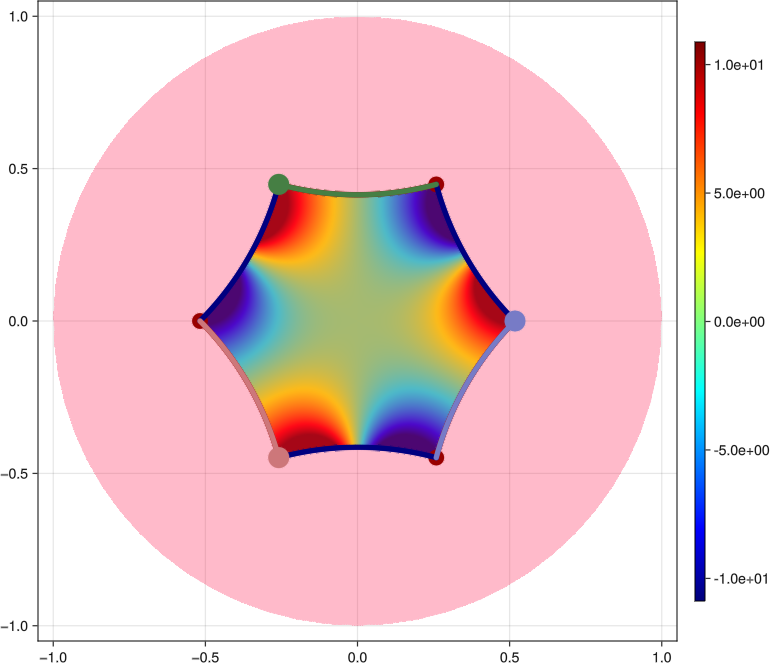}
            \end{subfigure}
        \end{tabular}
        &
        \begin{tabular}{c}
            \hspace{-1cm}
            \begin{subfigure}[t]{0.22\textwidth}
                \centering
                \includegraphics[width=\linewidth]{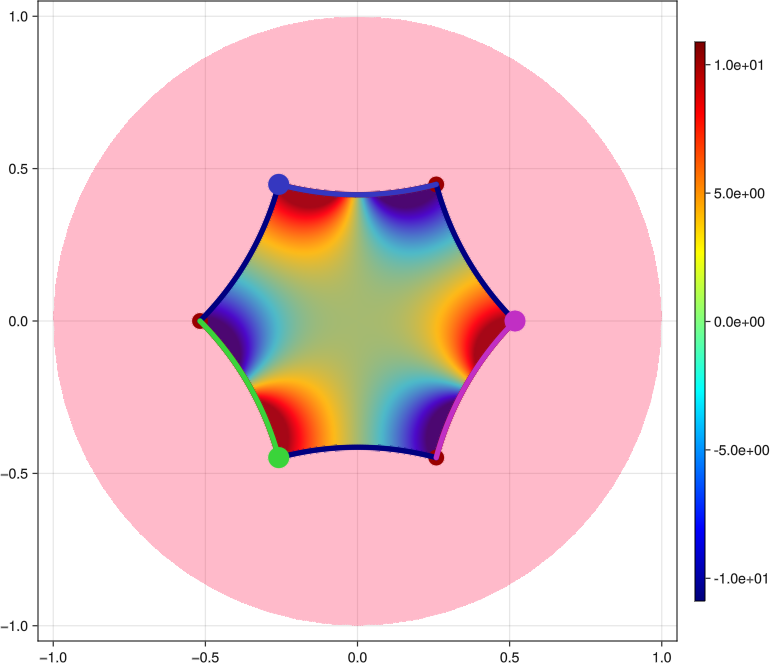}
            \end{subfigure}\\
            \hspace{-1cm}
            \begin{subfigure}[t]{0.22\textwidth}
                \centering
                \includegraphics[width=\linewidth]{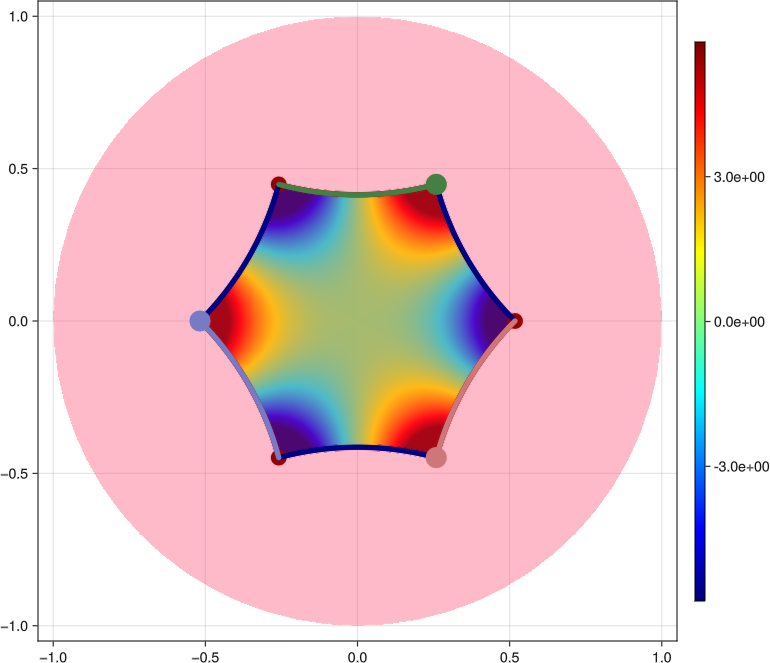}
            \end{subfigure}
        \end{tabular}
        &
        \begin{tabular}{c}
            \hspace{-0.75cm}
            \begin{subfigure}[t]{0.45\textwidth}
                \centering
                \includegraphics[width=\linewidth]{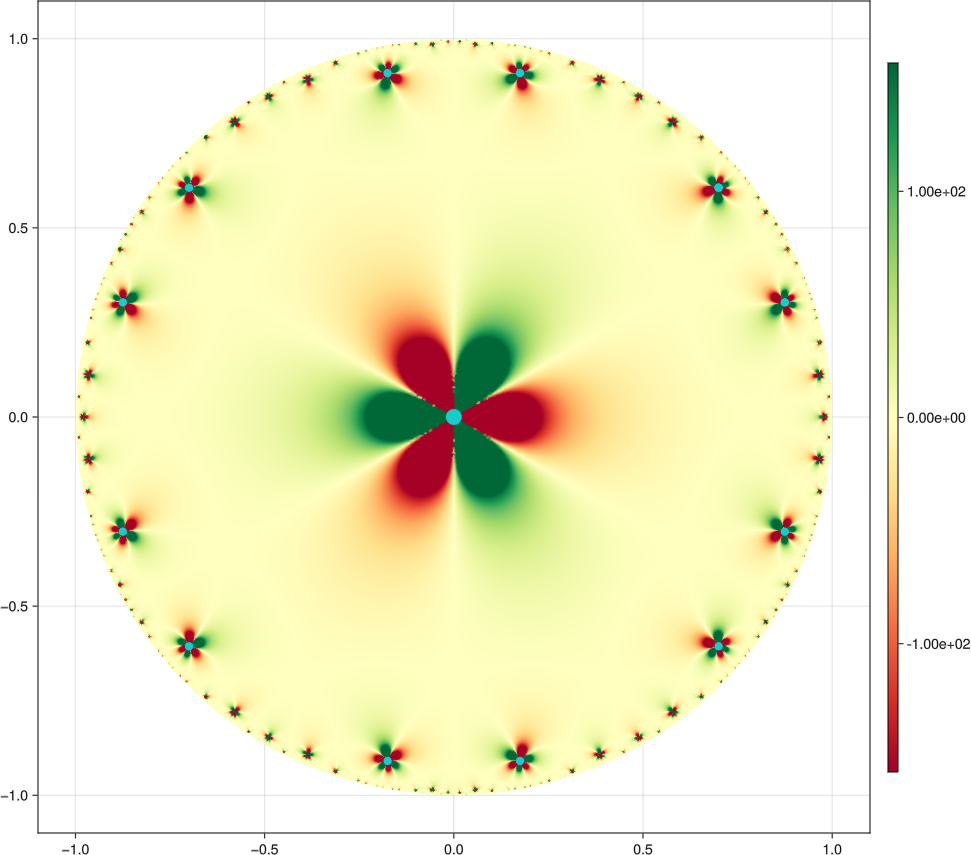}
            \end{subfigure}
        \end{tabular}
    \end{tabular}
    \caption{Pole of order $3$ on  $D_6 \times \mathbb{Z}_2$ surface. The Real part.}
    \label{fig_Weierstrass3D6R}
\end{figure}

\begin{figure}[!htb]
    \centering
    \begin{tabular}[t]{ccc}
        \begin{tabular}{c}
            \hfill
            \begin{subfigure}[t]{0.22\textwidth}
                \centering
                \includegraphics[width=\linewidth]{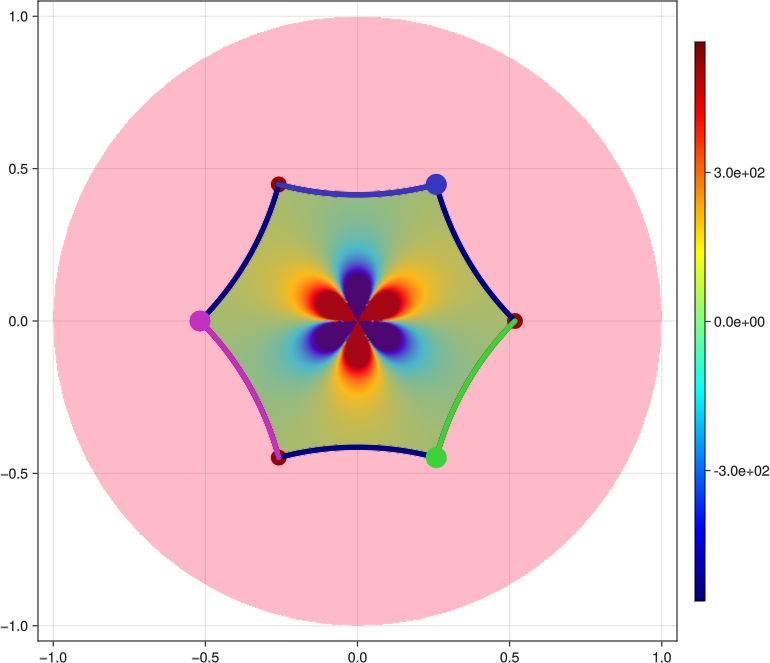}
            \end{subfigure}\\
            \hfill
            \begin{subfigure}[t]{0.22\textwidth}
                \centering
                \includegraphics[width=\linewidth]{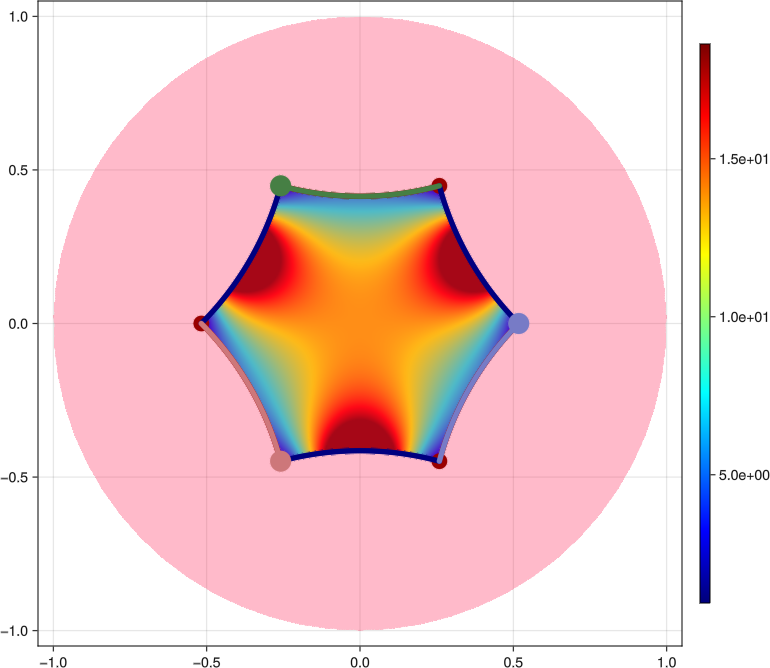}
            \end{subfigure}
        \end{tabular}
        &
        \begin{tabular}{c}
            \hspace{-1cm}
            \begin{subfigure}[t]{0.22\textwidth}
                \centering
                \includegraphics[width=\linewidth]{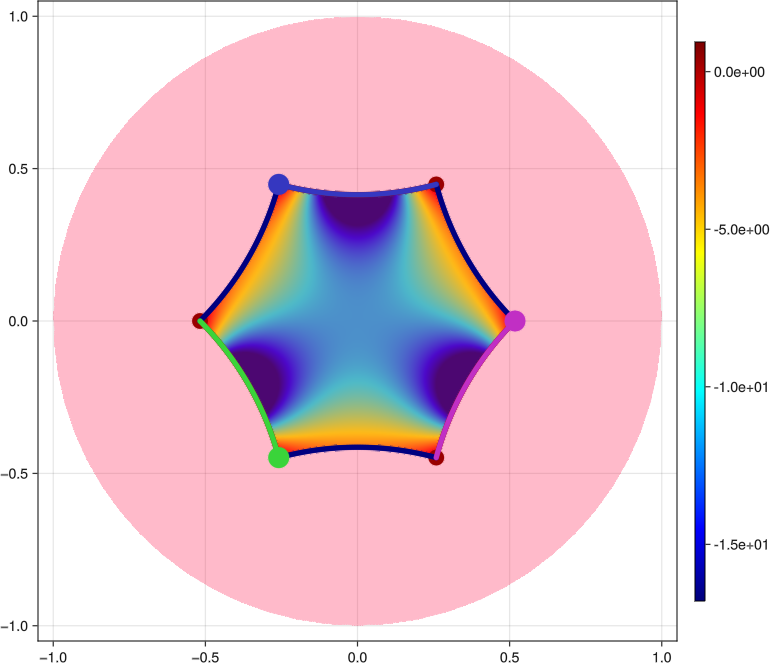}
            \end{subfigure}\\
            \hspace{-1cm}
            \begin{subfigure}[t]{0.22\textwidth}
                \centering
                \includegraphics[width=\linewidth]{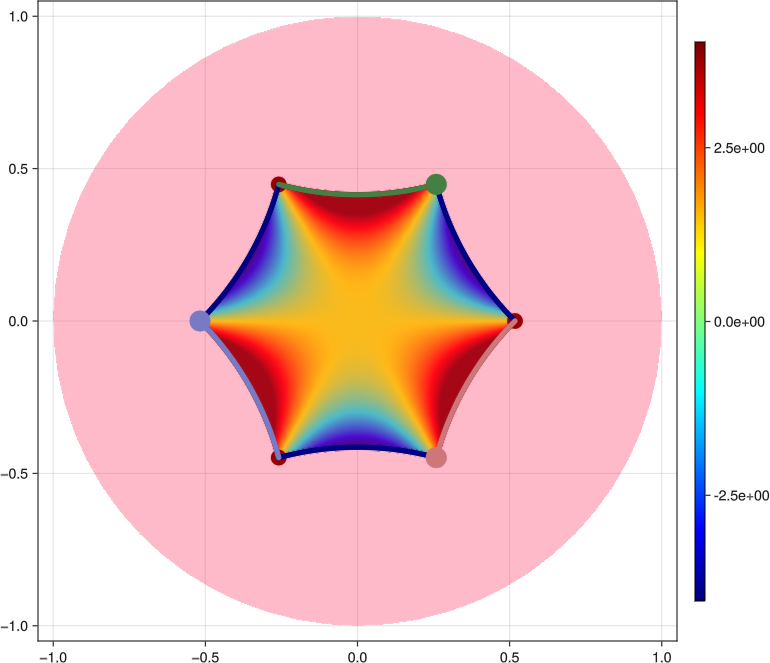}
            \end{subfigure}
        \end{tabular}
        &
        \begin{tabular}{c}
            \hspace{-0.75cm}
            \begin{subfigure}[t]{0.45\textwidth}
                \centering
                \includegraphics[width=\linewidth]{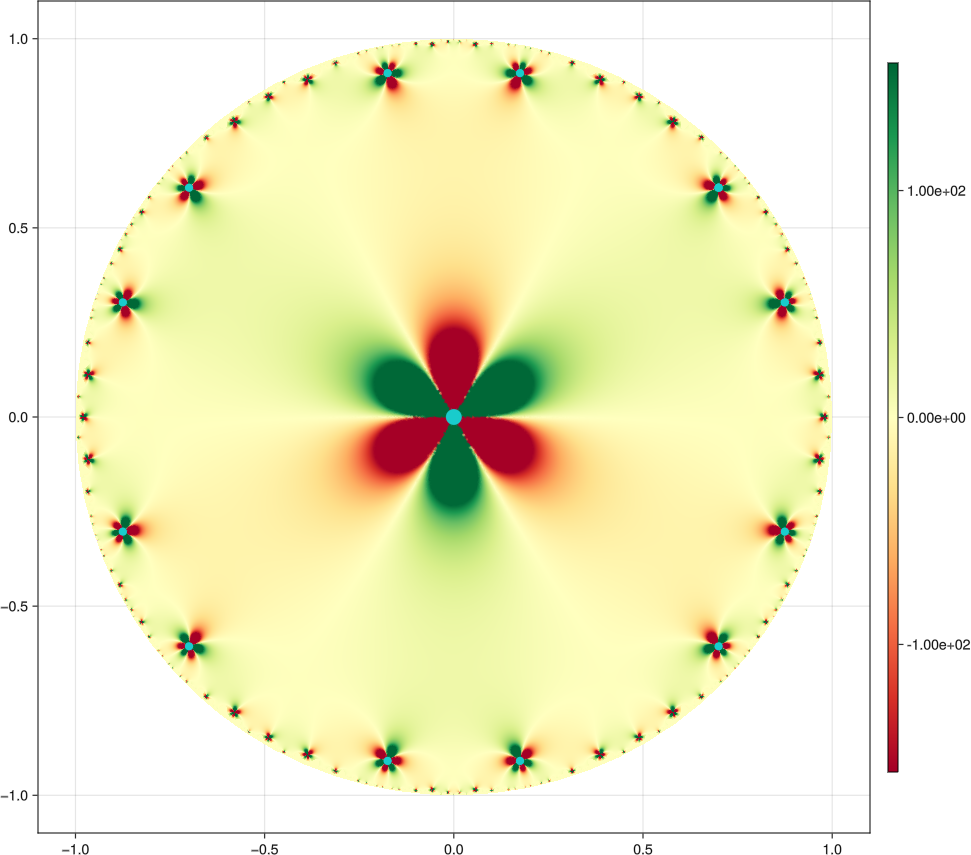}
            \end{subfigure}
        \end{tabular}
    \end{tabular}
    \caption{Pole of order $3$ on $D_6 \times \mathbb{Z}_2$ surface. The Imaginary part.}
    \label{fig_Weierstrass3D6I}
\end{figure}

\begin{figure}[!htb]
    \centering
    \begin{tabular}[t]{ccc}
        \begin{tabular}{c}
            \hfill
            \begin{subfigure}[t]{0.22\textwidth}
                \centering
                \includegraphics[width=\linewidth]{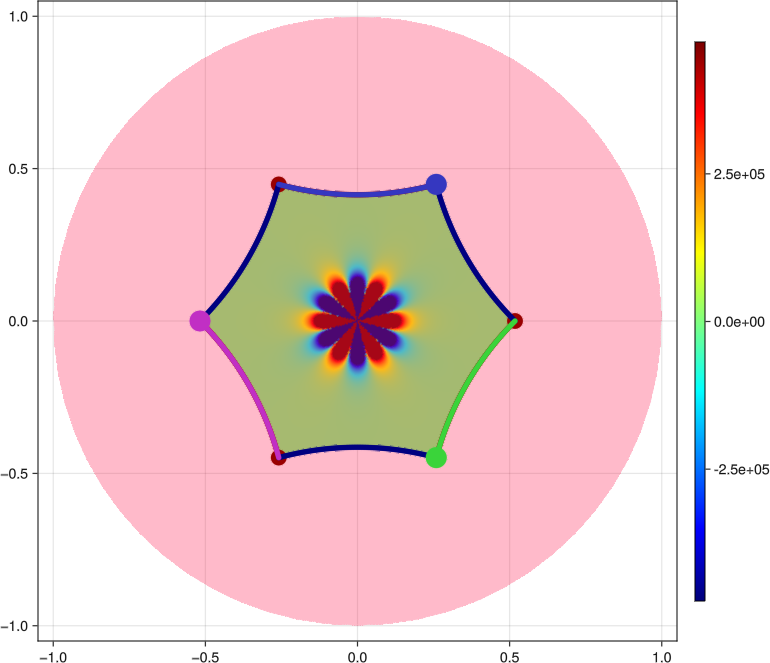}
            \end{subfigure}\\
            \hfill
            \begin{subfigure}[t]{0.22\textwidth}
                \centering
                \includegraphics[width=\linewidth]{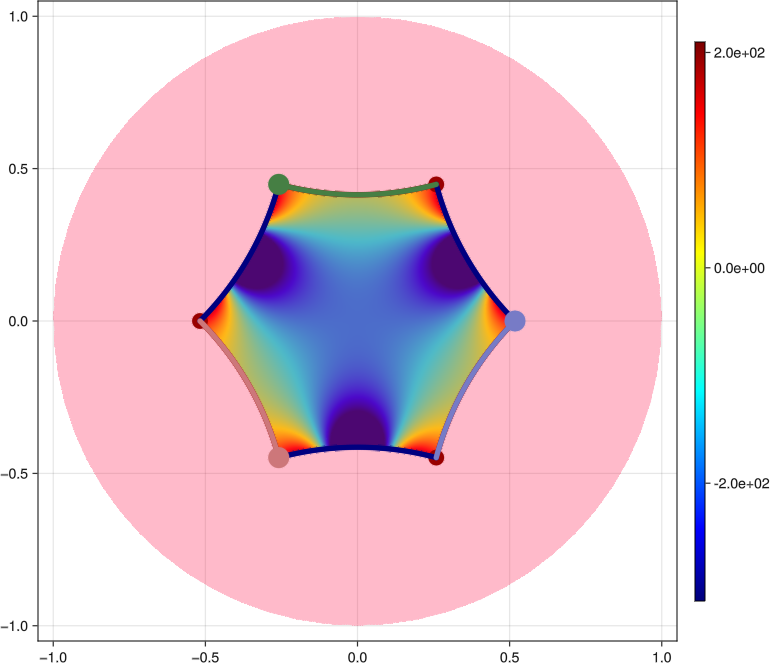}
            \end{subfigure}
        \end{tabular}
        &
        \begin{tabular}{c}
            \hspace{-1cm}
            \begin{subfigure}[t]{0.22\textwidth}
                \centering
                \includegraphics[width=\linewidth]{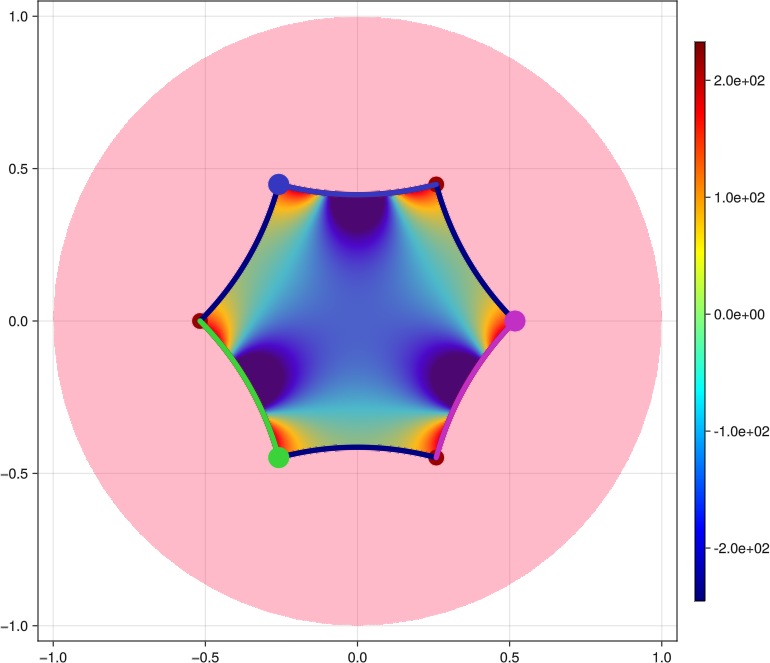}
            \end{subfigure}\\
            \hspace{-1cm}
            \begin{subfigure}[t]{0.22\textwidth}
                \centering
                \includegraphics[width=\linewidth]{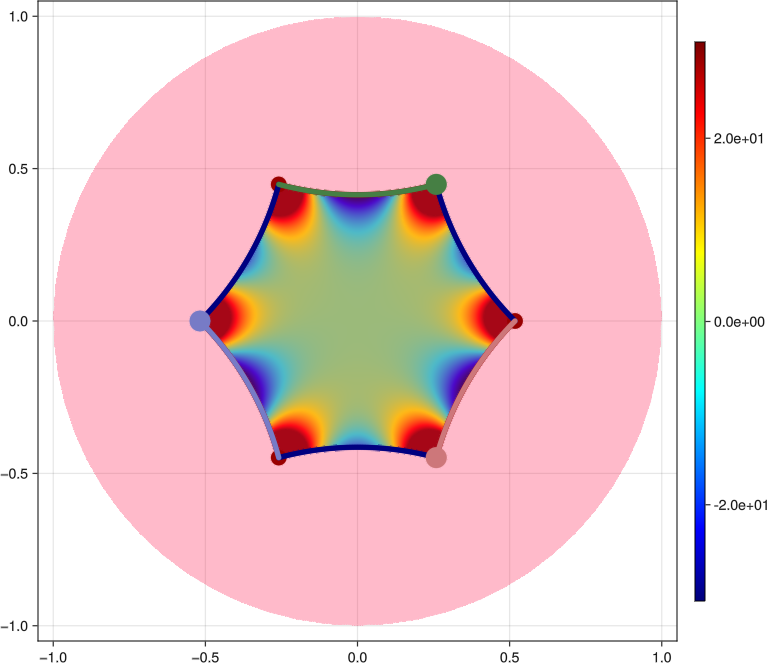}
            \end{subfigure}
        \end{tabular}
        &
        \begin{tabular}{c}
            \hspace{-0.75cm}
            \begin{subfigure}[t]{0.45\textwidth}
                \centering
                \includegraphics[width=\linewidth]{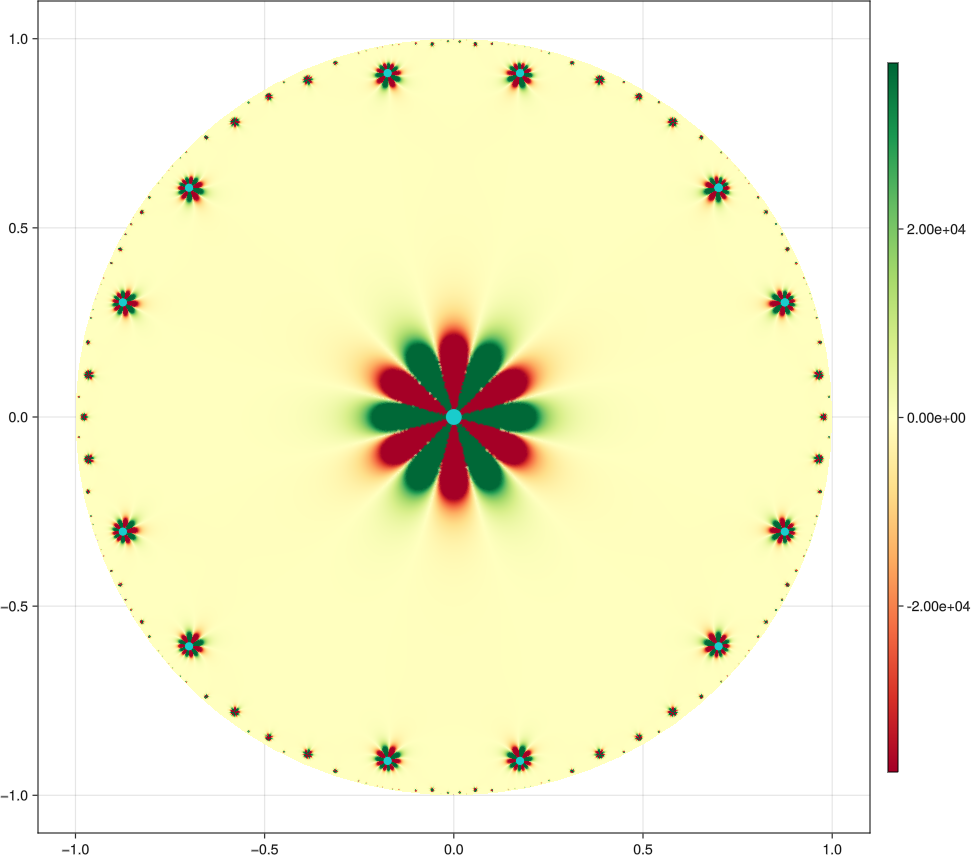}
            \end{subfigure}
        \end{tabular}
    \end{tabular}
    \caption{Pole of order $6$ on  $D_6 \times \mathbb{Z}_2$ surface. The Real part.}
    \label{fig_Weierstrass6D6R}
\end{figure}

\begin{figure}[!htb]
    \centering
    \begin{tabular}[t]{ccc}
        \begin{tabular}{c}
            \hfill
            \begin{subfigure}[t]{0.22\textwidth}
                \centering
                \includegraphics[width=\linewidth]{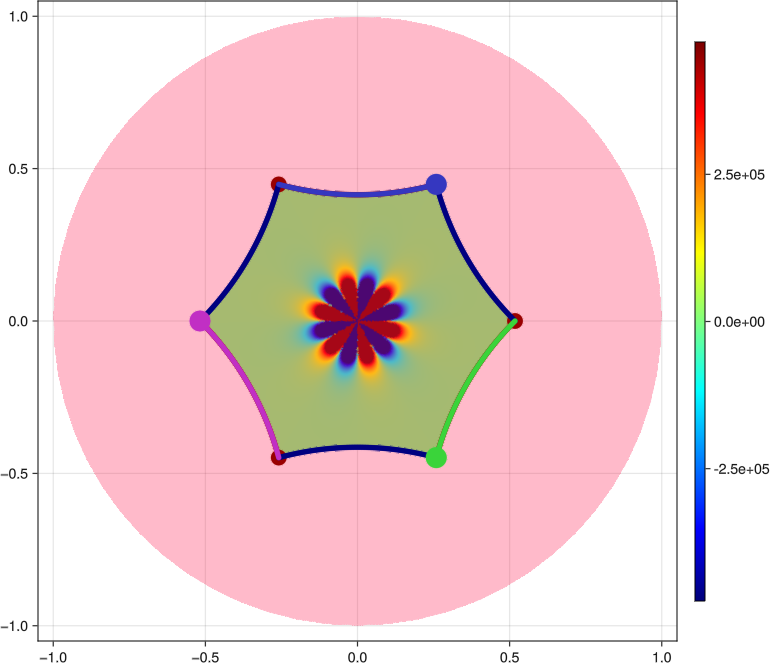}
            \end{subfigure}\\
            \hfill
            \begin{subfigure}[t]{0.22\textwidth}
                \centering
                \includegraphics[width=\linewidth]{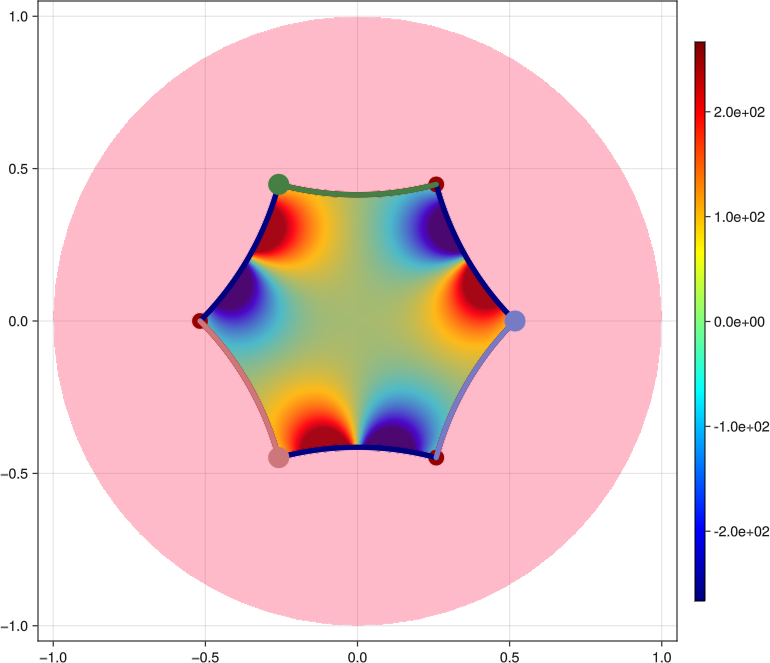}
            \end{subfigure}
        \end{tabular}
        &
        \begin{tabular}{c}
            \hspace{-1cm}
            \begin{subfigure}[t]{0.22\textwidth}
                \centering
                \includegraphics[width=\linewidth]{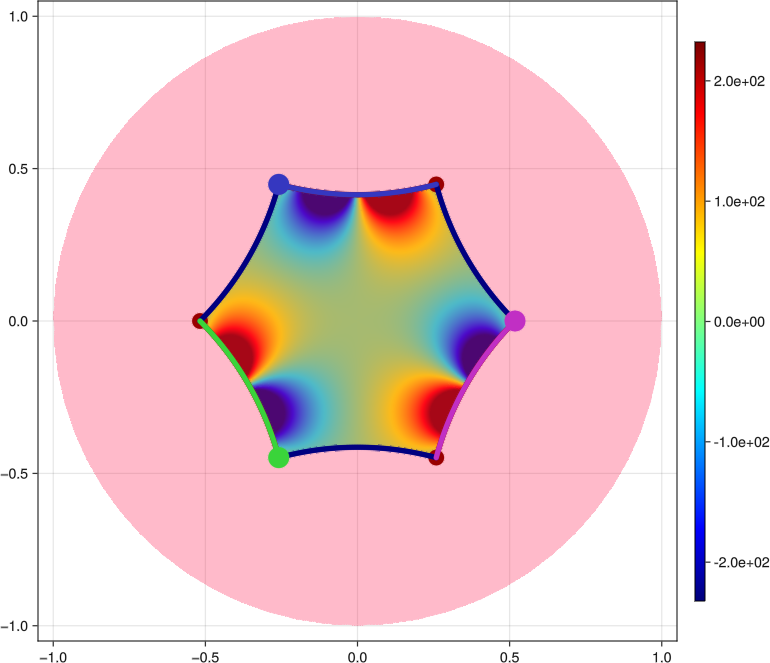}
            \end{subfigure}\\
            \hspace{-1cm}
            \begin{subfigure}[t]{0.22\textwidth}
                \centering
                \includegraphics[width=\linewidth]{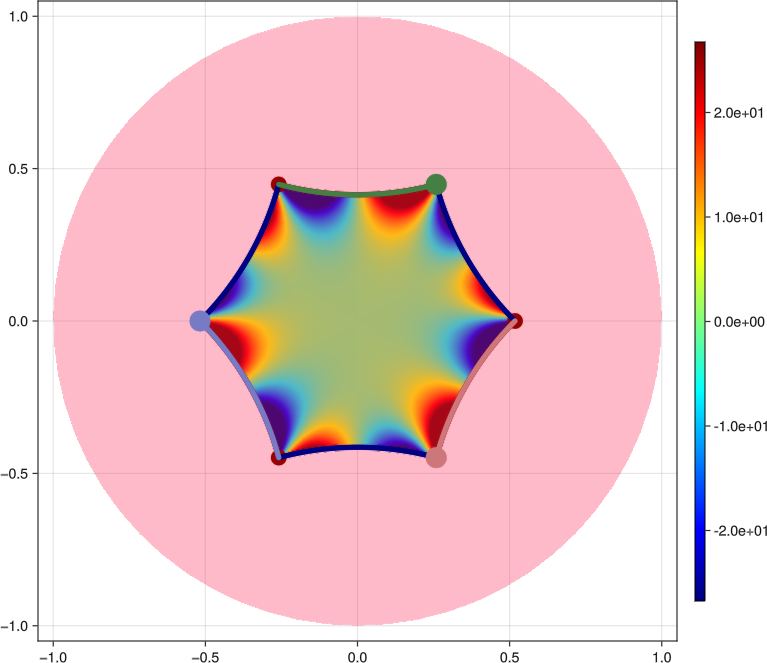}
            \end{subfigure}
        \end{tabular}
        &
        \begin{tabular}{c}
            \hspace{-0.75cm}
            \begin{subfigure}[t]{0.45\textwidth}
                \centering
                \includegraphics[width=\linewidth]{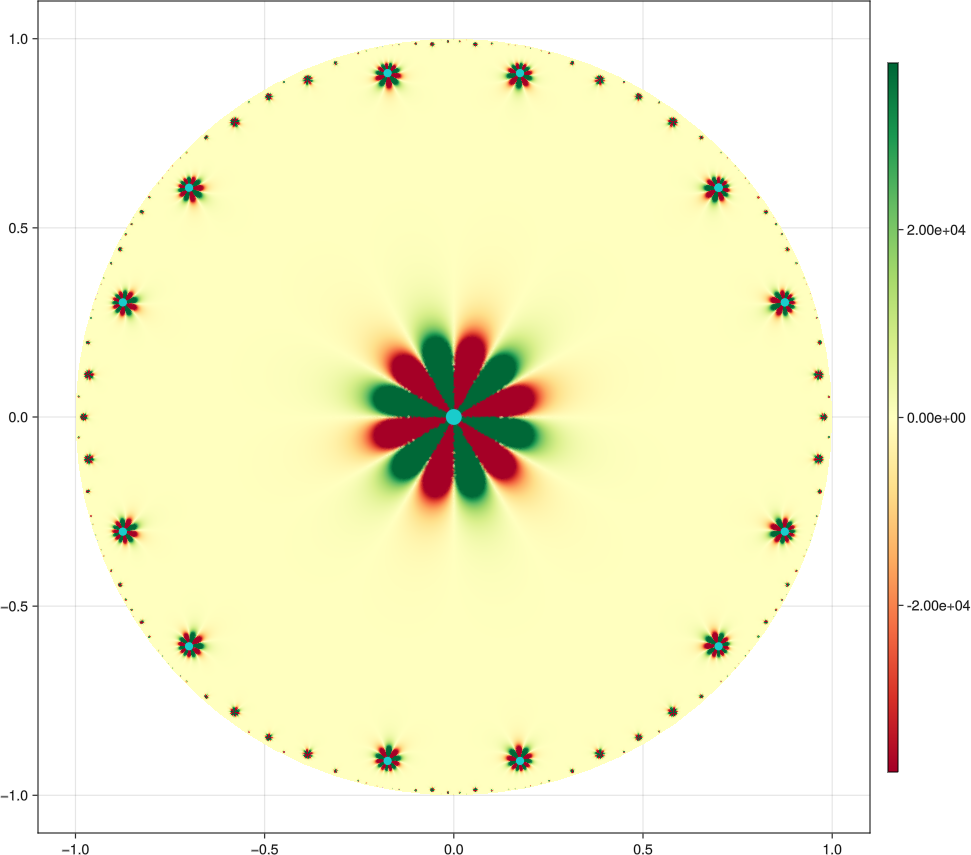}
            \end{subfigure}
        \end{tabular}
    \end{tabular}
    \caption{Pole of order $6$ on $D_6 \times \mathbb{Z}_2$ surface. The Imaginary part.}
    \label{fig_Weierstrass6D6I}
\end{figure}

\newpage 
\

\newpage 

\bibliographystyle{plain}
\bibliography{biblio}

\end{document}